\documentclass[a4paper]{amsart}
\usepackage{amssymb}
\usepackage[polutonikogreek, english]{babel}
\usepackage{graphicx}
\usepackage{pstricks}
\usepackage{enumerate}
\theoremstyle{plain}
\newtheorem{theorem}{Theorem}
\newtheorem{corollary}[theorem]{Corollary}
\newtheorem{lemma}[theorem]{Lemma}
\newtheorem{axiom}{Axiom}
\newtheorem{definition}{Definition}
\author{Wim Veldman}
\address{Institute for Mathematics, Astrophysics and Particle Physics, Faculty of Science, Radboud University Nijmegen,
Postbus 9010, 6500 GL Nijmegen, the Netherlands}
\email{W.Veldman@science.ru.nl}

\begin{document}
\title{Intuitionism: an inspiration?}
\maketitle

\begin{center}
\selectlanguage{polutonikogreek}
toig`ar >eg'w toi, xe~ine, m'al> >atrek'ews >agore'usw

\selectlanguage{english}
\smallskip
Well then, let me you, stranger,   precisely explain these matters\\\hfill \textit{Od.} $\alpha$ 214 \end{center} 

\section{Introduction}\label{S:intr} \subsection{Topology and foundations}We want to introduce the reader to the \textit{intuitionistic} view of  mathematics proposed, developed and defended by the Dutch mathematician L.E.J. Brouwer (1881-1966)\footnote{The title of the paper is taken from an exclamation by the philosopher L.~Wittgenstein, see \cite{veldman20}.}. 

Brouwer obtained his fame as a mathematician by a number of very important results in topology such as the \textit{Dimension Theorem}: \begin{quote} For all positive integers $m,n$,\\ if there is a homeomorphism from $[0,1]^m$ to $[0,1]^n$, then $m=n$, \end{quote} and the closely related \textit{Brouwer Fixed-point Theorem}: \begin{quote} For each positive integer $n$, for every continuous function $f$ from $[0,1]^n$ to $[0,1]^n$, there exists $p$ in $[0,1]^n$ such that $f(p)=p$. \end{quote}  

These results were obtained in 1911. 

 Brouwer had started thinking on the foundations of mathematics earlier than that. In his dissertation from 1907, see \cite{brouwer07}, he wrestled already with the concept of the continuum. In a famous paper from 1908, see \cite{brouwer08}, he attacked the principle of the excluded third $X\;\vee\;\neg X$. 

\subsection{Against the formalists}In his inaugural lecture at the University of Amsterdam in 1912, entitled: `{\it Intuitionism and Formalism}', see \cite{brouwer12}, Brouwer declares himself an opponent of  the {\it formalists}. The formalists, in Brouwer's words, explain the exactness and precision of  the statements of mathematics    by describing mathematics as a game with meaningless strings of symbols  according to very strict and precise rules. The game may turn out to be useful for certain purposes, and then may give `{\it a vague sensation of delight}', but these other purposes are not the concern of the mathematician. In the formalist's `game', there is no place for \textit{truth}, only for \textit{correctness}.  
 
 Brouwer most emphatically did \textit{not} want to explain away the meaningfulness of mathematical statements and the experience of truth. In his {\it intuitionistic} view,  the source of mathematical exactness is not to be found `{\it on paper}' but `{\it in the human intellect}'. 
 
\subsection{I. Kant} Brouwer was inspired by the philosopher I. Kant (1724-1804) who described the simple theorem `7+5=12' as the result of a construction in `\textit{pure intuition}'. 
 
 Pure intuition has to be distinguished from observation by the senses. In Kant's terms, it is the `form' of  all observations by the senses, and, as such, independent of all individual observations by the senses, i.e. {\it a priori}.  
 
 It is  very important that {\it one has to do something} in order to come to the insight of the truth of `7+5=12': counting 1, 2, 3, $\dots, 7$ and then, continuing, putting 1 below 8, 2 below 9, 3 below 10, 4 below 11, and, finally, 5 below 12. This is why Kant holds the mathematical judgment `7+5=12'  is not {\it analytic}, i.e. a matter of logic, but {\it synthetic}, see \cite[\S 2c]{kant}, i.e., going beyond logic, and, one might perhaps say,  a product of some activity by the judging subject.  
 
 \subsection{Languageless constructions} For the intuitionistic mathematician, every mathematical theorem  is, like `7+5=12', the result of  a {\it construction in pure intuition}. The successful completion of this construction, {\it the proof of the theorem}, gives joy and a sense of beauty, and one wants to tell one's friends.
 
 A construction in pure intuition is \textit{languageless}. Language comes in if I want to preserve the memory of the construction for myself, or if I want to make an attempt to explain to my fellow mathematician what I did, in the hope she will be able to do something similar. Language is not trustworty: it always may fail to do what one expects it to do, as we all know from our experience as unsuccessful teachers and  forgetful researchers. 
 
 There is no such thing as an exact linguistic description of a mathematical construction. The linguistic description only is the \textit{accompaniment} of an act of understanding, either between me and one of my former selves, or between me and a fellow mathematician.
 
 Nevertheless, it is our task to study the language of mathematics carefully and to find out where it perhaps may be improved. Although it is impossible we express ourselves in a perfect way, we always have to try to do better than we did until now.
 
 \subsection{Intuitionistic mathematics is constructive mathematics}Intuitionistic mathematics is \textit{not} a new, alternative kind of mathematics. Mathematics \textit{is} intuitionistic mathematics. Brouwer's \textit{revolution} is a \textit{call for reflection}. A sharpened awareness of what it is we do when doing  
  mathematics makes us more careful and more precise and, thereby, hopefully, better mathematicians.  
  
  Intuitionistic mathematics is  {\it constructive} mathematics. The reason is that mathematics itself essentially is constructive although this fact has been obscured by our mistaken trust in (classical) logic. 
  
  Thinking of Brouwer's ability in topology one may reflect that a continuous function $f$ from $[0,1]^m$ to $[0,1]^n$ has the property that, for each  $p$ in $[0,1]^m$, one may effectively find approximations of the function value $f(p)$ from approximations of the argument $p$.  \textit{Continuous} functions are \textit{constructive} functions:  there is a deep connection between Brouwer's  foundational interests and his topological concerns.
  
    \bigskip
\subsection{The contents of the paper}
The paper is divided into 21 Sections. Section 2 treats the \textit{basic intuition} giving rise to the nonnegative integers. Section 3 is a short reflection on theorems and their proofs.   In Sections 4-6, we consider three famous results, that usually are taken to be negative statements although they are not negative at all. In Section 4, we have a look at Euclid's theorem that there are infinitely many primes. In Section 5, we consider the theorem that $\sqrt 2$ is irrational. In Section 6, we introduce the reals and Cantor's theorem that there are uncountably many of them.  In Section 7, we introduce  \textit{reals oscillating around $0$}: if $x$ is such a number, one is unable to make out if $x<0$, $x=0$ or $x>0$.  Section 8 explains the reasons for interpreting the logical constants constructively and  rejecting the principle of the excluded third $X\vee \neg X$.  In Section 9 we study a basic result of real analysis: the {\it Intermediate Value Theorem}. The theorem is an existential statement that fails when taken at its constructive face value, but, if suitably reformulated and adapted, leads to several true and useful results. Finding such reconstructions of a  constructively invalid result is part of the \textit{intuitionistic program}.

 Intuitionistic mathematicians hold that every real function is (pointwise) continuous. 
 In Section 10, we treat the preliminary observation that a real function can not have a positive discontinuity. In Section 11, we go into the intuitionistic meaning of {\it negation}.  In Section 12, we introduce the  {\it Continuity Principle}, an \textit{axiom} used by Brouwer and we show that this axiom leads to the positive result that real functions are (pointwise) continuous. 

In Section 13, we introduce the {\it Fan Theorem}, a theorem that perhaps, like the Continuity Principle, should be called an  axiom:  we discover Brouwer's  `proof': most of us would call it a philosophical argument. One should note, however,  that, as long as we do not organize our mathematical results in a formal way, it is difficult to distinguish axioms from theorems. As we shall see, the Fan Theorem implies that a real function from $[0,1]$ to $\mathcal{R}$ is not only pointwise but even  \textit{uniformly} continuous. As an aside, we prove that the Fan Theorem fails to be  true in \textit{computable analysis}. In Section 14, we go into the r\^ole of the Fan Theorem in game theory. 

In Section 15, we consider {\it Brouwer's Thesis on Bars in  $\mathcal{N}=\omega^\omega$}, formulated as a {\it Principle of Bar Induction}. The Thesis on Bars in $\mathcal{N}$ implies the Fan Theorem but is a far stronger statement. We show how the argument for the Fan Theorem extends to the Thesis on Bars in $\mathcal{N}$.  As an application we prove, in Section 16,  the {\it Principle of Open Induction in $[0,1]$}. In Section 17 we give a second formulation of Brouwer's Thesis. This second formulation uses \textit{stumps}: inductively generated sets of finite sequences of natural numbers, whose r\^ole may be compared to the r\^ole of countable ordinals in classical, non-intuitionistic analysis. As an application we give, in Section 18,  an intuitionistic formulation and proof of the {\it Clopen Ramsey Theorem}, itself an impressive extension of Ramsey's result from 1928.  The proof of this intuitionistic theorem has not been published before. In Section 19 we study Borel sets. Because of the failure of De Morgan's Law $\neg\forall n \neg P(n)]\rightarrow \exists n[ P(n)]$, the subject of descriptive set theory needs a complete reconstruction. In Section 20 we reflect on the notion of a finite subset of the set $\mathbb{N}$ of the non-negative integers. This mathematical notion may be made intuitionistically precise in uncountably many ways. It is a perfect illustration of the subtlety and expressivity the language of mathematics has obtained from Brouwer's intervention.   

In  Section 21, we briefly describe how E. Bishop, who founded his own school of {\it constructive analysis} in the 1960's, treated Brouwer's legacy, and how also  P. Martin-L\"of's view on constructive mathematics was influenced by Brouwer.  

Many Sections may be read independently from other Sections.  
\section{The basic intuition}\label{S:intuition} \subsection{If you don't become like little children $\ldots$}\footnote{Matthew 18:3}$\:$

Mathematics is a child's game and
each of us, from his tender days, is familiar with the {\it infinite} sequence: $$0,\;1, \;2, \;3,\dots$$

The sequence of these \textit{natural} numbers or non-negative integers {\it has no end}.

No end? Can we always go on? Yes, yes, we   can always go on!

We never die, or get tired of the game. Please, do not plague us   with such silly objections.

We courageously stick to the fascinating perspective of the never finished infinite. 

\subsection{Induction and recursion}
The function $\mathsf{S}$  produces, given any non-negative integer $n$, the next one, $\mathsf{S}(n)$: $$\mathsf{S}(0)=1, \;\mathsf{S}(1) =2, \;\mathsf{S}(2)=3, \ldots$$
Let $P$ be a property of non-negative  integers such that one has a proof of $P(0)$, and, for each $n$, a proof of $P\bigl(\mathsf{S}(n)\bigr)$ from $P(n)$, i.e.    a proof of $P(1)$ from $P(0)$, a proof of $P(2)$ from $P(1)$, a proof of $P(3)$ from $P(2)$, and so on.  One then proves, step by step, first $P(1)$, then $P(2)$, then $P(3)$, and so on,  that is, one  proves, for each $n$, $P(n)$. We see this happening although the process is never finished.

The principle of \textit{complete induction on the non-negative integers} thus is obvious.

Suppose one defines: for each $ m$, $m+0=m$ and for each $n$, $m+\mathsf{S}(n)=\mathsf{S}(m+n)$. One  then      calculates, first   $m+1= m+\mathsf{S}(0)$, then  $m+2= m+\mathsf{S}(1)$, then  \\$m+3=m+\mathsf{S}(2)$, and so on, and, in this way,   one finds,  for each $n$, the meaning  of $m+n$.

This also is obvious. 

\subsection{The `logicist' attempt to `prove' induction}

In the nineteenth century,  G. Frege and R. Dedekind tried to  \textit{define} the non-negative integers, logically or  set-theoretically, and then to \textit{prove} the principle of complete induction and  the possibility of defining functions by recursion. 

 Brouwer considered such attempts futile and misleading. 
 
 \subsection{Awakening consciousness} Brouwer spent a lot of thought on finding the \textit{origin} of the idea of the non-negative integers, see for instance \cite{brouwer48}.

\smallskip
\textit{A first thing, then a next thing, then another thing, and again another thing  $\ldots$}

\smallskip

Mark the beginning: \textit{a first thing, then a next  thing}.  Brouwer called it the \textit{basic intuition of two-ity}. He saw it happening in an awakening of slumbering consciousness, a shift of attention, a moment's reflection, or a {\it move of time}: here I am, and now I see: `here I am', and that is two things.  Iterating this shift of attention, I obtain three things, four things, and so on. 

One is reminded of R. Dedekind's proof of  \cite[Theorem 66]{dedekind}: \textit{There exist infinite sets}, for instance, the `\textit{world of my thoughts, meine Gedankenwelt}'. Dedekind invokes a function $s\mapsto s'$ that associates to every thought $s$ of mine the thought $s'$: \textit{`I am thinking of $s$}' and concludes, from the fact that this function should be one-to-one and non-surjective, the latter because `\textit{I myself, mein eigenes Ich}' am a thought of mine not of the form $s'$, that I have infinitely many thoughts. 

\smallskip As we mentioned already, Brouwer refers to Kant. Kant held that the concepts of `time' and `space', the `forms' of `{\it inner sense}' and `{\it outer sense}', respectively,  are \textit{apriori} concepts, that is, independent from every observation by the senses. These apriori concepts are  studied by mathematics, in its two divisions: `arithmetic' and `geometry'. Brouwer wants to maintain Kant's conviction about the apriori character of `time' but rejects the apriori character of `space'. This is partly because Kant believed space to be Euclidean, and, after Kant, one had the non-euclidean revolution. Also, geometry may be `arithmetized' and be robbed of its status as an independent discipline. 

\subsection{The `set' $\mathbb{N}$} We use $\mathbb{N}$ to denote the {\it set}
 of the nonnegative integers. This \textit{`set'} is a well-understood totality but an ongoing project, in no sense finished or complete. We use $m,n,  \ldots, s, t, \ldots$ as variables over $\mathbb{N}$.

\section{Theorems and their proofs}  \subsection{Mere announcements} One does  not understand a  mathematical theorem  immediately upon hearing its statement.  The statement of the theorem is only a preliminary announcement, a partial communication,  a  promise. Its full meaning  will unfold itself  in the {\it proof}, and understanding this full meaning   is reserved for someone who takes the trouble to go through the proof and to reconstruct it for herself.  

\subsection{No to `Platonism'}\label{SS:notoplato} The statement of a theorem  tells us that   some  mental construction has been succesfully completed. It is not reporting a fact one might `observe'  in a mathematical reality lying outside us. \\The latter view has been strongly defended by Plato (427-347 B.C.). He  believed that the contemplation  of the timeless mathematical reality, that one approaches with the mind and not with the senses and that one should not touch with one's fingers\footnote{See Plato, {\it Politeia (the Republic) VII} 527a.},   would prepare us for the contemplation of what is the true interest of the philosopher: the timeless reality of the `{\it Ideas}'.  

 The intuitionistic mathematician rejects this view. In the ancient debate between Speusippus (405-338 B.C.) and Menaechmus (380-320 B.C.) as reported by Proclus (410-485), she would definitely have sided with Menaechmus. Speusippus held that every `problem' essentially is a `theorem', i.e. a truth to behold, whereas Menaechmus defended the view that every `theorem'  essentially is a `problem', i.e. a construction to find and carry out.\footnote{The 48 propositions of the First Book of Euclid's {\it Elements} (300 B.C.) are divided into 34 `theorems' and 14 `problems'.}  Proclus, in his wisdom,  judged that  both views contain some truth\footnote{`Both parties are right', see \cite[Prologue, Part Two, $\S$ 78]{proclus}}, and, indeed, one might say, both views try to explain a part of our mathematical experience. \\ In our days, many mathematicians  still feel they are discovering \textit{hard facts} not of their own making. They  prefer a so-called {\it realistic} view, reminiscent of Plato's, while admitting it is difficult to give a convincing account of the {\it mathematical reality} see \cite[p. 123]{hardy}.

As to Brouwer's {\it formalist} adversaries, some of them develop their formalisms with an eye to `reality' while others,  like Brouwer himself, keep far from it.  A famous defender of the realistic point of view is K. G\"odel (1906-1972). The formula $\mathsf{CH}$ in the language of set theory that is taken to represent Cantor's famous  continuum hypothesis\footnote{In an intuitionistic context, the classical continuum hypothesis has no immediate meaning. One may however study meaningful statements that seem to come close tot it, see \cite{gielends} and \cite[Sections 7 and 8]{veldman20b}.} has been shown, by G\"odel himself in 1938, and by P. Cohen in 1963, respectively,  not to be refuted  and not to be proven by the rules of first-order predicate logic from the formulas accepted as representing `the' axioms of set theory, provided the latter do not give a contradiction. G\"odel observed, even before Cohen proved his result,  that  this fact does not solve the problem: we still have to find out if $\mathsf{CH}$  holds in the mathematical universe, see \cite{goedel47}\footnote{“Of course it is happening inside your head, Harry, but why on earth should that mean that it is not real?” Dumbledore in: 
 J.K. Rowling, Harry Potter and the Deathly Hallows }.

The intuitionistic mathematician  discards the realistic view:  he believes it is  wrong and  harmful for the practice of mathematics. 
\subsection{Propositions (as yet) without proof} From our intuitionistic perspective, it  seems impossible to fully understand a proposition that has (as yet) no proof. We like to say, however,  that we also `\textit{understand}' the proposition `$0=1$', in the sense that we see this proposition never will have a proof. 

Besides, propositions that are only partially understood in the sense that we have as yet no proof and also do not see that we never will have one,   have an important part to play in mathematics.  Let us consider an example.  

Studying Euclid's proof that there are infinitely many primes\footnote{See   Section \ref{S:infprime}.}, one comes to  {\it understand} the statement: `\textit{there are infinitely many primes}'  and to enjoy its beauty and depth. One now may ask: `{\it Are there also infinitely many twin primes, i.e. prime numbers $p$ such that also $p+2$ is prime?}'. {\it By analogy}, we  have \textit{some} understanding of the latter proposition. We  can {\it imagine what a proof of this proposition might be} and, in this sense, catch its meaning, although, up to now, no proof or refutation has been found. 

By his {\it exemplary} proof, Euclid suggests us how to (partially) understand other propositions of the form `{\it there are  infinitely many numbers $n$ such that $P(n)$}'.

\subsection{Negative theorems} There  is a large supply of theorems whose statement  is  negative. Such theorems  report us that a construction of a kind we hoped for is impossible. We obtain such conclusions {\it imagining} ourselves  possessed of  a construction of the required kind and  seeing  that we then could prove $0=1$. We hastily retreat, admitting we were  imagining the impossible.

Indeed, when constructing proofs, we often use our {\it imagination}. We  make temporary extra assumptions {\it for the sake of the argument}, that later perhaps will have to be given up.  G. Gentzen (1909-1945), the inventor of `\textit{Natural Deduction}', put great emphasis on this fact.  Gentzen was inspired by Brouwer, see \cite[Section II, $\S 3$]{gentzen}. 

In a sense, propositions with a negative conclusion are \textit{second-rank}. They tell us about a lost hope, a failure. Fortunately, many mathematical theorems are {\it unnecessarily negative}. Their negative formulation is caused by  laziness or lack of attention. The intuitionistic mathematician likes to find the positive content of seemingly negative statements.

In Sections \ref{S:infprime}, \ref{S:irrat} and \ref{S:uncount}, trying to make this clear, we give  some examples.

\section{Infinitely many prime numbers}\label{S:infprime}

The following theorem is due to Euclid, see  \cite[Book IX, Proposition 20]{euclid}. It is important and rewarding to study the proof in the original. 

\begin{theorem}\label{T:euclid} The prime numbers are more than every quantity of prime numbers  given beforehand.

\end{theorem}
\begin{proof} Let the prime numbers  given beforehand be $A, B, C$. I say that there are more prime numbers than $A, B, C$. Let the least number measured by\footnote{the least common multiple of} $A,B, C$ be taken, and let it be $DE$, and let the unit $DF$ be added to it; we thus obtain $EF$.  $EF$ is prime or $EF$ is not prime\footnote{Euclid makes this case distinction as he did not consider a number to be a divisor of itself.}. First, let it be prime. Then the prime numbers $A, B, C, EF$ are more than $A, B, C$. Next, let $EF$ be not prime. Then there is a prime number that measures $EF$. Let it be $G$. I say $G$ is not the same with any of $A, B, C$. For, if possible, let it be so. Now $A, B, C$ measure $DE$ and thus also $G$ measures $DE$. But the number $G$ also measures $EF$ and thus the \emph{number} $G$   measures the  remainder $DF$ which is a \emph{unit}, and that is absurd. So $G$ is not the same with any of $A, B, C$. Thus prime numbers $A,B,C,G$ are found which are more than than the quantity of prime numbers $A,B,C$ given beforehand. \end{proof}

We now paraphrase Euclid's proof in contemporary terms. Euclid  explains us, how, given any finite list $q_0, q_1, \ldots, q_{n-1}$ of prime numbers, one  produces a prime number $q$ that is not on the list: 
consider $c:=$ the least common multiple of $q_0,q_1, \ldots, q_{n-1}$. If $c$ is prime,  take $q := c$. If $c$ is not prime,  determine the least  number $q>1$ that divides $c$. Then $q$  is prime, and, for every $i<n$, $q \neq q_i$.

One may seriously wonder if this paraphrase is better than the original. 

Note that Euclid  only treats the case: $n=3$. He of course expects the reader to pick up the general idea from this exemplary case. {\it The reader must be willing to understand him}. This is a  {\it conditio sine qua non} for every written proof.

 Also note that Euclid provides  an  \textit{algorithm}: he demonstrates how, given any finite list of prime numbers, one \textit{calculates} a prime number that still is missing.

Most importantly, as we observed already in the previous Section, Euclid  in a way lays down the meaning of the word \textit{`infinite'}. His  proof suggests   how to \textit{understand} a proposition of the form: `\textit{there are infinitely  many numbers satisfying  $A$}'.  One may prove such a thing by doing something like he does in the case of the prime numbers.

\medskip
Sometimes, a  description of Euclid's proof is given that does no justice to the argument.   According to this description, one is  invited  to \textit{assume}: there is a finite list containing all the primes.  Euclid   then \textit{reduces} this assumption   \textit{to absurdity} and  forces one to admit that the initial assumption is false: there is \textit{no} such finite list, \textit{which is to say:} there are \textit{not} finitely many primes. 

G.H.~Hardy, a great mathematician and a convinced platonist,  has this view, see \cite[Section 12]{hardy}. In the context of Euclid's theorem, he calls \textit{reductio ad absurdum} `\textit{one of a mathematician's finest weapons}', but adds a footnote stating: \textit{`The proof can be arranged so as to avoid a \emph{reductio} and logicians of some schools would prefer that it should be.'}

\section{$\sqrt 2$ is irrational}\label{S:irrat} 
 The theorem that $\sqrt{2}$ is irrational is also  due to the Greeks and dates from 450 B.C.. In those days, mathematicians sometimes were called \textit{`those who know about the odd and the even'}. The theorem uses the following
 
 \begin{lemma}\label{L:odd} For all positive integers $n$, if $n$ is odd, then $n^2$ is odd. \end{lemma}
 \begin{proof} Find $p$ such that $n = 2p +1$. Then $n^2 = (2p+1)(2p+1)=4p^2 + 4p +1= 2(2p^2 +2p) +1$. \end{proof}
 
 \begin{theorem}\label{T:sqrt2irr} There are no positive integers $m,n$  such that $\sqrt{2}=\frac{m}{n}$. \end{theorem}
 
 \begin{proof} Suppose there are. Find such $m,n$ and assume: at least one of $m,n$ is odd. Then $\sqrt{2} =\frac{m}{n}$, so $2n^2 = m^2$, so $m^2$ is even and, by Lemma \ref{L:odd}, also $m$ is even and $4$ divides $m^2=2n^2$, so $n^2$ is even and, again by Lemma \ref{L:odd}, $n$ is even. We thus see: both $m,n$ are even. Contradiction. \end{proof}
 
 Already in ancient times, some people worried about this proof, in particular about the sentence: `\textit{Suppose there are.}'.  How is it possible to \textit{suppose} that $\sqrt{2}$ is rational? `\textit{Suppose}' must mean something like `\textit{Imagine}'. How can one make a clear mental picture of a fact that will turn out to be false? How to imagine the impossible?

 In particular, Parmenides (fl. 475 B.C.) and Zeno (fl. 450 B.C.), leading members  of the {\it Eleatic school of philosophy},  felt this difficulty. The Dutch mathematician G.F.C. Griss (1898-1953) shared their concern and suggested to Brouwer to remove the proposition without proof, negation and the empty set from the discourse of mathematics. He started  a redevelopment of intuitionistic mathematics, see, for instance, \cite{griss}.  Brouwer did not want to go into his suggestion, and offered the somewhat surprising argument that, following it,  one would impoverish the subject, see \cite{brouwer48a}.  Nevertheless, in intuitionistic, and, more generally, in constructive mathematics, one avoids negation as much as possible, also for the reason that negative results lack {\it constructive content}.

The negatively formulated Theorem \ref{T:sqrt2irr} may  be replaced  by a  positive and {\it affirmative} result, as follows.
 
\begin{definition} A real number $x$ is  \emph{positively irrational} if and only if, given any rational number $q = \frac{m}{n}$ one may effectively find a positive integer $p$ such that $|x - \frac{m}{n}|\ge \frac{1}{p}$.\end{definition}
 
 \begin{theorem} $\sqrt{2}$ is positively irrational. \end{theorem}
 
 \begin{proof} Let $m,n$ be given. Note $2n^2 \neq m^2$, and therefore: $|2n^2 - m^2| \ge 1$ and $|2 - \frac{m^2}{n^2}|\ge \frac{1}{n^2}$, that is $|\sqrt{2}+\frac{m}{n}|\cdot|\sqrt{2}-\frac{m}{n}| \ge \frac{1}{n^2}$. If $\frac{m}{n} > 2$, then $\frac{m}{n} - \sqrt{2}>\frac{1}{2}$ and, if $\frac{m}{n} \le 2$, then  $|\sqrt{2}+\frac{m}{n}|< 4$ and $ |\sqrt{2}-\frac{m}{n}|> \frac{1}{4n^2}$. \end{proof}

 \section{Uncountably many reals}\label{S:uncount}

  We  sketch how to treat {\it integers} and {\it rationals} and then define the notion of a {\it real number}, essentially like G. Cantor (1845-1918) did it, and  Brouwer after him. 
  
  \subsection{Integers and rationals} One first introduces, on the set $\mathbb{N}$ of the natural numbers, the operations $+, \cdot, exp$ and the relations $<,\le$   by {\it recursive} definitions, as suggested in Section \ref{S:intuition}, and proves, by {\it induction},  their well-known nice properties. One then defines a {\it pairing function}  from $\mathbb{N}\times\mathbb{N}$ to $\mathbb{N}$, for instance the function $(m,  n) = 2^m(2n+1) -1$, with inverse functions $K,L:\mathbb{N}\rightarrow\mathbb{N}$. We will write $m'=K(m)$ and $m''=L(m)$, so, for each $m$, $m=(m', m'')$. 
  
  Then an equivalence relation $=_\mathbb{Z}$ on $\mathbb{N}$ is introduced by: \\$m=_\mathbb{Z} n \leftrightarrow m'+n''=m''+n'$. Integers, for us, are just natural numbers and not, as in the usual treatment,  the equivalence classes of the relation $=_\mathbb{Z}$.  Functions and relations on $\mathbb{Z}$ are  functions and relations on $\mathbb{N}$ that respect the relation $=_\mathbb{Z}$.  The functions $+_\mathbb{Z}, \cdot_\mathbb{Z}$, the relations $<_\mathbb{Z}, \le_\mathbb{Z}$,and the special elements $0_\mathbb{Z}, 1_\mathbb{Z}$ are defined in the way one would expect.
  
  One then defines $\mathbb{Q}:=\{m\mid \neg (m''=_\mathbb{Z} 0_\mathbb{Z})\}$. An equivalence relation $=_\mathbb{Q}$ is introduced on $\mathbb{Q}$ by: $m=_\mathbb{Q} n \leftrightarrow m'\cdot_\mathbb{Z} n'' =_\mathbb{Z} m''\cdot_\mathbb{Z} n'$. Rationals, for us, are just natural numbers, but, whenever we consider natural numbers as rationals, we respect $=_\mathbb{Q}$.  The functions $+_\mathbb{Q}, \cdot_\mathbb{Q}$, the relations $<_\mathbb{Q}, \le_\mathbb{Q}$, and the special elements $0_\mathbb{Q}, 1_\mathbb{Q}$ are defined straightforwardly. 
  
  In general, we omit the subscripts `$_\mathbb{Z}$', `$_\mathbb{Q}$' unless we fear for confusion,
  
  It is important to realize that all functions and relations mentioned so far are algorithmic.  We learned the algorithms at elementary school and know, for instance, how to find out, given rationals $p,q$, if $p=_\mathbb{Q} q$ or not.
  
 \subsection{Real numbers} We now are ready to introduce the reals.
  
 \begin{definition} A \textit{real number} is an \emph{ approximation process}, an infinite sequence $x=
  x(0), x(1), \ldots$ of pairs $x(n)=\bigl(x'(n), x''(n)\bigr)$ of rationals  such that 
       \begin{enumerate}[\upshape (i)] 
       \item \emph{$x$ is shrinking:} 
       for all $n$, $x'(n)\le x'(n+1)\le x''(n+1)\le x''(n)$  and 
       
       \item \emph{$x$ is dwindling:} for all $m$, there exists $n$ such that $x''(n) - x'(n) \le \frac{1}{2^m}$. \end{enumerate}
       
       $\mathcal{R}$ is the set of all real numbers. We use $x, y, \ldots$ as variables over $\mathcal{R}$. 
 \end{definition}      
       
Brouwer gave Cantor's  notion his own colouring. \\Every real number $x= x(0), x(1), \ldots$  is an \textit{infinite sequence of successive approximations}, and, like the infinite sequence $0,1,2, \ldots$ itself,   never finished and always under construction.  Even if an algorithm has been given determing the successive values of $x$, the finding of these successive values keeps us busy forever.   The {\it number} is the process of approximation itself: it makes no sense to ask for {\it the limit point} that the successive approximations are striving for.      
       
 For all reals $x$, for all rationals $p,q$, for all $n$, if $p<_\mathbb{Q} x'(n)<_\mathbb{Q} x''(n)<_\mathbb{Q} q$, we say that $(p,q)$ is an \textit{approximation of $x$ of precision $q-_\mathbb{Q} p$.}

       \begin{definition}  
  For all real numbers $x, y$, we define: 
  
  $x <_\mathcal{R} y$ ($x$ \emph{is smaller than $y$}) if and only if $\exists n[x''(n) <_\mathbb{Q} y'(n) ]$, and: 
  
  $x\le_\mathcal{R} y$ (\emph{$x$ is not greater than $y$}) if and only if $\forall n [x'(n)\le_\mathbb{Q} y''(n)]$, and:
   
   $x \;\#_\mathcal{R}\; y$ ($x$ \emph{is apart from } $y$) if and only if either $x <_\mathcal{R} y$ or $ y <_\mathcal{R} x$, and:

   $x =_\mathcal{R} y$ (\emph{$x$ coincides with $y$}) if and only if $x\le_\mathcal{R}y$ and $y\le_\mathcal{R}x$.
   \end{definition}
 We  suppress the subscript `$_\mathcal{R}$' if we expect that, doing so, we  will not confuse the reader. 
  
   Traditionally, our `{\it real numbers}' would be called `{\it number generators}'.\\The name `{\it real number}' would then be reserved for equivalence classes of the form \\$[x]=\{y\in\mathcal{R}\mid y =_\mathcal{R}x\}$. \\Also Brouwer does so, using the terms `{\it point}' and `{\it point kernel}'.
   
 We  will not make this distinction   but will only consider  operations and relations on $\mathcal{R}$ that respect the equivalence relation $=_\mathcal{R}$. 
   
   \smallskip The relations $<_\mathcal{R}$ and $\#_\mathcal{R}$ are \textit{positive} relations. From a constructive point of view, they are more important than the relations $\le_\mathcal{R}$ and $=_\mathcal{R}$. The latter two relations are {\it negative}, and, moreover, expressible in terms of the positive relations, as:  $x\le_\mathcal{R} y \leftrightarrow \neg (y<_\mathcal{R} x)$ and: $x=_\mathcal{R} y \leftrightarrow \neg (x \;\#_\mathcal{R}\; y)$. Given real numbers $x,y$ one may {\it prove}:  `$x=_\mathcal{R} y$' by assuming: `$x\;\#_\mathcal{R}\; y$' and obtaining a contradiction. 
   
   It is important that the positive relations $<_\mathcal{R}$ and $=_\mathcal{R}$ are {\it co-transitive}, i.e. \\$x<_\mathcal{R} y\rightarrow (x<_\mathcal{R} z \;\vee\;z<_\mathcal{R} y)$ and $x\;\#_\mathcal{R}\; y\rightarrow (x\;\#_\mathcal{R}\; z \;\vee\;z\;\#_\mathcal{R}\;y)$.
   
   It suffices to prove this for the first relation. \\Given $x,y,z$ such that $x<_\mathcal{R}y$, find $n$ such that $z''(n)-z'(n)<_\mathbb{Q} y'(n) -x''(n)$ and decide: either $x''(n)<_\mathbb{Q}z'(n)$ and $x<_\mathcal{R} z$ or: $z''(n)<_\mathbb{Q} y'(n)$ and $z<_\mathcal{R}y$.

  \subsection{Cantor's proof} In 1873, Cantor proves the following, see \cite[\S 2]{cantor}:

  \begin{theorem}\label{T:cantor} For every infinite sequence  $x_0, x_1, \dots$  of reals, there exists a real  $x$  apart from every real in the sequence, that is: $x \;\#_\mathcal{R}\; x_0, x \;\#_\mathcal{R}\; x_1, \ldots $ \end{theorem}
 \begin{proof} Let an infinite sequence $x_0, x_1, x_2, \ldots$ of reals be given. 
 
 Define   $x=x(0), x(1), \ldots$, as follows, step by step.  
 
 First, define $x(0)=(0,1)$.
 
 Now assume  $x(n)$ has been defined and $x''(n)-x'(n)=_\mathbb{Q}\frac{1}{3^n}$.   
 
 Find $m_0 :=$ the least $m$ such that $x_n''(m)-x_n'(m)<_\mathbb{Q}\frac{1}{3^{n+1}}$. 
 
 Note: \textit{either} $\frac{2}{3}x'(n) +_\mathbb{Q}\frac{1}{3}x''(n)<_\mathbb{Q}x_n'(m_0)$   \textit{or} $x_n''(m_0)<_\mathbb{Q}\frac{2}{3}x'(n) +_\mathbb{Q}\frac{1}{3}x''(n)$. 
 
   \textit{If} $\frac{2}{3}x'(n) +_\mathbb{Q}\frac{1}{3}x''(n)<_\mathbb{Q}x_n'(m_0)$,  define: $x(n+1) := \bigl(x'(n), \frac{2}{3}x'(n) +\frac{1}{3}x''(n)\bigr)$ and, \textit{if not}, define $x(n+1) := \bigl( \frac{1}{3}x'(n) +\frac{2}{3}x''(n), x''(n)\bigr)$.  
   
    By this choice of $x(n+1)$, we ensure $x''(n+1)-x'(n+1)=_\mathbb{Q}\frac{1}{3^{n+1}}$ and $x \;\#_\mathcal{R}\; x_n$. \end{proof}
  
  Our constructive argument differs but little from Cantor's.  Cantor makes it perfectly clear how,   given any countable list  $x_0, x_1, \ldots ,$ of   reals,  one produces a real  positively different from every real on the list.
  
  We thus see that the statement `{\it $\mathcal{R}$ is uncountable}', essentially, is a {\it positive}
  statement.
  
 Note that, by his {\it exemplary} proof,  Cantor suggest to  us how we should  understand and prove a  proposition of the form `\textit{$A\subseteq\mathcal{R}$ is uncountable}'.
 
 \smallskip One  defines operations $+_\mathcal{R}$ and $\cdot_\mathcal{R}$ on $\mathcal{R}$ by: \\for all $n$, $(x+_\mathcal{R} \;y)(n):=\bigl(x'(n) +_\mathbb{Q} y'(n), x''(n)+_\mathbb{Q}y''(n)\bigr)$ and \\$(x\cdot_\mathcal{R} y)(n)= \bigl(\min_\mathbb{Q}(x'(n)\cdot_\mathbb{Q} y'(n), x'(n)\cdot_\mathbb{Q} y''(n),x''(n)\cdot_\mathbb{Q} y'(n),x''(n)\cdot_\mathbb{Q} y''(n)),\\ \max_\mathbb{Q}(x'(n)\cdot_\mathbb{Q} y'(n), x'(n)\cdot_\mathbb{Q} y''(n),x''(n)\cdot_\mathbb{Q} y'(n),x''(n)\cdot_\mathbb{Q} y''(n))\bigr)$. 
 
 Also subtraction and an absolute value function $x\mapsto |x|$ may be defined. 
 
 For every rational $q$, we define $q_\mathcal{R}$ in $\mathcal{R}$ such that, for all $n$,  \\$q_\mathcal{R}(n)= (q-_\mathbb{Q}\frac{1}{2^n}, q+_\mathbb{Q}
 \frac{1}{2^n})$. 
 
 We omit subscripts `$_\mathcal{R}$' where we think it does no harm to do so.
  
  \section{Fugitive integers and  oscillating reals}\label{S:osc} The following definition lies at the basis of  many {\it counterexamples in Brouwer's  style} to  various mathematical results that do not stand a constructive reading.
  
  We consider the decimal expansion of $\pi$, evaluating it step by step, hunting for the first block of 99 consecutive 9's. We would like to define $k_{99}$ as the least number $n$  such that at the places $n, n+1, \ldots n+98$ in the decimal expansion of $\pi$ we find the value $9$, but we have to be careful as we do not know if such a number $n$ exists. 
  
   \begin{definition}[The fugitive `number' $k_{99}$]\label{D:k99} Let $d= d(0), d(1), \ldots$ be the decimal expansion of $\pi$, that is: $d$  is the function from $\mathbb{N}$ to $\{0,1, \ldots,9\}$ such that \\$\pi = 3+\sum_{n =0}^{\infty}d(n)\cdot 10^{-n-1}$.  
  
  For each $n$, we define:
  
  $ k_{99}\le n$ if and only if  $\exists j\le n \forall i < 99[d(j+i) = 9]$, and:
  
  $n<k_{99}$ if and only  if $\forall j\le n \exists i<99[d(i+j)\neq 9]$, and:
  
  $n=k_{99}$ if and only if $ k_{99}\le n$ and $\forall j<n[j<k_{99}]$.
  \end{definition}
 Note that we do not define a natural number $k_{99}$ but only the meaning of an expression like `$ k_{99}\le n$'.
 
 Also note that, for each $n$, one may find out, for each of the propositions `$k_{99}\le n$', `$n<k_{99}$', and `$n=k_{99}$', if they are true or not, by simply calculating the first $n$ values of the decimal expansion of $\pi$.
 
 The statement `$P:=\exists n[n=k_{99}]$' is a prime example of an {\it undecided proposition}, i.e. we do not have a proof of $P$, but we also do not have a proof that we  never will find one, and we have no procedure to find either one of these proofs in finitely many steps. 
 
  The problem of the 99 9's in the decimal expansion of $\pi$ is not important in itself. It only is a {\it pedagogical example}, showing that we have no method to solve {\it this kind of problems}. Should someone, by historical accident, find the 99 9`s, then one easily formulates a similar proposition that is still undecided.

  \begin{definition}[The `oscillating' reals  $\rho_0, \rho_1, \rho_2$]\label{D:osc}$\;$ \\
  We define real numbers $0_\mathcal{R}$, $\rho_0, \rho_1$ and $\rho_2$, as follows.
  
  $0_\mathcal{R}$ is the real number such that, for all $n$, $0_\mathcal{R}(n)=(-\frac{1}{2^n}, \frac{1}{2^n})$.
  
  \smallskip
   $\rho_0=(\frac{1}{2})^{k_{99}}$ is the real number such that, for all $n<k_{99}$,  $\rho_0(n) :=(-\frac{1}{2^n}, \frac{1}{2^n})$, and, for all $n \ge k_{99}$,  $\rho_0(n) :=(\frac{1}{2^{k_{99}}}, \frac{1}{2^{k_{99}}})$.
   
   \smallskip
  
  $\rho_1=(-\frac{1}{2})^{k_{99}}$ is the real number such that,   for all  $n< k_{99}$,   $\rho_1(n) =(-\frac{1}{2^n}, \frac{1}{2^n})$, and, for all $n \ge k_{99}$,  \textit{if $k_{99}$ is even}, then $\rho_1(n):=(\frac{1}{2^{k_{99}}}, \frac{1}{2^{k_{99}}})$, and, \textit{if $k_{99}$ is odd}, then $\rho_1(n) :=(-\frac{1}{2^{k_{99}}}, -\frac{1}{2^{k_{99}}})$. 
  
  \smallskip
   Finally,  $\rho_2:= \rho_0+_\mathcal{R}\rho_1$.\end{definition}
   
   Now note the following:
   
   \smallskip

 (1)  $0_\mathcal{R} =_\mathcal{R} \rho_0\leftrightarrow \forall n[n < k_{99}]$,  and we have no proof of `$\forall n[n < k_{99}]$'.

  Also: $\neg(0_\mathcal{R} =_\mathcal{R} \rho_0)\leftrightarrow \neg\forall n[n < k_{99}]$,  and we have no proof of `$\neg \forall n[n < k_{99}]$'. 
  
  Finally: $0_\mathcal{R} <_\mathcal{R} \rho_0\leftrightarrow \exists n[n=k_{99}]$,  and we have no proof of `$\exists n[n = k_{99}]$'.
  
  We thus see that, in general, given real numbers $x,y$, we have no means of proving one of the propositions `$x=_\mathcal{R} y$', `$\neg(x=_\mathcal{R} y)$' , `$x <_\mathcal{R}y$'.

  \smallskip (2) Note: $0_\mathcal{R} \le_\mathcal{R}\rho_0$, but, as we saw, we have no proof of `$0_\mathcal{R}=_\mathcal{R} \rho_0$', nor of `$0_\mathcal{R}<_\mathcal{R} \rho_0$'.
  
  We thus see that, in general, given real numbers $x,y$ such that $x\le_\mathcal{R}y$, we have no means of proving either one of the propositions `$x=_\mathcal{R} y$' or  `$x <_\mathcal{R} y$'.

  \smallskip

 (3)  Note: $0_\mathcal{R}\le_\mathcal{R} \rho_1 \leftrightarrow \forall n[n = k_{99}\rightarrow \exists m [n=2m]]$, and we have no proof of: `$\forall n[n = k_{99}\rightarrow \exists m [n=2m]]$'.
  
  Note: $\rho_1\le_\mathcal{R}0_\mathcal{R} \leftrightarrow \forall n[n = k_{99}\rightarrow \exists m [n=2m+1]]$, and we have no proof of: `$\forall n[n = k_{99}\rightarrow \exists m [n=2m+1]]$'.

We thus see that, in general, given real numbers $x,y$, we have no means of proving either one of the propositions `$x\le_\mathcal{R} y$', `$y\le_\mathcal{R} x$'.

  \smallskip

 (4) Note:  $\rho_2 =_\mathcal{R} 2\cdot \rho_0\leftrightarrow \forall n[n=k_{99}\rightarrow \exists m[n=2m]]$ and  \\ $\rho_2=_\mathcal{R} 0_\mathcal{R}\leftrightarrow\forall n[n=k_{99}\rightarrow \exists m[n=2m+1]]$. We have no means of proving either one of the propositions `$\rho_2=_\mathcal{R} 2\cdot \rho_0$', `$\rho_2=_\mathcal{R} 0_\mathcal{R}$'.
 \\ We will use use the number $\rho_2$ in Section \ref{S:ivt}.
  
  \medskip
   $\rho_0, \rho_1, \rho_2$ are examples of real numbers  \textit{oscillating around $0_\mathcal{R}$}.
   
   $\rho_0$ \textit{oscillates above $0_\mathcal{R}$} and $\rho_1$ \textit{oscillates up and down around $0_\mathcal{R}$}.
   
    $\rho_2=\rho_0+\rho_1$ \textit{oscillates between $0_\mathcal{R}$ and $2\cdot \rho_0$}.
  
 \section{Rejecting $P\;\vee\;\neg P$}\label{S:tndfalse}
 
 \subsection{Mathematics and Logic}
 
 Brouwer describes mathematics as being developed by the mind having come to awareness and   playfully exploring its possibilities, see, for instance, \cite{brouwer48}.  Mathematics is not dependent on  any evidence from the senses. All other science is applied mathematics, as mathematics recognizes and enforces patterns on the data acquired by observation.
 
  Logic in particular is not a \textit{foundational} science, prescribing the way one should think, but an \textit{observational} science.  Mathematics is not reigned by logic. Mathematics is a \textit{languageless activity of the mind}. As we observed already in Section \ref{S:intr}, only when we want to communicate about mathematics, language comes in. We  communicate with other people, trying to induce them to make mathematical constructions like the ones we ourselves are making, and we  communicate with ourselves, understanding the weakness of our memory, and hoping to be able  to remind ourselves tomorrow of what we did today.  
 Observing ourselves and others when we are busy communicating about mathematics, we take notes of the sounds and signs that are used. We  discover patterns and regularities, and then promote such regularities to \textit{laws of logic}. There is no guarantee, however, that someone, even myself,  who is making sounds or giving other signs in accordance with these laws is actually succeeding in making successful mathematical constructions. In this sense, the `\textit{laws of logic}' are \textit{unreliable}.

\medskip
Learning logic is part of learning the language of mathematics. Learning the meaning of the connectives `{$\ldots$\it or$\ldots$}', `{$\ldots$\it and $\ldots$}', `{$\ldots$ if $\ldots$ \it}', and  `{\it not :$\ldots$}', and the quantifiers `$\exists x \in V[\ldots x \ldots]$' and `$\forall x \in V[\dots x\ldots]$',  is like  learning the meaning of   the expressions  \textit{`infinite'} or \textit{`uncountable'} in the case of Theorems \ref{T:euclid} and \ref{T:cantor}. 

A \textit{connective} $\ast$ is a general method to obtain  a new proposition, $P\ast Q$, from any two given propositions, $P$ and $Q$. Understanding $P$ and $Q$ means  having some idea about what counts as a {\it proof} of $P,Q$, respectively. If we are able to point out in general terms  what we  should consider as a proof of $P\ast Q$, given our understanding of proofs of $P,Q$, we  will have explained  the meaning of the connective $\ast$. 

\subsection{Disjunction}\label{SS:disj} Let us start with the case of disjunction, $\vee$.

 If someone says to us: `\textit{I have a proof of $P \vee Q$, $P$ \textit{or} $Q$}', we \textit{expect}, from our experience with earlier situations in which the word \textit{`or'} occurred, that he will come up with a proof of  $P$ or with a proof of $Q$. We therefore adopt the rule that a proof of $P\vee Q$ must consist either in a proof of $P$ or in a proof of $Q$.

 One might  bring forward that disjunction is not always taken in this constructively strong sense.  There were other situations too! Some mathematicians might even say:  `\textit{When \emph{I} make  a statement $P\vee Q$, I merely want to remind my reader that $\neg P$ and $\neg Q$ are not both true}'. 
  
  The intuitionistic mathematician  then answers: `\textit{ Sometimes one uses  the stronger interpretation and at other times one uses the weaker one. Let us end this confusion and make the  first and stronger interpretation the {\em canonical} one and always watch what we want to say.}'

\medskip
Our decision implies that the  statement:

\begin{center}
$\exists n[ k_{99}\le n]\;\vee\;\forall n[n<k_{99}]$, i.e. 
  $\exists n[ k_{99}\le n ]\;\vee\;\neg \exists  n[ k_{99}\le n]$.\end{center}
is not a true statement.

Why?

   Until now, we did not find  a place in the decimal expansion of $\pi$ where there is an uninterrupted sequence of $99$ digits 9: we have no proof of `$\exists n[ k_{99}\le n]$'.
   
    Until now, we also did not find a clever argument showing that such an uninterrupted sequence of $99$ digits $9$ never will be found: we have no proof of `$\neg \exists n[ k_{99}\le n]$'. 
     
     Until now, therefore, we have no proof of `$\exists n[ k_{99}\le n]\;\vee\;\neg \exists  n[ k_{99}\le n]$'.
     
     \smallskip
    Note that there is an  asymmetry in the two horns of the dilemma. The (possible) truth of $\exists n[k_{99}\le n]$ will become clear to anyone who just patiently calculates the successive values of the decimal expansion of $\pi$, but the discovery of the (possible) truth of $\forall n[n<k_{99}]$  requires mathematical ingenuity.
 
 \smallskip
     
     The example shows that the principle \begin{center} {\it If} $\forall n[P(n)\;\vee\;\neg P(n)]$, {\it then} $\exists n[P(n)]\;\vee\;\neg \exists n[P(n)]$ \end{center} is unreliable.
     
     How did we come to trust it?
     
      We unthinkingly generalized our experience with handling \textit{finite} sets and \textit{bounded} quantifiers. After all, the following  holds, for any given $m$, even for $m=10^{10^{10}}$: \begin{center} {\it If} 
$\forall n\le m[P(n)\;\vee\;\neg P(n)]$, {\it then} $\exists n\le m[P(n)]\;\vee\;\neg \exists n\le m[P(n)]$, \end{center}
 because, at least in principle, we can check each of the numbers $0, 1, \ldots , m$ and find out if one of them has the property $P$.
 
 The  case $\{0,1, \ldots, m\}$ does not extend to the case $ \{0,1,2,\dots\}$. 
\\ Checking each of the infinitely many numbers $0,1,2,\ldots$ is \textit{not} feasible. 
 If we think we can do it we must imagine ourselves to be angels rather than human beings. 
 \subsection{`Reckless' or `hardy' statements} When in a playful mood, we call the statement: $$\exists n[k_{99}=n]\;\vee\;\forall n[n<k_{99}]$$ a \textit{reckless} or \textit{hardy} statement. Actually, it is foul play. We call the mathematician upholding the above statement as true `reckless' because we take him to read the statement constructively. Probably, he will protest.
 
 For us, the terms `reckless' and `hardy' indicate that the statement called so has no constructive proof, although the non-intuitionistic mathematician, using his own reading, sees no objection.
 
 Other examples of reckless statements  are: $$\neg\forall n[n<k_{99}]\;\vee\;\forall n[n<k_{99}]$$
 $$\exists n[n=k_{88}]\rightarrow \exists n[n=k_{99}]$$
 $$\forall n[2n\neq k_{99}]\;\vee\;\forall n[2n+1\neq k_{99}]$$
 $$\exists n[n=k_{88}]\rightarrow (\forall n[2n\neq k_{99}]\;\vee\;\forall n[2n+1\neq k_{99}])$$
 
 `$n=k_{88}$' here stands for: `$n$ is the least number $m$ such that at the places $m, m+1, \ldots, m+87$ in the decimal expansion of $\pi$ we find the value $8$'.
 
 Theorems implying a reckless statement themselves will be considered reckless.
 \subsection{Brouwer's example and the halting problem}\label{SS;haltingproblem}
 
The above example of an undecided proposition, `$\exists n[n=k_{99}]$', is simple and very important. A first such example was given by Brouwer  in \cite{brouwer08}.

   The example is related to  the   discovery, in 1936, of the `\textit{halting problem}' by A.M.~Turing (1912-1954). We briefly describe this problem using the approach developed by S.C.~Kleene (1909-1994), see \cite[Chapter XI]{kleene52a}. 
  
  Natural numbers are used for coding \textit{algorithms} for (partial) functions from $\mathbb{N}$ to $\mathbb{N}$. 
  There is a computable predicate on triples of natural numbers called $T$. For all $e,n,z$, $T(e,n,z)$ if and only if $z$ codes a successful computation according to the algorithm coded by $e$ at the argument $n$. The algorithm coded by $e$ \textit{halts at $n$} if and only if $\exists z[T(e,n,z)]$. 
  There is a function $U$ from $\mathbb{N}$ to $\mathbb{N}$  that extracts from every $z$ coding a computation the outcome $U(z)$ of the computation.
  A function $\alpha$ from $\mathbb{N}$ to $\mathbb{N}$ is \textit{computable} if and only if there exists $e$ such that \\$\forall n\exists z[T(e,n,z)\;\wedge\;U(z)=\alpha(n)]$. 
  A number $e$ with this property is called an \textit{index} of the function $\alpha$.
  
   Every number $e$ determines a partial function  $\varphi_e$ from $\mathbb{N}$ to $\mathbb{N}$: \\  for all $n$, $\varphi_e(n) = U(z)$, where $z$ is the least number $w$ such that $T(e, n, w)$. \\$\varphi_e$ is only defined at $n$ if $\exists w[T(e,n,w)]$.

\smallskip
   The \textit{(self-)halting problem} is the question if there exists a computable function $\alpha$ from $\mathbb{N}$ to $\mathbb{N}$ such that, for each $e$, 
 \[\alpha(e)\neq 0\leftrightarrow \exists z[T(e,e,z)]\leftrightarrow \varphi_e \;\mathit{is\;defined\;at\;e}.\] Such a computable function $\alpha$ would be called a \textit{solution to the halting problem}.
 
  Turing's theorem is that there is no solution to the halting problem.
 
 \smallskip One may prove, without much difficulty, that there exists $e$ such that \\$\exists n[ k_{99}\le n]\leftrightarrow \exists z[T(e,e,z)]$: the example we considered is one of the problems in the scope of the halting problem. 
 Therefore, a solution to the halting problem would also solve the problem: `$\exists n[k_{99}\le n]$'.  This fact gives some apriori probability to Turing's result. 
 
 \subsection{The other logical `constants'}\label{SS:otherlogic} The so-called \textit{proof-theoretical interpretation} of the disjunction we sketched in Subsection \ref{SS:disj} is extended to  the other logical `constants' as follows.  
 
 \smallskip
 A proof of `{\it $P$ and $Q$}' is a proof of `$P$' together with a proof of `$Q$'. 
 
 A proof of `{\it if $P$, then Q}'  is\\a proof of `$Q$' possibly using `$P$' as an extra assumption. 
 
 A proof of `{\it not $P$}' is a proof of `if $P$, then 0=1'.
 
 \smallskip For explaining the quantifiers, we need the notion of a {\it propositional function}. Let $V$ be a set and let $v\mapsto P(v)$ a function that associates to every $v$ in $V$ a proposition $P(v)$. We then may introduce propositions $\exists x \in V[P(x)]$ and $\forall x\in V[P(x)]$ by stipulating:
 
 A proof of `$\exists x\in V[P(x)]$' consists of an element $v$ of $V$ together with a proof of the proposition `$P(v)$'. 
 
 A proof of $\forall x \in V[P(x)]$ is a function $v\mapsto p(v)$ associating to every $v$ in $V$ a proof $p(v)$ of the proposition `$P(v)$'. 
 
 \smallskip These explanations are not to be considered as exact definitions. They only indicate a provisional intention as to how we want to use our words. 
 
 \section{A test case: the Intermediate Value Theorem}\label{S:ivt}
\subsection{The theorem}  A \textit{function} from $\mathcal{X}\subseteq\mathcal{R}$ to $\mathcal{R}$ is an effective  method $f$, that, given any input $x$ from $\mathcal{X}$, will produce a well-defined outcome $f(x)$ in $\mathcal{R}$, of course respecting the fundamental relation of {\it real coincidence}, i.e. \begin{quote} {\it for all $x,y$ in $\mathcal{X}$, if $x=_\mathcal{R}y$, then $f(x)=_\mathcal{R} f(y)$}. \end{quote}    We take the notion of a function as a \textit{primitive} notion, i.e. we do not explain what a function is by using notions that have been introduced before.

\begin{definition}\label{D:realcontinuousfunction} Let $f$ be a function from $\mathcal{X}\subseteq  \mathcal{R}$ to $\mathcal{R}$ and let $x$ in $\mathcal{X}$ be given. 
\\$f$ is \emph{continuous at $x$} if and only if $\forall p\exists m\forall y\in\mathcal{X}
[|y-x|<\frac{1}{2^m}\rightarrow |f(y)-f(x)|<\frac{1}{2^p}]$. 
\\$f$ is \emph{continuous} if and only if $f$ is continuous at every $x$ in $\mathcal{X}$.  
 \end{definition} 
   The \textit{Intermediate Value Theorem} states the following:
  \begin{quote}
  
\textit{Let $f$ be a  continuous function from $[0,1]$ to $\mathcal{R}$. 
\\ Then $\forall y \in \mathcal{R}[f(0)\le y \le f(1)\rightarrow\exists x \in [0,1][f(x) = y]]$.}\end{quote}

The first version of this Theorem dates from 1817 and is due to B. Bolzano (1781-1848), see \cite{bolzano}.  

Note that the intuitionistic mathematician reads the statement of the Theorem differently from his non-intuitionistic colleague: he understands the logical constants constructively, as sketched in Section \ref{S:tndfalse}. 

\subsection{The  proof as we learnt it} One may use the fruitful method of {\it successive bisection}, as follows. 

\smallskip Let $y$ be given such that $f(0)\le y\le f(1)$.
 
Define   $x$ in $[0,1]$, step by step, as follows.

Define $x(0) := (0,1)$. 
 
\smallskip  Now let $n$ be given such that  $x(n)$ has been defined already.

 Consider $m:=\frac{x'(n)+x''(n)}{2}$, the \textit{midpoint} of the rational interval $\bigl(x'(n), x''(n)\bigr)$.

 If $f(m)\le y$, define $x(n+1) =\bigl(m, x''(n)\bigr)$. 
 
 If  $y<f(m)$, define $x(n+1)= \bigl( x'(n),m\bigr)$.

 \smallskip
 This completes the definition of $x$.
 
 \medskip
 $x$ is a well-defined real number, and, as one verifies by induction, \\
for every $n$, $x''(n)-x'(n)=\frac{1}{2^n}$ and $f\bigl(x'(n)\bigr) \le y \le f\bigl(x''(n)\bigr)$.

We claim: $f(x) =y$, and we prove this claim as follows.

\smallskip Assume $y<_\mathcal{R} f(x)$. Using the continuity of $f$ at $x$, find $r$ such that \\
$\forall z\in[0,1][|x-z|<\frac{1}{2^r}\rightarrow y<_\mathcal{R} f(z)]$.

Note: $|x-x'(r+1)|<\frac{1}{2^r}$, and, therefore, $y<_\mathcal{R} f\bigl(x'(r+1)\bigr)]$, but also   \\$f\bigl(x'(r+1)\bigr)<_\mathcal{R}y$. 

Contradiction. 
  
  Conclude: $\neg\bigl(y <_\mathcal{R} f(x)\bigr)$, i.e. $f(x)\le_\mathcal{R} y$.
  
   \smallskip By a similar argument, conclude: $y \le_\mathcal{R} f(x)$ and: $f(x)=_\mathcal{R} y$. 
\subsection{An objection to the `proof'} The construction of the point $x$, in the above argument, can not be carried out, as, in general, we are unable to decide: \begin{center} $f(m)\le_\mathcal{R} y$ {\it or} $\;y<_\mathcal{R} f(m)$. \end{center}
\subsection{Two counterexamples to the theorem}\label{SS:counterivt} We give two examples showing that the statement of the theorem,  if one reads it constructively, sometimes fails to be true. 

\medskip \textit{First example}.

Let $\rho_1=(-\frac{1}{2})^{k_{99}}$ be the number oscillating up and down around $0_\mathcal{R}$  introduced in Definition \ref{D:osc}.
 Define a function $f_0$ from $[0,1]$ to $\mathcal{R}$ such that \begin{enumerate}[\upshape (i)] \item $f_0(0) = 0$ and $f_0(1) = 1$ and $f_0(\frac{1}{3})=f_0(\frac{2}{3}) = \frac{1}{2}+\rho_1$, and \item $f_0$ is linear on the interval $[0,\frac{1}{3}]$, on the interval $[\frac{1}{3}, \frac{2}{3}]$ and on the interval $[\frac{2}{3},1]$. \end{enumerate}

Note:  $ 0 <\rho_1\leftrightarrow\forall x\in[\frac{1}{3},1][f_0(x)>\frac{1}{2}]$ and \\$\rho_1< 0\leftrightarrow\forall x\in[0,\frac{2}{3}][f_0(x)<\frac{1}{2}]$.

\smallskip

Assume we 
find $x$ such that $f_0(x) = \frac{1}{2}$.

Find $p$ such that $x''(p)-x'(p)<\frac{1}{3}$ and note: 
\textit{either} $ \frac{1}{3}<x'(p)$ \textit{or}  $x''(p)< \frac{2}{3}$.

\smallskip
If $\frac{1}{3}<x'(p)$, then $\frac{1}{3}<x$ and  $\neg (0<\rho_1)$, that is: $ \rho_1\le 0$, and,

if $x''(p)<_\mathbb{Q} \frac{2}{3}$, then $x<\frac{2}{3}$ and  $\neg (\rho_1<0)$, that is: $0 \le \rho_1$.

\smallskip
 Conclude:  either  $0\le \rho_1$ or $\rho_1 \le 0$.

This is a \textit{reckless} or \textit{hardy} conclusion.

\bigskip

\textit{Second example}.

Let $\rho_0=(\frac{1}{2})^{k_{99}}$ the number oscillating above $0_\mathcal{R}$ introduced in Definition \ref{D:osc} and let $\rho_2=\rho_0+\rho_1=(\frac{1}{2})^{k_{99}}+(-\frac{1}{2})^{k_{99}}$ be the number oscillating between $0_\mathcal{R}$ and $2\cdot \rho_0$, also introduced in Definition \ref{D:osc}.

 Define a function $f_1$ from $[0,1]$ to $\mathcal{R}$ such that \begin{enumerate}[\upshape (i)] \item $f_1(0) = 0$ and $f_1(\frac{1}{2}) =\rho_0$ and $ f_1(1)=\rho_2$, and \item $f_1$ is linear on the interval $[0,\frac{1}{2}]$ and on the interval $[\frac{1}{2},1]$. \end{enumerate}
 
 Note: $f_1(0)\le \frac{1}{2}\cdot\rho_2 \le f_1(1)$.
 
 \noindent If $\exists n[2n=k_{99}]$, \\then $0<\rho_0\;\wedge\;\rho_2=2\cdot\rho_0$ and $ \forall x\in [0,1][f_1(x) = \frac{1}{2}\cdot\rho_2 \rightarrow x=\frac{1}{2}]$, and, \\if $\exists n[2n+1=k_{99}]$,\\ then $0<\rho_0 \;\wedge\;\rho_2=0$ and $\forall x\in [0,1][f_1(x)= \frac{1}{2}\cdot\rho_2\rightarrow (x=0\;\vee\; x=1)]$. 
  
\smallskip  Suppose we find $x$ in $[0,1]$ such that $f_1(x)=\frac{1}{2}\cdot\rho_2$.
  
 Find $p$ such that $x''(p)-x'(p) <\frac{1}{2}$ and note:\\ \textit{either} $x''(p)<\frac{1}{2}$ \textit{or} $\frac{1}{2}<x'(p)$ \textit{or} $0<x'(p)$ and $x''(p) <1$.
 
 If $x''(p)<\frac{1}{2}$ or $\frac{1}{2}<x'(p)$, then $\neg \exists n[2n=k_{99}]$.
 
 If $0<x'(p)$ and $x''(p) <1$, then $\neg \exists n[2n+1=k_{99}]$.
 
 We thus see: $\neg \exists n[2n=k_{99}]$ or $\neg \exists n[2n+1=k_{99}]$.
 
 This is a \textit{reckless} or \textit{hardy} conclusion.

\smallskip Note that the above counterexamples pose a serious problem to the mathematician. If the Intermediate Value Theorem does not stand a straightforward constructive reading, what then is its meaning? The defender of the theorem  must mumble something like: `Well, I did not mean that you really could \textit{find} a point where $f$ assumes the intermediate value, but $\ldots$'

 Mathematical theorems without `numerical content' are a shame and an embarassment. 

 We should not, because of these examples, put aside the Intermediate Value Theorem as nonsense, but, arguing carefully, try and find related statements that are constructively true.   
\subsection{An approximate version} It helps to weaken the conclusion of the Intermediate Value Theorem by requiring only that, for every given accuracy $\frac{1}{2^p}$, there exists a point where the function assumes a value closer to the given intermediate value than $\frac{1}{2^p}$.

Brouwer, who knew his own famous fixed-point theorem to be constructively false, formulated and proved such an \textit{approximate} fixed-point theorem, see \cite{brouwer51} \footnote{The 1-dimensional case of Brouwer's fixed-point theorem is closely related to the Intermediate value Theorem.}. Note that the proof of the new theorem still is using the method of successive bisection. 

\begin{theorem}[The Approximate Intermediate Value Theorem]\label{T:appivt}$\;$\\ \indent Let $f$ be a  continuous function from $[0,1]$ to $\mathcal{R}$.

Then $\forall p\forall y \in \mathcal{R}[f(0)\le y \le f(1)\rightarrow\exists x \in [0,1][ y -\frac{1}{2^p}<f(x) < y+\frac{1}{2^p}]$.\end{theorem}

\begin{proof}Let $p$ be given and 
 let $y$ be given such that $f(0)\le y\le f(1)$.
 
Define   $x$ in $[0,1]$, step by step, as follows.

Define $x(0) := (0,1)$ and 
 note: $f\bigl(x'(0)\bigr)-\frac{1}{2^{p+1}}<y<f\bigl(x''(0)\bigr)+\frac{1}{2^{p+1}}$. 
 
\smallskip  Now let $n$ be given such that  $x(n)$ has been defined already.

 Consider $m:=\frac{x'(n)+x''(n)}{2}$ and $z:=f(m)$. 
 
 Find $l$:= the least $k$ such that  both $z''(k)-z'(k) <\frac{1}{2^{p+1}}$ and \\
 $y''(k)-y'(k)<\frac{1}{2^{p+1}}$
 and define $s:=z(l)$. 
 
  \smallskip 1. 
 If $s''< y'(l) +\frac{1}{2^{p+1}}$, define $x(n+1) =(m, x''(n))$. \\Then:  $f(m)\le s''<y'(l) +\frac{1}{2^{p+1}}\le y+\frac{1}{2^{p+1}}$ and: $f(m) -\frac{1}{2^{p+1}}< y$. 
 
\smallskip 2. If not $s'' < y'(l)+\frac{1}{2^{p+1}}$, define $x(n+1)= (x'(n), m)$. \\Then: $y''(l)<y'(l)+\frac{1}{2^{p+1}}\le s''<s'+\frac{1}{2^{p+1}}$ and: $y<f(m)+\frac{1}{2^{p+1}}$.

 \smallskip
 This completes the definition of $x$.
 
 \medskip
 $x$ is a well-defined real number, and, \\
for every $n$, $x''(n)-x'(n)=\frac{1}{2^n}$ and $f\bigl(x'(n)\bigr) -\frac{1}{2^{p+1}}<y<f\bigl(x''(n)\bigr)+\frac{1}{2^{p+1}}$.

Using the fact that  $f$ is continuous at $x$,
find $r$ such that \\
$\forall z\in[0,1][x-\frac{1}{2^r}<z<x+\frac{1}{2^r}\rightarrow f(x) -  \frac{1}{2^{p+1}}<f(z)<f(x)+\frac{1}{2^{p+1}}]$.

 \smallskip  
1. Note:  $x-\frac{1}{2^r}< x'(r+1) \le x$, and: \\$f(x)< f(x'(r+1)) +\frac{1}{2^{p+1}}< y+\frac{1}{2^{p+1}} + \frac{1}{2^{p+1}} = y +\frac{1}{2^p}$.  
  
\smallskip 
 2. Note: $ x\le x''(r+1)< x+\frac{1}{2^r}$, and:\\ $y<f\bigl(x''(r+1)\bigr)+\frac{1}{2^{p+1}}< f(x) +\frac{1}{2^{p+1}} + \frac{1}{2^{p+1}}= f(x)+\frac{1}{2^p}$.

  Conclude: $y -\frac{1}{2^p}<f(x)<y+\frac{1}{2^p}$.\end{proof}
  
\subsection{Locally non-constant functions} One may also try to restrict the class of functions to which the Intermediate Value Theorem applies. Observe that the two functions mentioned as counterexamples in Subsection \ref{SS:counterivt} have the property that, over a whole interval, one does not know if they change their value. Let us try to forbid such behaviour. We require  constructive evidence that, for every $y$, for every interval $(p,q)$, the function  does not have, on the interval $(p,q)$, the constant value $y$. 

\begin{definition}A function $f$  from $[0,1]$ to $\mathcal{R}$ is \emph{locally non-constant} if and only if \\$\forall y \in \mathcal{R}\forall p \in \mathbb{Q}\forall q \in \mathbb{Q}[0<p<q<1 \rightarrow \exists x \in [0,1][p<x<q \;\wedge\; f(x) \;\#_\mathcal{R}\; y]]$. \end{definition}
 The following Lemma says that, if, for a given intermediate value $y$,  the function   $f$ assumes, in any given interval,  a value positively different from $y$, then there is  a point where $f$ assumes the value $y$. 
\begin{lemma}\label{L:ivtlnc}Let $f$ be a  continuous function from $[0,1]$ to $\mathcal{R}$.

Let $y$ be given such that $f(0)\le y \le f(1)$ and 

$\forall p \in \mathbb{Q}\forall q \in \mathbb{Q}[0<p<q<1 \rightarrow \exists x \in [0,1][p<x<q \;\wedge\; f(x) \;\#_\mathcal{R}\; y]]$.

Then $\exists x \in [0,1][f(x) = y]$. \end{lemma}
\begin{proof} Let $y$ be given such that $f(0)\le y \le f(1)$.

Define  $x$ in $[0,1]$, as follows, step by step.

Define $x(0) := (0,1)$.

\smallskip 
Now let $n$ be given such that  $s:=x(n)$ has been defined.

 Using the continuity of $f$, find    a rational number $q$ such that
  \\$\frac{2}{3}s'+_\mathbb{Q}\frac{1}{3}s''<q<\frac{1}{3}s'+_\mathbb{Q}\frac{2}{3}s''$ and
  $f(q)\;\#_\mathcal{R}\;y$. 
  
 If $f(q)<_\mathcal{R} y$, define $x(n+1):= \bigl(x'(n), q\bigr)$, and,\\
if $y<_\mathcal{R} f(q)$, define $x(n+1):= \bigl(q,x''(n)\bigr)$.

This completes the definition of $x$.

\smallskip
$x$ is a well-defined real number, as,  for each $n$, \\ 
$x'(n)\le x'(n+1)\le x''(n+1)\le x''(n)$  and $x''(n+1) -x'(n+1) \le \frac{2}{3}\bigl(x''(n) - x'(n)\bigr)$.

Moreover, for each $n$,  $f\bigl(x'(n)\bigr) \le_\mathcal{R} y \le_\mathcal{R} f\bigl(x''(n)\bigr)$.

\smallskip

Assume: $f(x) <_\mathcal{R} y$.

Find $p$ such that  $\frac{1}{2^p}< y-_\mathcal{R} f(x)$, and 
find $q$ such that \\$\forall z \in [0,1][|z-x| < \frac{1}{2^q} \rightarrow |f(z) - f(x)|<_\mathcal{R} \frac{1}{2^p}]$. \\
Then   $\forall z \in [0,1][|z-x| < \frac{1}{2^q} \rightarrow f(z) <_\mathcal{R} y$. Find $r$ such that $(\frac{2}{3})^r<\frac{1}{2})^q$.  
\\Conclude: $f(x''(r)) <_\mathcal{R} y$. Contradiction.

We thus see: $\neg(f(x) <_\mathcal{R} y)$, that is: $f(x) \ge_\mathcal{R} y$.  

\smallskip For similar reasons: $f(x) \le_\mathcal{R} y$.  Conclude: $f(x) =_\mathcal{R} y$. \end{proof}

\begin{theorem}[The Intermediate Value Theorem for locally non-constant functions]\label{T:ivtlnc}$\;$ 
\indent Let $f$ be a continuous and locally non-constant function from $[0,1]$ to $\mathcal{R}$. 
Then $\forall y[f(0)\le y \le f(1)\rightarrow \exists x\in[0,1][f(x)=y]]$. \end{theorem} 
\begin{proof} Use Lemma \ref{L:ivtlnc}. \end{proof}
   Many real functions may be proven to be locally non-constant and Theorem \ref{T:ivtlnc} is widely applicable\footnote{Theorems \ref{T:appivt} and \ref{T:ivtlnc} may be found in \cite[Theorem 2.5]{bridgesrichman}.}. 
   
   \smallskip From a constructive point of view, the next result,  Theorem \ref{T:ivtneg} is not a nice result. The result nevertheless is instructive. The intuitionistic mathematician may say: `{\it Yes, this explains the classical mathematician's belief that the Intermediate Value Theorem is true.  He means no more than that it is impossible to have  constructive evidence, for every point $x$ in the domain of the function, that the value assumed at $x$ is apart from the given intermediate value. Unfortunately, even if this impossible, it may also be impossible  to find a point where the intermediate value is assumed.}'

\begin{theorem}[The Negative Intermediate Value Theorem]\label{T:ivtneg}$\;$ \\ \indent Let $f$ be a continuous function from $[0,1]$ to $\mathcal{R}$.

Then $\forall y[f(0)\le y\le f(1)\rightarrow \neg\forall x \in [0,1][f(x) \;\#_\mathcal{R}\; y]]$. \end{theorem}
\begin{proof} Use Lemma \ref{L:ivtlnc}. \end{proof}
\subsection{At most countably many exceptions}
One may also observe that the Intermediate Value Theorem goes through for `most' intermediate values. 
\begin{theorem}[The Intermediate Value Theorem holds with at most countably many exceptions\footnote{In \cite{bishopbridges}, one finds various theorems `\textit{with countably many exceptions}', for instance, Theorem 4.9: {\it if $f; [0,1]\rightarrow \mathcal{R}$ is uniformly continuous, then for all but countably many $y\ge \inf(f)$ the set\\ $\{x \in [0,1]\mid f(x) \le y\}$ is \emph{(constructively)} compact.}}]\label{T:ivtexceptions}Let $f$ be a continuous function from $[0,1]$ to $\mathcal{R}$. 

There exists an infinite sequence $y_0, y_1, \ldots$ of elements of $[0,1]$ such that 

$\forall y [\bigl(f(0)\le y\le f(1) \;\wedge\;\forall n[ y \;\#_\mathcal{R}\; y_n]\bigr) \rightarrow \exists x \in [0,1][f(x) = y]].$ \end{theorem}

\begin{proof} Let $q_0, q_1, q_2 \ldots$ be an enumeration of all rational numbers $q$ in $[0,1]$. For each $n$,  define $y_n:=f(q_n)$. 

Let $y$ in $[0,1]$ be given such that $\forall n[ y \;\#_\mathcal{R} \;y_n]$. 

Define  $x$ in $[0,1]$, step by step.

Define $x(0) := (0,1)$ and note: $f(0)  <_\mathcal{R} y<_\mathcal{R} f(1) $.

Now let $n$ be given such that   $x(n)$ is defined and $f\bigl(x'(n)\bigr) <_\mathcal{R} y <_\mathcal{R} f\bigl(x''(n)\bigr)$.

 Consider $m:=\frac{x'(n)+x''(n)}{2}$.
  If $f(m) <_\mathcal{R} y$,  define: $x(n+1):= (m, x''(n))$, and,
  if $ y<_\mathcal{R} f(m)$, define: $x(n+1):= (x'(n), m)$.  
   Then \\$f\bigl(x'(n+1)\bigr) <_\mathcal{R} y <_\mathcal{R} f(\bigl(x''(n+1)\bigr)$.

This completes the definition of $x$.

\smallskip
Using the argument given in the last part of the proof of Lemma \ref{L:ivtlnc}, one proves: $f(x) =_\mathcal{R} y$.
\end{proof}
\subsection{Perhaps, perhaps, perhaps,$\ldots$}\label{SS:mmp} $\;$

Suppose we change the conclusion of the Intermediate Value Theorem into \begin{center} $(\ast)$ $\exists x_0[f(x_0)\;\#_\mathcal{R} \;y \rightarrow \exists x_1[f(x_1) =_\mathcal{R} y]]$,\\{\it Perhaps\footnote{A similar use of the expression `\textit{Perhaps}' is made in Section \ref{S:finiteness}.}, \it $f$ assumes the value $y$.}\end{center}  One may think of $(\ast)$ as follows. You are offered $x_0$ where $f$ should assume the value $y$. If you discover $x_0$ clearly  does not work, you may go back to the shop and ask for a better number and will be offered $x_1$ and $x_1$ certainly will work. 

One may reconsider the two counterexamples given in Subsection \ref{SS:counterivt}, the functions $f_0$ and $f_1$, in the light of this definition.

\smallskip
We have seen that the statement `\textit{$f_0$ assumes the value $\frac{1}{2}$}' is reckless, but, if $f_0(\frac{1}{2}) \;\#_\mathcal{R}\; \frac{1}{2}$, then  $\rho_1\;\#_\mathcal{R}\; 0$ and: {\it either}  $0 <_\mathcal{R} \rho_1$ and $f_0(\frac{1}{3}\cdot\frac{1}{1+2\cdot\rho_1})=\frac{1}{2}$,  {\it or} $\rho_1<_\mathcal{R} 0$ and $f_0(\frac{2}{3}-\frac{2\cdot\rho_1}{1-2\cdot\rho_1})=\frac{1}{2}$. We thus see: {\it perhaps, $f_0$ assumes the value $\frac{1}{2}$.}

\smallskip We also have seen that the statement `\textit{$f_1$ assumes the value $\frac{1}{2}\cdot\rho_2$}' is reckless, but, if $f_1(0)\;\#_\mathcal{R} \;\frac{1}{2}\cdot\rho_2$, then $\rho_2\;\#_\mathcal{R}\;0_\mathcal{R}$ and $\rho_0=_\mathcal{R}\rho_1$ and $f_1(\frac{1}{2})=_\mathcal{R}\frac{1}{2}\cdot \rho_2$. We thus see: {\it perhaps, $f_1$ assumes the value $\frac{1}{2}\cdot \rho_2$}.

\smallskip The funny thing is that the operation {\it Perhaps} may be repeated. One may consider the statement:

 \begin{center} $\exists x_0[f(x_0)\;\#_\mathcal{R}\; y \rightarrow \exists x_1[f(x_1) \#_\mathcal{R}\;y \rightarrow \exists x_2[f(x_2)=_\mathcal{R}y]]]$,\\{\it Perhaps, perhaps, $f$ assumes the value $y$}. \end{center}

   We give an example showing that this statement indeed may be weaker than the previous one, with only one `perhaps'.
 
\smallskip \textit{Another  example}.

Consider  
 $\rho_1=(-\frac{1}{2})^{k_{99}}$ but also $\rho_3:= (-\frac{1}{2})^{k_{88}}$. \\
 Define a function $f_2$ from $[0,1]$ to $\mathcal{R}$ such that \begin{enumerate}[\upshape (i)] \item $f_2(0) = 0$ and $f_2(1) = 1$ and $f_2(\frac{1}{5})=f_2(\frac{2}{5})= \frac{1}{2}+\rho_1$, and\\ $f_2(\frac{3}{5})=f_2(\frac{4}{5})= \frac{1}{2}+\rho_3$, and, \item for each $i\le 4$, $f_2$ is linear on the interval $[\frac{i}{5},\frac{i+1}{5}]$. \end{enumerate} 
 
 Let us first observe: if $f_2(\frac{1}{5})\;\#_\mathcal{R}\; \frac{1}{2}$ and also $f_2(\frac{3}{5})\;\#_\mathcal{R}\; \frac{1}{2}$, then $\rho_1\;\#_\mathcal{R}\; 0$ and $\rho_3  \;\#_\mathcal{R} \;0$ and $\exists x[f(x)=_\mathcal{R} \frac{1}{2}]$. Conclude: \begin{quote} {\it Perhaps, perhaps, $f_2$ assumes the value $\frac{1}{2}$}. \end{quote}
 
 Now assume: {\it Perhaps, $f_2$ assumes the value $\frac{1}{2}$}. \\Find $x_0$ such that, if  $f_2(x_0)\;\#_\mathcal{R}\;\frac{1}{2}$, then $\exists x_1[f_2(x_1) =_\mathcal{R} \frac{1}{2}]$. 
 
 Distinguish two cases.
 
 \smallskip {\it Case (1)}. $x_0 <_\mathcal{R} \frac{8}{15}$. Assume: $\rho_1<_\mathcal{R} 0$. \\Define $x_1:=\sup(x_0, \frac{2}{5})$ and note $f(x_1) = f(\frac{2}{5})  + (x_1-\frac{2}{5})\cdot 5(\rho_3-\rho_1)=\\\frac{1}{2} +\rho_1 + (x_1-\frac{2}{5})\cdot 5(\rho_3-\rho_1)\le_\mathcal{R} \frac{1}{2} + \rho_1 + (\frac{8}{15}-\frac{2}{5})\cdot 5(\rho_3-\rho_1)=\\\frac{1}{2} +\frac{1}{3}\rho_1 + \frac{2}{3}\rho_3$. \\Conclude: $f(x_0)\le_\mathcal{R} f(x_1) <_\mathcal{R} \frac{1}{2}$ or $0<_\mathcal{R} \rho_3$. 
 
 In both cases, one may conclude: \\$\exists x \in [0,1][f_2(x)=_\mathcal{R} \frac{1}{2}]$ and: $0\le_\mathcal{R}\rho_3\;\vee\; \rho_3\le_\mathcal{R} 0$.  
 
 We thus see: if $\rho_1<_\mathcal{R} 0$, then $0\le_\mathcal{R}\rho_3\;\vee\; \rho_3\le_\mathcal{R} 0$, i.e. \\if $\exists n[2n+1=k_{99}]$, then $\forall n[2n+1\neq k_{88}]\;\vee\;  \forall n[2n\neq k_{88}]$, \\a {\it reckless} or {\it hardy} conclusion. 
 
 \smallskip  {\it Case (2)}. $\frac{7}{15}<_\mathcal{R} x_0 $. By a similar argumnet, we obtain the result:  \\ if $0<_\mathcal{R} \rho_3$, then $0\le_\mathcal{R}\rho_1\;\vee\; \rho_1\le_\mathcal{R} 0$, i.e. \\if $\exists n[2n=k_{88}]$, then $\forall n[2n+1\neq k_{99}]\;\vee\;  \forall n[2n\neq k_{99}]$, \\a {\it reckless} or {\it hardy} conclusion.  
 
 \smallskip We thus see that the statement \begin{quote} `{\it Perhaps, $f_2$ assumes the value $\frac{1}{2}$}' \end{quote} has reckless consequences. 
 
 \medskip One now may define: \begin{center} {\it $Perhaps_0(f,y)$  if and only if $\exists x[f(x)=y]$} \end{center} and, for each $n$, \begin{center} {\it $Perhaps_{n+1}(f, y)$ if and only if $\exists x[f(x)\;\#_\mathcal{R}\;y \rightarrow Perhaps_n( f, y)] $ }\end{center} and also\begin{center} {\it $Perhaps_\omega(f, y)$ if and only if $\exists n[Perhaps_n(f, y)]$},  \end{center}and \begin{center}{\it $Perhaps_{\omega+1}(f, y)$ if and only if $\exists x[f(x)\;\#_\mathcal{R}\;y\rightarrow Perhaps_\omega(f, y)]$}. \end{center} and prove that in general, alle these statements have different meanings. 
  Going this path further, increasing the number of `{\it perhapses}'  further into the transfinite, one  finds   uncountably many intuitionistically different versions of the notion: `$\exists x[f(x)=_\mathcal{R}y]$', see \cite{veldman05}.  
\section{A  function from $\mathcal{R}$ to $\mathcal{R}$ is nowhere positively discontinuous}
 
   In the next three Sections we study the intuitionistic claim that every real function is continuous.  The first theorem, Theorem \ref{T:negativecontinuity}, is weak and negative.

\subsection{No positive discontinuity}

 \begin{definition}\label{D:continuity}Let $f$ be a function from $\mathcal{X}\subseteq\mathcal{R}$ to $\mathcal{R}$ and 
let   $x$ in $\mathcal{X}$ be given.

  \smallskip $f$ is \emph{positively discontinuous} at $x$ if and only if \\$\exists p\forall m\exists y\in \mathcal{X}[|y-x| < \frac{1}{2^m} \;\wedge\; | f(y)-f(x)|\ge\frac{1}{2^p}].$

  \end{definition}

 \begin{theorem}\footnote{Cf. \cite[Theorem 1]{brouwer75}.}\label{T:negativecontinuity}Let $f$ be a  function from $\mathcal{X}\subseteq\mathcal{R}$ to $\mathcal{R}$. Assume there exists $x$ in $\mathcal{X}$ such that $f$ is positively discontinuous at $x$. Then there exists $z$ in $\mathcal{R}$ such that we can \emph{not} calculate $f(z)$, so $z$ can \emph{not} belong to $\mathcal{X}$.  \end{theorem}
 
 \begin{proof} Let $\mathcal{X}, f, x$ satisfy the conditions of the theorem.
 
 Without loss of generality, we may assume: $x=0=f(0)$.

 Find $p$ such that $\forall m\exists y\in \mathcal{X}[| y |< \frac{1}{2^m} \;\wedge\; | f(y)|>\frac{1}{2^p}].$

Find an infinite sequence  $y_0, y_1, \ldots$ of reals  such that \\$\forall m[| y_m | < \frac{1}{2^m} \;\wedge\; | f(y_m)|>\frac{1}{2^p}].$

\smallskip We  give two arguments establishing the conclusion of the theorem.
 
 \medskip {\it First argument}.
 
I start constructing a  real number $z$ in an infinite sequence of steps, defining successively  $z(0), z(1), z(2), \ldots$. \\ I promise to take care that for each $n$, $z'(n)\le z'(n+1)\le z''(n+1)\le z''(n)$ and $z''(n)-z'(n) =\frac{1}{2^{n-1}}$. I am free, within these bounds,  to choose each $z(n)$ as I like it.

I have in mind  to perhaps `freeze' the free development of $z$, at some step $n$,  by making a decision that determines all values from that step $n$ on.

  I define $z(0)=(-1,1)$ and  do not freeze $z$ at step $0$. For each $n>0$,  if $z$ has not yet been frozen at one of the earlier steps, 
\textit{either} I do not freeze $z$ at this step and  define $z(n) = (-\frac{1}{2^n}, \frac{1}{2^n})$,  \textit{or} I freeze $z$ by defining, for each $i$,   $z(n+i) = y_{n-1}(q+i)$, where $q$ is the least $k$ such that $-\frac{1}{2^{n-1}}<y'_{n-1}(k)\le y''_{n-1}(k)<\frac{1}{2^{n-1}}$.

   $z$ is a well-defined real number.

Note: if I never make use of my freedom to freeze $z$, then $z=_\mathcal{R} 0$ and $f(z)=_\mathcal{R}0$.

Note: if I do use this freedom at some step  $n>0$, then   $z=_\mathcal{R}y_{n-1}$ and    \\$|f(z)|=_\mathcal{R} |f(y_{n-1})|>_\mathcal{R}\frac{1}{2^p}$. 

But \textit{I am free} and do not know myself if I ever will make use of the possibility to freeze $z$ or not.

I thus am unable to find an approximation of $f(z)$ smaller than $\frac{1}{2^p}$.

We must conclude: $f$ is \textit{not} defined at the argument $z$ and $z$ does \textit{not} belong to the domain $\mathcal{X}$ of $f$. 

\medskip {\it Second argument}.

 Define $z$ such that, for each $n$, \\if $ n<k_{99}$, then $z(n) = (-\frac{1}{2^n}, \frac{1}{2^n})$, and, if $ k_{99}\le n$, then $z(n) = y_{k_{99}}(q +n)$, \\where $q$ is the least $k$ such that $-\frac{1}{2^{k_{99}}}<y'_{k_{99}}(k)\le y''_{k_{99}}(k)<\frac{1}{2^{k_{99}}}$.

\smallskip

 Note: if $\forall n[n<k_{99}]$, then $f(z) =_\mathcal{R} f(0_\mathcal{R})=_\mathcal{R}0$, and, \\if $\exists n[n= k_{99}]$, then $|f(z)|=_\mathcal{R}|f(y_{k_{99}}|>_\mathcal{R}\frac{1}{2^p}$.

Assume we may calculate $f(z)$. By finding an approximation of $f(z)$ smaller than $\frac{1}{2^{p}}$ we will be able to decide: $f(z) \;\#_\mathcal{R}\; 0$ or $|f(z)| <_\mathcal{R} \frac{1}{2^p}$.

  \textit{Case (a)}. $f(z) \;\#_\mathcal{R}\;0$, and, therefore: $\neg \forall n[n<k_{99}]$.
  
  \textit{Case (b)}.   $|f(z)|< \frac{1}{2^p}$, and, therefore: $\forall n[n<k_{99}]$. 
  
  Conclude: $\neg \forall n[n<k_{99}]\;\vee\;\forall n[n<k_{99}]$. 
  
  This is a \textit{reckless} conclusion. 
  
  We must admit that, at this point of time, $f$ is \textit{not} defined at $z$ and $z$ does \textit{not} belong to the domain $\mathcal{X}$ of $f$.
  \end{proof}
  
Clearly, a function from $\mathcal{X}\subseteq \mathcal{R}$ with a positive discontinuity can \textit{not} be calculated at every point of $\mathcal{R}$.  
  In Brouwer's terminology: such a function  is \textit{not} a \textit{full} function.
 
  \section{The meaning of `not'}\label{S:not}
 \subsection{A weak and a strong interpretation}$\;$\\May we conclude, from Theorem \ref{T:negativecontinuity}: \begin{quote} {\it A function $f$ from $\mathcal{R}$ to $\mathcal{R}$ is \emph{not} positively discontinuous at any point $x$ in $\mathcal{R}$}? \end{quote} 
 
We have to be careful. There is some ambiguity in the use of the word `\textit{not}'.  When saying:  
 \begin{center}`$\textit{not:} \;(\exists n[n=k_{99}]\;\vee\; \forall  n[n<k_{99}])$'\end{center} 
 we   mean no more than: \begin{center}`\textit{As yet, we have  no proof of `$\exists n[n=k_{99}\;\vee\;\forall n[n<k_{99}]$'}.  \end{center}
 
 Most of the time, however, we take a proposition `not $P$',  $`\neg  P'$, in the stronger sense indicated in Subsection \ref{SS:otherlogic}: `\textit{Assuming $P$, one is led to a contradiction}'. This strong interpretation of `{\it not}'  is the canonical one. It follows from reading `$not$-$P$' as `$P \rightarrow 0=1$' and interpreting the  implication `$P\rightarrow Q$' as: `{\it Assuming $P$, one may prove $Q$}'.

  Following this strong and canonical interpretation, one may prove, for any proposition $P$,
  
 \begin{center} $\neg\neg (P\vee \neg P)$, \end{center}

 as follows: Assume: $\neg (P\vee\neg P)$, i.e. $P\vee\neg P$ leads to a contradiction. Then both $P$ and $\neg P$ will lead to a contradiction. We thus see: $\neg P$ and: $\neg P$ leads to a contradiction. We therefore have a contradiction. 

Conclude: $\neg(P\vee\neg P)$ leads to a contradiction, i.e.  $\neg\neg (P\vee\neg P)$. 

In particular, we  find: \begin{center} $\neg\neg(\exists n[n=k_{99}]\;\vee\;\forall n[n<k_{99}])$. \end{center}

\subsection{Some useful remarks}\label{SS:useful}  $\;$\\ Defining $Q:= \exists n[n=k_{99}]\;\vee\;\neg \exists n[n=k_{99}]$, we see that there are propositions $Q$ such that $\neg\neg Q$ is true but $Q$ itself is reckless or hardy. On the other hand, for every proposition $Q$, if $Q$, that is: $Q$ has a proof, then $\neg Q$, (the statement that $Q$ leads to a contradiction), leads indeed to a contradiction, i.e. if $Q$, then $\neg\neg Q$. 

Later in this paper, we need the following {\it law of contraposition}:\\ {\it if $P\rightarrow Q$, then $\neg Q \rightarrow \neg P$}, and its consequence: {\it if $P\rightarrow Q$, then $\neg\neg P\rightarrow \neg \neg Q$}.  

This law is easily seen to be valid for the proof-theoretical interpretation,

Using it, one sees that,  
as $Q\rightarrow \neg\neg Q$ is always true, also $\neg\neg\neg Q\rightarrow \neg Q$ is always true.   (The scheme $\neg\neg\neg Q\rightarrow Q$ was discovered by Brouwer and may be called \textit{Brouwer's first law of logic}).

Also note: \\if $(P\;\vee\neg P)\rightarrow \neg Q$, then, as $\neg\neg (P\;\vee\;\neg P)$, one has: $\neg\neg\neg Q$, and, therefore:  $\neg Q$. 

This fact may be useful in practice: if it is our aim to prove a negative statement $\neg Q$ it does no harm to make an extra ssumption $P\vee\neg P$.

\smallskip One should keep in mind: $\neg P$  means: $P$ \textit{is reducible to absurdity}. Therefore, the constructive mathematician has nothing to criticize if someone proves $\neg P$ by reducing $P$ to absurdity. If one does so, one does what one has to do.\footnote{Hardy's \textit{finest weapon} is not bad in itself but can only be used for proving negative statements. Constructively, negative statements are not so useful. There are better statements and finer weapons.} She feels unhappy, however,  if one proves $P$ itself by reducing $\neg P$ to absurdity. The conclusion of the latter procedure is $\neg\neg P$ only, a statement that, most of the time, is  weaker than $P$. 
\subsection{More fugitive numbers}
The following Definition extends Definition \ref{D:k99}.
\begin{definition}[The fugitive numbers $k_\alpha$]\label{D:k_alpha} Let an infinite sequence \\$\alpha=\alpha(0), \alpha(1), \alpha(2), \ldots$ of natural numbers be given.
 For each $n$ we define\footnote{Note that we do not define a natural number $k_\alpha$ bot only the meaning of an expression like `$ k_\alpha\le n$'.} :
  
  $ k_\alpha\le n$ if and only if  $\exists j\le n [\alpha(j)\neq 0]$, and:
  
  $n<k_\alpha$ if and only  if $\forall j\le n [\alpha(j)=0]$, and:
  
  $n=k_\alpha$ if and only if $ k_\alpha\le n$ and $\forall j<n[j<k_\alpha]$, i.e. \\$n$ is the least $j$ such that $\alpha(j)\neq 0$.
  
 $\mathbf{LPO}$, the \textit{limited principle of omniscience}\footnote{This statement got its name from  the American analyst E.~Bishop, who, fifty years after Brouwer, founded his own school of \textit{constructive analysis}, see \cite[page 3]{bridgesrichman}. The name comes from calling the Principle of the Excluded Third $P\;\vee\;\neg P$ the \emph{Principle of Omniscience}, a perhaps infelicitous decision, as the name suggests mathematics is a matter of `\textit{knowing facts}'.}, is the statement \begin{center}$\forall \alpha[\exists n[k_\alpha\le n] \;\vee\;\forall n[n<k_\alpha]].$ \end{center} 
 
 $\mathbf{WLPO}$, the \textit{weak limited principle of omniscience} is the statement \begin{center}$\forall \alpha[\neg\forall  n[n<k_\alpha] \;\vee\;\forall n[n<k_\alpha]].$ \end{center}\end{definition}

Using the second argument in the proof of Theorem \ref{T:negativecontinuity} one may see: \\\textit{If a real function from $\mathcal{R}$ to $\mathcal{R}$ has a positive discontinuity, then $\mathbf{WLPO}$.}

For an informal constructive mathematician like Brouwer, $\mathbf{LPO}$, and also \\$\mathbf{WLPO}$, represent \textit{the absurd}, \textit{the contradiction}, a statement one  feels sure never to be able to prove. For her,   the argument for Theorem \ref{T:negativecontinuity}   establishes, for every  `full' function $f$ from $\mathcal{R}$ to $\mathcal{R}$,  \begin{center} $\neg$(there exists $x$ in $\mathcal{R}$ such that $f$ is positively discontinuous at $x$).\end{center}

In  Section \ref{S:bcp}, we introduce {\it Brouwer's Continuity Principle}, an axiomatic assumption one perhaps finds  plausible after having seen and accepted the first argument in the proof of  Theorem \ref{T:negativecontinuity}. Brouwer's Continuity Principle   proves: $\neg \mathbf{WLPO}$, see Theorem \ref{T:notlpo}.

\section{The Continuity Principle}\label{S:bcp}  Let $f$ be a (`full') function from $\mathcal{R}$ to $\mathcal{R}$. The statement that there is no point  where $f$ is positively discontinuous is, constructively, a weak statement and not a very useful one.    The positive result that $f$ is continuous at every point is obtained from an axiomatic assumption first used by Brouwer.

\subsection{Understanding Brouwer's Continuity Principle}\label{SS:bcp}

We first have to agree upon some notation. 

In the following definition, we use $k, n_0, n_1, \ldots, s, t, u,\ldots$ as variables over the set $\mathbb{N}$.

\begin{definition}[Coding finite sequences of natural numbers by natural numbers]\label{D:coding} Let $p:\mathbb{N}\rightarrow\mathbb{N}$ be the function enumerating the primes: $p(0)=2, p(1)=3, p(2) = 5, \ldots$.

For each $k$, for all $n_0, n_1, \ldots,n_{k-1}$, $\langle n_0, n_1, \ldots, n_{k-1}\rangle := p(k-1)\cdot \prod_{i<k}p(i)^{n_i}-1$.

We also define: $\langle\;\rangle :=0.$

For all $s$, for each  $k$, for all $n_0, n_1, \ldots,n_{k-1}$, if $s=\langle n_0, n_1, \ldots, n_{k-1}\rangle$, then $length(s)=k$ and, for all $i<length(s)$, $s(i):=n_i$.  

\smallskip For all $k$, $\omega^k:=\{s\mid length(s) =k\}$. 

\smallskip For all $k,l$, for all $s$ in $\omega^k$, for all $t$ in $\omega^l$, $s\ast t$ is the element of $\omega^{k+l}$ satisfying $\forall i<k[s\ast t(i) =s(i)]$ and $\forall j<l[s\ast t(k+j)=t(j)]$. 

\smallskip For all $s,t$, $s\sqsubseteq t\leftrightarrow \exists u[s\ast u =t]$ and $s\sqsubset t \leftrightarrow (s\sqsubseteq t \;\wedge\; s\neq t)$. 

\smallskip For all $s,t$, $s\perp t\leftrightarrow \neg(s\sqsubseteq t \;\vee\; t\sqsubseteq s)$.

\end{definition}\begin{definition}
$\mathcal{N}:=\omega^\omega$ is the set of all infinite sequences $\alpha=\alpha(0), \alpha(1), \alpha(2),\ldots$ of non-negative integers. We use $\alpha, \beta, \ldots \varphi, \psi, \ldots$ as variables over $\mathcal{N}$.

$\mathcal{N}$ is sometimes called \emph{Baire space}.  

\smallskip
For each $\alpha$, for each $m$,  $\overline \alpha m :=\langle \alpha(0), \alpha(1), \ldots, \alpha(m-1)\rangle.$

For all $s$, for all $\alpha$, $s\sqsubset \alpha \leftrightarrow \exists n[\overline \alpha n = s]$ and $s\perp \alpha\leftrightarrow \neg(s\sqsubset \alpha)$.

For each $n$, we let $\underline n$ be the infinite sequence with the constant value $n$, that is,  for all $i$, $\underline n (i) = n$. 

\end{definition}
\begin{axiom}[Brouwer's Continuity Principle, $\mathbf{BCP}$]\label{A:bcp}

 For every  $R\subseteq\mathcal{N} \times \mathbb{N}$,

if $\forall \alpha \in \mathcal{N}\exists n[\alpha Rn]$, then $\forall \alpha\in\mathcal{N}\exists m \exists n \forall \beta \in \mathcal{N}[\overline \alpha m \sqsubset \beta  \rightarrow \beta R n]$.

\end{axiom}

How may we convince each other that thia axiom is `reasonable' and that we should accept it as a rule for our common mathematical game?

\smallskip
We should ask what we mean by the `set' $\mathcal{N}$ of all infinite sequences of natural numbers. This `set' is a kind of framework in which all kinds of infinite sequences will grow, although only a very few of them have been realized up to now.

 Cantor's suggestion  that a set is the result of taking together its (earlier constructed?) elements is a notion that does not make sense to us.

An infinite sequence $\alpha$ may be given  by a \textit{finite description} or an \textit{algorithm} that enables us to find, one by one,  the successive values $\alpha(0), \alpha(1), \ldots$ of $\alpha$.    Brouwer calls such infinite sequences \textit{lawlike sequences}.

As we saw in the first argument in the proof of Theorem \ref{T:negativecontinuity}, there are also infinite sequences $\alpha$  that are not governed by a law.  The successive values $\alpha(0), \alpha(1), \alpha(2), \ldots$ of $\alpha$ then are  \textit{freely chosen, step by step}.   At any given moment, only a finite initial part $\bigl(\alpha(0), \alpha(1), \ldots , \alpha(n-1)\bigr)$ of  $\alpha$ is known. One may ask for further values, but, having received the answers, one still only knows a finite initial part of $\alpha$, albeit  a longer one.  

There are intermediate possibilities. Having begun with choosing freely, step by step, $\alpha(0), \alpha(1), \ldots \alpha(n-1)$ one may  decide,   to `freeze' $\alpha$ and to determine  the remaining values by means of some rule or algorithm. 

 There is even a possibility  I make $\alpha$ \textit{dependent on my future mathematical experience}, by saying, for instance, $\alpha(n) =0$ if at the moments $0,1, \ldots, n$, I did not find a proof of the $abc$-conjecture, and $\alpha(n) = 1$, if I did.

If we think of the set $\mathcal{N}$ as a framework making room for   lawlike sequences as well as for free-choice-sequences and sequences that are of some intermediate kind, we are thinking of, what Brouwer calls, \textit{the full continuum}. 

Now assume $R\subseteq \mathcal{N}\times \mathbb{N}$ is given such that $\forall \alpha \exists n[\alpha Rn]$. \textit{Every} infinite sequence $\alpha$, \textit{also a lawlike sequence $\alpha$},  may be thought of as being produced, step by step, by a \textit{black box}. At the moment we  produce a number  $n$ such that $\alpha Rn$,   only finitely many values of $\alpha$, say $\alpha(0), \alpha(1), \ldots, \alpha(m-1)$, will have become known. Clearly,  for any infinite sequence such that $\beta(0) =\alpha(0)$, $\beta(1) = \alpha(1), \ldots$ and $\beta(m-1) = \alpha(m-1)$, one may conclude: $\beta Rn$. 

 \smallskip Adopting Axiom \ref{A:bcp} as a starting point for our further mathematical discourse, we lay down a \textit{canonical meaning} for expressions of the form `$\forall \alpha\exists n[\alpha Rn]$'.  Earlier, we did so for expressions of the form \textit{`infinite', `uncountable'} and \textit{`or'}.

\begin{theorem}\label{T:notlpo} $\mathbf{BCP} \Rightarrow \neg\mathbf{WLPO}$. \end{theorem}\begin{proof} Assume $\mathbf{WLPO}$ i.e.  $\forall \alpha[\neg\forall n[n<k_\alpha]\;\vee\;\forall n[n<k_\alpha]]$, i.e.  \\$\forall \alpha[\neg\forall n[\alpha(n)=0]\;\vee\;\forall n[\alpha(n)=0]]$.
\\Consider $\alpha = \underline 0$ and,  
applying $\mathbf{BCP}$, find $m$ such that\\ \textit{either}  $\forall \beta[\overline{\underline 0}m \sqsubset\beta\rightarrow \neg\forall n[\beta(n)= 0]]$ \textit{or}  $\forall \beta[ \overline{\underline 0}m \sqsubset \beta \rightarrow \forall n[\beta(n)= 0]]$. 
\\The first of these two statements   is not true, consider: $\beta = \underline 0$, and the second one also fails, consider: $\beta = \underline{\overline 0}m\ast\underline 1$. 

We thus  obtain a contradiction.\end{proof}

\subsection{The (pointwise) continuity of functions from $\mathcal{R}$ to $\mathcal{R}$}$\;$

In the next Definition, we introduce a function  $\alpha\mapsto u_\alpha$ associating to every $\alpha$ in $\mathcal{N}$ a real number $u_\alpha$. 

\begin{definition} We let $(r_0', r_0''), \;(r_1', r_1''), \;(r_2', r_2''), \;\ldots$ be a fixed enumeration of all pairs of rationals.

We define a mapping associating to each $s\neq 0$ a pair $(u_s', u_s'')$ of rationals, as follows. 

For each $n$, \emph{if} $0< r_n''-r_n' \le 1$, then $(u'_{\langle n \rangle}, u''_{\langle n \rangle})= (r_n', r_n'')$, and, \\\emph{if not}, then $(u'_{\langle n \rangle}, u''_{\langle n \rangle})= (0,1)$. 

For each $s\neq 0$, for each $n$, \emph{if} $u'_s<r'_n<r''_n<u''_s$ and\\ $0< r_n''-r_n' \le \frac{1}{2} (u''_s-u'_s)$, then $(u'_{s\ast\langle n \rangle}, u''_{s\ast\langle n \rangle})=(r_n', r_n'')$, and, \\\emph{if not}, then  $(u'_{s\ast\langle n \rangle}, u''_{s\ast\langle n \rangle})=(\frac{2}{3}u'_s + \frac{1}{3}u''_s,\frac{1}{3}u'_s + \frac{2}{3}u''_s)$.  

\smallskip For each  $\alpha$ in $  \mathcal{N}$ we let $u_\alpha$ be the  infinite sequence  of pairs of rationals such that, \\for all $n$, $u_\alpha(n)=(u_{\overline \alpha(n+1)}', u_{\overline \alpha(n+1)}'')$.  \end{definition}

\begin{theorem}\footnote{See \cite{veldman1982}, \cite[\S 2]{veldman2001} and \cite[\S 4]{veldman20b}.}\label{T:realfunctionscontinuous}\begin{enumerate}[\upshape (i)] \item For every $\alpha$, $u_\alpha$ is a real number. \item For every real number $x$, there exists $\alpha$ such that $x=_\mathcal{R}u_\alpha$.
\item $\mathbf{BCP}\Rightarrow$ Every  function from $\mathcal{R}$ to $\mathcal{R}$ is continuous at every point $x$.  \end{enumerate} \end{theorem}
\begin{proof} The proof of (i) and (ii) is left to the reader.

\smallskip (iii) Let $f$ be a real function and let a real $x$ be given. We shall prove that $f$ is continuous at $x$.

Let $p$ be given. We have to find $m$ such that, for every real $y$, if $|y-x| < \frac{1}{2^m}$, then  $ |f(y)-f(x)|<\frac{1}{2^p}$. 

Note: $\forall \alpha \exists n[r'_n < f(u_\alpha) < r_n'']$ and 
find $\alpha$ such that $x=_\mathcal{R} u_\alpha$.

Applying $\mathbf{BCP}$, find $q, n$ such that $\forall \beta[\overline \alpha q \sqsubset \beta \rightarrow r'_n <  f(u_\beta)<r_n'']$. 

Find $i,j$ such that $u_\alpha(n-1) = (r_i', r_i'')$ and $u_\alpha(n)= (r_j',r_j'')$.

 Note: $r_i'<_\mathbb{Q} r_j'<_\mathbb{Q} r_j''<_\mathbb{Q} r_i''$. 

Find $m$ such that $\frac{1}{2^m}<\min_\mathbb{Q}(r_j'-_\mathbb{Q}r_i', r_i''-_\mathbb{Q}r_j'')$.

 Note: for every real $y$, if $|y-x|<\frac{1}{2^m}$, then $r_i'<y<r_i''$ and there exists $\beta$ such that $\overline \alpha n \sqsubset \beta$ and $y=_\mathcal{R} u_\beta$ and, therefore: $r_n'< f(u_\beta)<r_n''$ and $r_n' < f(y) < r_n''$.
 
 Conclude: for every real $y$, if $|y-x|<\frac{1}{2^m}$ then $|f(y)-f(x)| < r_n''-r_n' <\frac{1}{2^p}$.    \end{proof}

 \section{The Fan Theorem}\label{S:fantheorem}

 \subsection{Proving the Fan Theorem}
 
 The intuitionistic mathematician holds not only that every function from $[0,1]$ tot $\mathcal{R}$, like every function from $\mathcal{R}$ to $\mathcal{R}$, see Theorem \ref{T:realfunctionscontinuous}(iii) is  continuous at every point, but also that every function from $[0,1]$ to $\mathcal{R}$ is {\it continuous uniformly on $[0,1]$}. In order to obtain this result, she first proves the  {\it Fan Theorem}. 
 
 We need some definitions in order to formulate the Fan Theorem.
 \begin{definition} Let $\mathcal{X}\subseteq\mathcal{N}$ and  $B\subseteq\mathbb{N}$. 
 
 $B$ \emph{is a bar in} $\mathcal{X}$, notation: $Bar_\mathcal{X}(B)$,  if and only if: $\forall \alpha \in \mathcal{X}\exists s \in B[s\sqsubset \alpha]$.\smallskip
 
 $\mathcal{C}:=2^\omega:=\{\alpha\mid\forall n[\alpha(n)<2]\}$.
 
 $\mathcal{C}$ is  sometimes called \emph{Cantor space}.
 
 \smallskip $Bin:=2^{<\omega}:=\{s\mid \forall i<length(s)[s(i)<2]\}$. 
 
 \smallskip For each $n$, $Bin_n:=\{s \in Bin\mid length(s)=n\}=\{s\in \omega^n\mid\forall j<n[s(j)<2]\}$. \end{definition}
\begin{definition}
For each  $\mathcal{X}\subseteq\mathcal{N}$, for each $s$, we define:
$\mathcal{X}\cap s:=\{\alpha \in \mathcal{X}|s \sqsubset \alpha\}.$ \end{definition}

The Fan Theorem is the statement that every bar in $\mathcal{C}$ has a finite subbar.  Brouwer actually proved the more general but equivalent statement that  every bar in a so-called \textit{fan} has a finite subbar. Here, $B\subseteq \mathbb{N}$ is called {\it finite}  if and only if $\exists u\forall s[s \in B\leftrightarrow\exists i<length(u)[s=u(i)]]$. 

 We avoid defining the  notion of a {\it fan}. For understanding the name of the Theorem, it suffices to know every fan is a subset of $\mathcal{N}$ and that $\mathcal{C}$ is the prime example of a fan. 
 
 The proof of the Fan Theorem turns on   a {\it philosophical claim}.

 \begin{theorem}[Brouwer's Fan Theorem]\label{T:bft} For every $B\subseteq \mathbb{N}$, if $Bar_\mathcal{C}(B)$, then
 there exists  $u$ such that $\forall i <length(u)[u(i) \in B]$ and  $\forall \alpha \in \mathcal{C}\exists i<length(u)[  u(i)\sqsubset \alpha].$ \end{theorem}
 
 \begin{proof}\footnote{In \cite{brouwer27} and \cite{brouwer54}, the Fan Theorem is  obtained as a corollary of the `Bar Theorem', in this paper Theorem \ref{T:bimbid}. We give an argument for the Fan Theorem inspired by Brouwer's proof of the Bar Theorem, like A. Heyting did in \cite[\S 3.4.2]{heyting56}.}
 Let $B\subseteq \mathbb{N}$  be given such that $Bar_\mathcal{C}(B)$.

 Define, for every $s$ in $2^{<\omega}$,
  $s$ is {\it $B$-safe}, or simply: {\it safe}, if and only if   $Bar_{\mathcal{C}\cap s}(B)$. 
 
 As $Bar_\mathcal{C}(B)$,  $\langle \;\rangle$ is  safe.
 
We now ask the probably somewhat surprising question: `{\it How is it possible we know this?}' and, 
after some reflection, we come forward with the \textit{claim}: there  must be a \textit{canonical proof} of the fact: 
 `$\langle \;\rangle$ is  safe.' 
 
 \smallskip What do we mean by: `a canonical proof'?
  
  The canonical proof is an arrangement of statements of the form: `{\it $s$ is safe}'. 
 The \textit{conclusion} of the canonical proof is the statement:
 `$\langle \;\rangle$ is  safe.'
 
 The \textit{starting points}  of the canonical proof have the form:
 
     \begin{center} $s\in 2^{<\omega}$ and $s \in B$.
     
     \textit{Therefore:} $s$ is safe.
     
     \end{center}
     
     There are only two kinds of reasoning steps: \textit{forward reasoning steps} and \textit{backward reasoning steps}.
     
     \smallskip
     A \textit{forward reasoning step} has the form:
     
     \begin{center} 
     
     $s\in 2^{<\omega}$, and 
      $s\ast\langle 0 \rangle$ is safe,
     and $s\ast\langle 1 \rangle$ is safe.

     \smallskip
     \textit{Therefore:}  $s$ is safe. 
     \end{center}
     
     \smallskip A \textit{backward reasoning step} either has the form:
     
       \begin{center}
       
       $s\in 2^{<\omega}$ and $s $ is safe.
       
       \smallskip
       \textit{Therefore:} $s\ast\langle 0 \rangle$ is safe.
       \end{center}
   
   or it has the form      
      \begin{center}
       
       $s\in 2^{<\omega}$ and $s $ is safe.
       
       \smallskip
       \textit{Therefore:} $s\ast\langle 1 \rangle$ is safe.
       \end{center}

       Clearly, the canonical proof may be visualized as a finite tree, with the statement `$\langle \;\rangle$ is safe' at its bottom node and statements `$s\in B$' at its top nodes. Each node that is not a top node has either one or two upstairs neighbours.
       
     One easily sees that all resoning steps are sound and that a canonical proof of `$\langle\;\rangle$ is safe' is indeed a proof of `$\langle\;\rangle$ is safe'.  
     
     \smallskip Our claim gives expression to the feeling that, if we know: `$Bar_\mathcal{C}(B)$', this knowledge must be based upon an orderly organization of elementary pieces of information of the form: `$s\in B$'.

       \medskip Now, trusting our claim, let us take a canonical proof of: `$\langle \; \rangle$ is safe.' We are going to use it in order to build another proof.

       Define, for every $s$ in $2^{<\omega}$,  
 $s$ is {\it supersafe} if and only if there exists $u$ such that $\forall i<\mathit{length}(u)[u(i) \in B]$ and  $\forall \alpha \in \mathcal{C} \cap s\exists i<length(u) [ u(i)\sqsubset\alpha ]$.  
  
  Note that the conclusion we are after is: `$\langle \;\rangle$ is supersafe.'
  
  \smallskip

  Replace,  in the canonical proof of `$\langle\;\rangle$ is safe',   every statement  `$s$ is safe' by the statement `$s$ is supersafe'. 
  
  The result will be   another \textit{valid proof}.
   
   Why? 
   
   \smallskip First,  the new starting points are sound:  if $s\in 2^{<\omega}$ and $s\in B$, define: $u:=\langle s\rangle$ and note: $u(0)\in B$ and  $\forall \alpha\in \mathcal{C}\cap s[ u(0)\sqsubset\alpha]$.
   
   \smallskip
   Then, the new forward reasoning steps are sound.

   For assume:  $s\in 2^{<\omega}$ and both  $s\ast\langle 0 \rangle$ and    $s\ast\langle 1 \rangle$ are supersafe.
      
      Find $u_0, u_1$ such that: for each $ j <2$,  $\forall i\le \mathit{length}(u_j)[u_j(i)\in B]$ and \\$ \forall \alpha \in \mathcal{C}\cap s\ast\langle j\rangle \exists i<\mathit{length}(u_j)[u_j(i)\sqsubset \alpha ].$
      
      Define $u:=u_0\ast u_1 $ and note 
      \\$\forall i <\mathit{length}(u)[u(i) \in B]$ and  $\forall \alpha \in \mathcal{F}\cap s \exists i<\mathit{length}(u)[u(i) \sqsubset \alpha ],$\\
      that is: $s$ is supersafe.
      
      \smallskip
      Also, the new backward reasoning steps are sound.
      
      For assume: $s\in 2^{<\omega}$ and $s$ is supersafe.
      
      Find $u$ such that 
      
      $\forall i <\mathit{length}(u)[u(i) \in B]$ and $\forall \alpha \in \mathcal{C}\cap s \exists i<\mathit{length}(u)[u(i) \sqsubset \alpha ].$
      
      Note: for both $j<2$, $\mathcal{C} \cap s\ast\langle j\rangle \subseteq \mathcal{C}\cap s$ and  $s\ast\langle j \rangle$ is  supersafe.
       
       \smallskip
       Conclude: our new proof is sound and its  conclusion is true, i.e.
       $\langle \;\rangle$   is supersafe.
       \end{proof}
               
     Generalizing the proof of Theorem \ref{T:bft} we obtain the following  conclusion:   \begin{theorem}[Bar Induction in $\mathcal{C}$]\label{T:fanbi}Let  $B\subseteq\mathbb{N}$ be a bar in $\mathcal{C}$. 
        \begin{enumerate}[\upshape (i)]\item Assume  $B\subseteq C\subseteq 2^{<\omega}$  and, for all $s$ in $2^{<\omega}$, \\ $s\in C$ if and only if  $\forall j<2[s\ast\langle j\rangle\in C]$.  Then $\langle \;\rangle \in C$.\item Assume  $B\subseteq C\subseteq 2^{<\omega}$  and, for all $s$ in $2^{<\omega}$,   \\if  $\forall j<2[s\ast\langle j\rangle\in C]$, then $s\in C$.   Then $\langle \;\rangle \in C$. \end{enumerate} \end{theorem}
 
  \begin{proof} (i) Define, for each $s$ in $2^{<\omega}$, $s$ is \textit{safe} if and only if  $Bar_{\mathcal{C}\cap s}(B)$. In the canonical proof of: `$\langle\;\rangle$ is safe', replace every statement `$s$ is safe' by `$s \in C$'. The result will be a proof of: `$\langle \;\rangle \in C$'. 
  
  \smallskip (ii) Define $C^\ast:=\{s \in  2^{<\omega}[\exists t\sqsubseteq  s[t\in C]\}$.

  Note: for all $s$ in $2^{<\omega}$, if $s\in C^\ast$, then $\forall j<2[s\ast\langle j\rangle \in C^\ast]$. 
  
  Now assume: $s\in 2^{<\omega}$ and $\forall j<2[s\ast\langle j\rangle \in C^\ast]$.  We may distinguish  two cases. 
  
  \textit{Case 1.} $\forall j<2[ s\ast\langle j \rangle \in C]$. Then $s \in C\subseteq C^\ast$. 
  
  \textit{Case 2.} $\exists t\sqsubseteq s[t\in C]$. Then $s\in C^\ast$. 
  
  We thus see: for all $s$ in $2^{<\omega}$,  $s\in C^\ast$ if and only if  $\forall j<2[\sigma(s\ast\langle j\rangle\in C^\ast]$.
  
  Applying (i), we conclude: $\langle \;\rangle \in C^\ast$, and: $\langle\;\rangle \in C$. 
  \end{proof}
  
  Theorem \ref{T:fanbi}(ii) shows that, in the canonical proof of $Bar_\mathcal{C}(B)$, introduced in the proof of Theorem \ref{T:bft}, the backward steps might have been left out.  They may not be left out in the canonical proof of the {\it Bar Theorem}, the infinitary analog of the Fan Theorem, see Subsection \ref{SS:backward}.    
  
  \subsection{Every function from $[0,1]$ to $\mathcal{R}$ is uniformly continuous}$\;$      
     Brouwer's Continuity Principle implies that every function from $[0,1]$ to $\mathcal{R}$, like every function from $\mathcal{R}$ to $\mathcal{R}$, see  Theorem \ref{T:realfunctionscontinuous}(iii), is  continuous at every point.  Using the Fan Theorem, we now prove that every function from $[0,1]$ to $\mathcal{R}$ is even { continuous {\it uniformly on $[0,1]$}.  This is historically the first application of the Fan Theorem, see  \cite{brouwer27}.  
        
 \smallskip
 In the next Definition, we introduce a function  $\alpha\mapsto c_\alpha$ associating to every $\alpha$ in $\mathcal{C}$ a real number $c_\alpha$ in $[0,1]$.     
  \begin{definition}  
  
  We define a mapping  associating to every $s$ in $2^{<\omega}$ a pair $(a_s, b_s)$ of rationals such that $(a_{\langle\;\rangle}, b_{\langle\;\rangle})=(0,1)$ and for each $s$ in $2^{<\omega}$, \\$(a_{s\ast\langle 0\rangle}, b_{s\ast\langle 0\rangle}) =(a_s, \frac{1}{3}a_s + \frac{2}{3}b_s)$ and $(a_{s\ast\langle 1\rangle}, b_{s\ast\langle 1\rangle}) =( \frac{2}{3}a_s + \frac{1}{3}b_s, b_s)$.
 
 \smallskip
 
 For each $\alpha$ in $\mathcal{C}$ we let $c_\alpha$ be the real $x$ such that, for each $n$, $x(n)=(a_{\overline \alpha n}, b_{\overline \alpha n})$. 
 \end{definition}    
 \begin{theorem}[Uniform-Continuity Theorem]\label{T:uc} Every  function $f$ from $[0,1]$ to $\mathcal{R}$  is  continuous \emph{uniformly on $[0,1]$}, i.e.
       $$\forall p \in \mathbb{N}\exists m\in \mathbb{N}\forall x \in [0,1]  \forall y\in[0,1][|x -y| < \frac{1}{2^m} \rightarrow |f(x) -f(y| <\frac{1}{2^p}].$$ \end{theorem}
       
       \begin{proof} Let $f$ be a function from $[0,1]$ to $\mathcal{R}$ and let $p$  be given. We want to find $m$ such that 
 $\forall x \in [0,1]  \forall y\in[0,1][|x -y| < \frac{1}{2^m} \rightarrow |f(x) -f(y| <\frac{1}{2^p}].$ 
 
  \smallskip     
Let  $B$ be the set of all $s$ in  $2^{<\omega}$ such that, for all $x,y$ in $[a_s, b_s]$, $|f(x)-f(y)|<\frac{1}{2^p}$. 
       
        We first prove that $B$ is a bar in $\mathcal{C}$, i.e.  $Bar_\mathcal{C}(B)$. 
        
       Let $\alpha$ in $\mathcal{C}$ be given. As $f$ is continuous at $c_\alpha$, see Theorem \ref{T:realfunctionscontinuous}(iii), find $m$ such that $\forall y\in[0,1][|c_\alpha -y| < \frac{1}{2^m} \rightarrow |f(c_\alpha) -f(y| <\frac{1}{2^{p+1}}]$. Find  $n$ such that $b_{\overline \alpha n}-a_{\overline \alpha n}<\frac{1}{2^m}$ and note: $\overline \alpha n \in B$. 
       
       We thus see that every $\alpha$ in $\mathcal{C}$ has an initial part in $B$. 
       
       \smallskip
      
        Now let  $C$ be the set of all $s$ in  $2^{<\omega}$ such that, for some $m$, \\for all $x,y$ in $[a_s, b_s]$, if $|x-y|<\frac{1}{2^m}$, then $|f(x)-f(y)|<\frac{1}{2^p}$.

        \smallskip
        Note: $B\subseteq C$.
        
        \smallskip

         Let $s$ in $2^{<\omega}$ be given such that $\forall i<2[s\ast\langle i \rangle \in C]$. We are going to prove: $s\in C$. 
         
           First find $m_0$ such that, for all $x,y$ in $(a_{s\ast\langle 0\rangle}, b_{s\ast\langle 0\rangle}) =(a_s, \frac{1}{3}a_s + \frac{2}{3}b_s)$, if $|x-y|<\frac{1}{2^{m_0}}$, then $|f(x)-f(y)|<\frac{1}{2^p}$. Then find $m_1$ such that, for all $x, y$ in $(a_{s\ast\langle 1\rangle}, b_{s\ast\langle 1\rangle}) =( \frac{2}{3}a_s + \frac{1}{3}b_s, b_s)$, if $|x-y|<\frac{1}{2^{m_1}}$, then $|f(x)-f(y)|<\frac{1}{2^p}$.
           
           Now find $n:=length(s)$. 
           
           Note: $b_s-a_s=(\frac{2}{3})^n$, and, for all $x, y$ in $[a_s, b_s]$, if $|x -y|<\frac{1}{3}(\frac{2}{3})^n$, then either $x, y$ both are in $ [a_s, \frac{1}{3}a_s + \frac{2}{3}b_s]$ or $x, y$ both are in $[ \frac{2}{3}a_s + \frac{1}{3}b_s, b_s]$.

        Find $m$ such that $m\ge m_0$ and $m\ge m_1$ and $\frac{1}{2^m} < \frac{1}{3}(\frac{2}{3})^n$ and note: for all $x, y$ in $[a_s, b_s]$, if $|x-y|<\frac{1}{2^{m}}$, then $|f(x)-f(y)|<\frac{1}{2^p}$. We thus see: $s\in C$. 
        
        Clearly, for all $s$ in $Bin$, if $\forall i<2[s\ast\langle i \rangle \in C]$, then $s\in C$.

        \smallskip
       Using Theorem \ref{T:fanbi}(ii),  conclude: $\langle\;\rangle \in C$, that is:

           $\exists m \forall x \in [0,1]\forall y \in [0,1][ |
        x -y|< \frac{1}{2^m} \rightarrow |f(x) - f(y)| < \frac{1}{2^n}].$
          \end{proof}
 
 \subsection{The failure of the Fan Theorem in computable analysis}\label{SS:failurefan}         
          In the next Theorem, we use some notations introduced in  Subsection \ref{SS;haltingproblem}.
        \begin{theorem}[Kleene's Alternative\footnote{Kleene discovered this theorem in 1950, see \cite{kleene52} and \cite[Lemma 9.8]{kleenevesley65}.}]\label{T:ka} There exists $B^\ast\subseteq\mathbb{N}$ such that \begin{enumerate}[\upshape (i)]\item Every \emph{computable} $\alpha$ in $\mathcal{C}$ has an initial part in $B$, i.e.  $\exists s\in B^\ast[s\sqsubset\alpha]$. \item Every finite subset of $B$ positively fails to be a bar in the set \\$\{\alpha \in \mathcal{C}\mid \alpha \; is\; computable\}$, i.e. \\for all $u,p$, if $length(u)=p>0$ and $\forall i<p[u(i)\in B^\ast]$, \\then there exists a computable $\alpha$ in $\mathcal{C}$ such that $\forall i<p[u(i)\perp \alpha]$. \item $B^\ast$ is an algorithmically decidable subset of $\mathbb{N}$. \end{enumerate}\end{theorem}
 \begin{proof} First, let $B$ be the set of all $s$ in $Bin$  such that, for some $e$, $length(s)=e+1$ and $\forall i\le e\exists z[T(e,i,z) \;\wedge\;U(z) =s(i)].$ 
 
 $ \langle\;\rangle \notin B$ and,  for every $e$, there is at most one $s$ in $B$ such that $length(s)=e+1$.

Let $\alpha$ be a computable element of $\mathcal{C}$.  Find $e$ such that $\alpha = \varphi_e$. \\Note:  $\overline \alpha (e+1)=\overline{\varphi_e}(e+1)\in B$. 
 
 \smallskip
 
 Let $u,p$ be given such that $length(u)=p>0$ and $\forall n<p [u(n) \in B]$. \\We may assume: $0<length\bigl(u(0)\bigr)< length\bigl(u(1)\bigr)<\ldots<length\bigl(u(p-1)\bigr)$ and, therefore, $\forall n<p[\mathit{length}\bigl(u(n)\bigr) >n]$. Define $\alpha$ such that, for each  $n<p$,  $\alpha(n) = 1-\bigl(u(n)\bigr)(n)$, and, for each $n\ge p$,  $\alpha(n) =0$.  Note: $\alpha$ is computable and $\forall n<p[u(n)\perp \alpha]$.
 
  \medskip
 Unfortunately, $B$ is not an algorithmically decidable subset of $\mathbb{N}$.
 
 We therefore introduce another subset of $\mathbb{N}$.
 
   Let $B^\ast$ be the set of all $s$ in  $Bin$ such that, for some $e< \mathit{length}(s)$, \\$\forall i\le e\exists z<\mathit{length}(s)[T(e,i,z) \;\wedge\;U(j) =s(i)]$, and, therefore, $\overline s (e+1) \in B$. 
  
    Note that $B^\ast$ is an algorithmically decidable subset of $\mathbb{N}$.
   
 Let $\alpha$ be a computable element of $\mathcal{C}$. Find $e$ such that $\alpha=\varphi_e$. Find $n>e$ such that    $\forall i\le e\exists z<n[T(e,i,z) \;\wedge\;U(z) =s(i)]$ and conclude: $\overline \alpha n \in B^\ast$. 
   
  We thus see that every computable $\alpha$ in $\mathcal{C}$ has an initial part in $B^\ast$. 
   
   Let $u,p$ be given such that $length(u)=p>0$ and $\forall n<p [u(n) \in B^\ast]$. Find $v$ such that $length(v)=p$ and $\forall n<p[v(n)\sqsubseteq u(n)\;\wedge\; v(n)\in B]$. Find a computable $\alpha$ in $\mathcal{C}$ such that $\forall n<p[v(n)\perp\alpha]$ and conclude $\forall n<p[u(n)\perp\alpha]$. 
   
   We thus see that every finite subset of $B^\ast$ positively fails to be a bar in \\$\{\alpha \in \mathcal{C}\mid \alpha\;is\;computable\}$. 
     \end{proof}
        
The subject of \textit{computable analysis} starts from the assumption that every $\alpha$ in $\mathcal{N}$ is given by an algorithm, i.e., if one uses \textit{Church's Thesis}, by a Turing-algorithm. One may study this subject from an intuitionistic point of view. As appears from Theorem \ref{T:ka}, computable analysis   is dramatically different from intuitionistic analysis. For instance, there exists a real function from $[0,1]$ to $\mathcal{R}$ that is everywhere continuous but positively unbounded. Many more such results may be found in \cite{veldman2011b}.  
   \section{The Determinacy Theorem as an equivalent of the Fan Theorem}
The Uniform-Continuity Theorem, Theorem  \ref{T:uc}, is the first application of the Fan Theorem, and may be called the goal for which it was devised. In this Section we want to introduce to the reader to a second and more recent application, an application in the theory of games.

 In the first two Subsections we consider two kinds of two-move-games for Players $I,II$ and define when such games are {\it constructively determinate from the viewpoint of Player $I$}. Games from the first Subsection may fail to have this property but we will see that the  Fan Theorem proves that the games from the second Subsection always have it. In the third Subsection we prove that, conversely, this result implies an important case of the Fan Theorem.   In the fourth Subsection we introduce the intuitionistic Determinacy Theorem. \subsection{Games in $2\times \omega$}\label{SS:nondet}
 
 For every $C\subseteq 2\times \omega$, we introduce a game $\mathcal{G}_{2\times \omega}(C)$.  \\There are two players, Player $I$ and Player $II$. Every play  goes a follows. Player $I$ chooses $i$ in $\{0,1\}$, thereafter Player $II$ chooses $n$ in $\omega$ and the play is finished. Player $I$ wins the play if and only if $\langle i, n\rangle \in C$. Player $II$ tries to prevent that Player $I$ wins the play. 
 \\Clearly, Player $I$ has a \textit{winning first move} if and only if $\exists i\forall n[\langle i,n\rangle \in C]$. A \textit{strategy} for Player $II$, on the other hand,  is a pair of numbers $\langle p_0, p_1\rangle$, and the strategy is \textit{winning} if $\forall  i<2[\langle i, p_i\rangle \notin C]$. 
 
 The game $\mathcal{G}_{2\times\omega}(C)$ is called  \textit{determinate} if and only if \textit{either} Player $I$ has a winning first move \textit{or} Player $II$ has a winning strategy. As one may conjecture, it may happen that this decision can not be taken in a constructive way. We will see this in a moment.

 We  now define: the game $\mathcal{G}_{2\times\omega}(C)$ is \textit{determinate from the viewpoint of Player $I$} if and only if:\begin{quote}  \textit{If Player $II$ does not have a winning strategy in the constructively strong sense: $\forall p_0\forall p_1[\langle 0, p_0\rangle \in C\;\vee\;\langle 1, p_1\rangle \in C]$, then  Player $I$ has a winning first move, i.e. $\exists i<2\forall n[\langle i, n\rangle \in C]$}\end{quote} and ask ourselves the question: is the game $\mathcal{G}_{\omega\times 2}(C)$ always determinate from the viewpoint of Player $I$?
 
 \smallskip The following example shows that, sometimes, it is not.\\ Define $C:=\{\langle i, n\rangle \mid k_{99}\le n\rightarrow \exists l[k_{99}=2l+i]\}$. \\Let $p_0, p_1$ be given. If $k_{99}\le p_0$, one may find $i<2$ and $l$  such that $k_{99}=2l+i$ and $\langle i, p_i\rangle \in C$. If $ p_0< k_{99}$, then $\langle 0, p_0\rangle\in C$. We thus see  Player $I$ has an answer to any given strategy of Player $II$. \\ On the other hand, if Player $I$ has a winning first move, then $\exists i<2\forall n[\langle i, n\rangle \in C]$ and $\exists i<2[\exists j [j=k_{99}]\rightarrow \exists n[2n+i=k_{99}]]$, i.e. $\neg \exists n[2n=k_{99}]\;\vee\;\neg\exists n[2n+1=k_{99}]$, a \textit{reckless} or \textit{hardy} statement. 
 
 Note that the statement:  `the game $\mathcal{G}_{2\times\omega}(C)$ is determinate' is also reckless, as, for every $C\subseteq 2\times\omega$, if $\mathcal{G}_{2\times \omega}(C)$ is determinate, then $\mathcal{G}_{2\times \omega}(C)$ is determinate from the viewpoint of Player $II$. 
 
 \subsection{Games in $\omega\times 2$}\label{SS:gamesbyfan}\footnote{See \cite[\S 4]{veldman82}.} For every $C\subseteq\omega\times 2$, we introduce a game $\mathcal{G}_{ \omega\times 2}(C)$.  \\There are again two players, Player $I$ and Player $II$. Every play  goes a follows. Player $I$ chooses $n$ in $\omega$, thereafter Player $II$ chooses $i$ in $ \{0,1\}$ and the play is finished. Player $I$ wins the play if and only if $\langle n,i\rangle \in C$. Player $II$ tries to prevent that Player $I$ wins the play. 
 Clearly, Player $I$ has a \textit{winning first move} if and only if $\exists n\forall i<2[\langle n,i\rangle \in C]$. A \textit{strategy} for Player $II$ is an element $\tau$ of $\mathcal{C}$ and  the strategy $\tau$ is \textit{winning} if and only if $\forall n[\langle n, \tau(n)\rangle \notin C]$. 
 
 The game $\mathcal{G}_{\omega\times 2}(C)$ is called  \textit{determinate} if and only if \textit{either} Player $I$ has a winning first move \textit{or} Player $II$ has a winning strategy.
 
 Considering the example $C:=\{\langle n, i\rangle \in \omega \times 2\mid n=k_{99}]$, we see that it may happen that the statement: `$\mathcal{G}_{\omega\times 2}(C)$ is determinate' is a reckless one. 
 
  We now define: the game $\mathcal{G}_{\omega\times 2}(C)$ is \textit{determinate from the viewpoint of Player $I$} if and only if \begin{quote} \textit{If Player $II$ does not have a winning strategy in the strong sense: $\forall \tau\in\mathcal{C}\exists n[\langle n, \tau(n)\rangle \in C]$, then  Player $I$ has a winning first move, i.e. $\exists n\forall i<2[\langle n, i\rangle \in C]$.}\end{quote}
  
  The Fan Theorem proves that, for  every $C\subseteq \omega \times 2$, the game $\mathcal{G}_{\omega\times2}(C)$ is determinate from the viewpoint of Player $I$.
   
  For assume $\forall \tau\in\mathcal{C}\exists n[\langle n, \tau(n)\rangle \in C]$. \\Define $B:=\bigcup_m\{ t\in Bin_m \mid \exists n<m)[\langle n, t(n)\rangle \in C]\}$ and note: $B$ is a bar in $\mathcal{C}$. Applying the Fan Theorem, find $u$ such that $\forall i<length(u)[u(i)\in B]$ and $\forall \tau \in \mathcal{C}\exists  i<length(u)[u(i)\sqsubset\tau]$. Find $m$ such that $\forall i<m[length(u(i))<m]$ and conclude:  $\forall t \in Bin_m\exists n<m[\langle n, t(n)\rangle \in C]$. \\We now prove, by backwards induction: \\for each $j\le m$, either $\forall t \in Bin_j \exists n<j[\langle n , t(n)\rangle \in C]$ or $\exists n\forall i<2[\langle n, i\rangle \in C]$. \\This clearly holds if $j=m$.\\ Now assume $j+1\le m$ and $\forall t \in Bin_{j+1} \exists n<j+1[\langle n, t(n)\rangle \in C]$. For every $t$ in $Bin_j$, one may consider $t\ast \langle 0\rangle$ and conclude: $\exists n<j[\langle n, t(n)\rangle \in C]$ or $\langle j, 0\rangle \in C$. Therefore: either $\langle j, 0\rangle \in C$ or $\forall t \in Bin_j\exists n<j[\langle n, t(n)\rangle \in C]$. For similar reasons,  either $\langle j, 1\rangle \in C$ or $\forall t \in Bin_j\exists n<j[\langle n, t(n)\rangle \in C]$.  Conclude: either: $\forall i<2[\langle j, i\rangle \in C]$ and $\exists n\forall i<2[\langle n , i\rangle \in C]$ or: $\forall t \in Bin_j\exists n<j[\langle n, t(n)\rangle \in C]$. Repeating this step $m$ times we find the conclusion: $\exists n\forall i<2[\langle n, i\rangle \in C]$.

  \subsection{Recovering the Fan Theorem}\label{SS:recft} Let us consider the result of the last subsection:\begin{quote} $(\#)$ For every $C\subseteq \omega \times 2$, \\if $\forall \tau\in\mathcal{C}\exists n[\langle n, \tau(n)\rangle \in C]$, then $\exists n[\langle n, 0\rangle \in C\;\wedge\;\langle n, 1\rangle \in C]$.\end{quote} We prove that this statement implies an important case of the Fan Theorem. The Fan Theorem is the statement that every $B\subseteq \mathbb{N}$ that is  a bar in $\mathcal{C}$ has a finite subset that is a bar in $\mathcal{C}$. We prove that $\#$ implies this statement for the case that $B$ is a \textit{decidable} subset of $\mathbb{N}$. 
     
    Assume $(\#)$ and  let $B\subseteq \mathbb{N}$ be given such that $Bar_\mathcal{C}(B)$ and \textit{one may decide, for each $n$, $n\in B$ or $n\notin B$.} For all $n$, we define: $B_n:=\{s\in B\mid s<n\}$. \\We want to prove the statement $QED:=\exists n[Bar_\mathcal{C}(B_n)]$\footnote{We read $QED$ not in its usual sense: `\textit{quod erat demonstrandum, what had to be proven}' but as: `\textit{quod est demonstrandum, what has to be proven}'.}. \\We let $C$ be the set of all pairs $\langle s, i\rangle$ such that \textit{either}: $QED$ \textit{or}: $Bar_{\mathcal{C}\cap s\ast\langle i \rangle}(B_n)$ and not $Bar_{\mathcal{C}\cap s\ast\langle 1-i \rangle}(B_n)$. Note that, for each $s$ in $Bin$, if both $s\ast\langle 0\rangle$ and $s\ast\langle 1\rangle$ are in $C$, then $QED$. Also note that $C$, like $B$, is a decidable subset of $\mathbb{N}$. \\Let $\tau$ in $\mathcal{C}$ be given. Find $\alpha$ in $\mathcal{C}$ such that $\forall i[\alpha(i)=\tau(\overline\alpha i)]$. Find $n$ such that $\overline \alpha n \in B$ and define $m:=\overline \alpha n +1$. Note: $Bar_{C\cap \overline \alpha n}(B_m)$. \\Note that, for each $i$, if $i+1\le n$ and $Bar_{\mathcal{C}\cap \overline\alpha(i+1)}(B_m)$, then \textit{either}: $Bar_{\mathcal{C}\cap \overline\alpha(i)}(B_m)$ \textit{or}: $\neg Bar_{\mathcal{C}\cap \overline\alpha(i)\ast\langle 1-\alpha(i)\rangle}(B_m)$ and  $\langle\overline\alpha i, \alpha(i)\rangle \in C$. Using  backwards induction, \\starting from the fact: $Bar_{\mathcal{C}\cap\overline\alpha n}(B_m)$, we prove:\\ for each $i\le n$, either $Bar_{\mathcal{C}\cap\overline\alpha i}(B_m)$ or $\exists j< n[\langle\overline\alpha j, \alpha(j)\rangle \in C]$. \\Taking $i=0$, we find: \textit{either} $Bar_\mathcal{C}(B_m)$ \textit{or} $\exists j[\langle\overline \alpha j, \alpha(j)\rangle \in C]$, i.e.  \textit{either}  $QED$ \textit{or} $\exists j[\langle\overline \alpha j, \tau(\overline \alpha j)\rangle \in C]$. In both cases, one may conclude: $\exists s[\langle s, \tau(s)\rangle \in C]$.

     We thus see: $\forall \tau \in \mathcal{C}\exists s[\langle s, \tau(s)\rangle \in C]$. Using $(\#)$, we conclude: \\$\exists s[s\ast\langle 0 \rangle \in C \;\wedge\;s\ast\langle 1 \rangle \in C]$, and, therefore: $QED$, i.e. $\exists n[Bar_\mathcal{C}(B_n)]$. 
     
     This shows that every decidable subset  $B$ of $\mathbb{N}$ that is a bar in $\mathcal{C}$ has a finite subset that is a bar in $\mathcal{C}$.
     \subsection{The Intuitionistic Determinacy Theorem} $\;$
     
     We now study games with infinitely many moves rather than two moves only.  \begin{definition} $(\omega\times 2)^\omega:=\{\alpha\mid\forall n[\alpha(2n+1)<2]\}$.
     
     \smallskip For each $\sigma$, for each $\alpha \in (\omega\times 2)^\omega$, $\alpha \in_I\sigma:= \forall n[\alpha(2n)=\sigma\bigl(\overline \alpha(2n)\bigr)]$
     
     \smallskip For each $\tau$ in $\mathcal{C}$, for each $\alpha$ in $(\omega\times 2)^\omega$,  $\alpha\in_{II}\tau:= \forall n[\alpha(2n+1)=\tau\bigl(\overline \alpha(2n+1)\bigr)]$. \end{definition}
     
     For every $\mathcal{X}\subseteq (\omega\times 2)^\omega$, we introduce a game $\mathcal{G}_{(\omega\times 2)^\omega}(\mathcal{X})$. There are two players, Player  $I$ and Player $II$. Every play takes infinitely many moves and goes as follows. Player $I$ chooses $n_0$ in $\omega$, thereafter Player $II$ chooses $i_0$ in $2=\{0,1\}$, then Player $I$ chooses $n_1$ in $\omega$ and, thereafter, Player $II$ chooses $i_1$ in $2$, and so on. Player $I$ wins the play if the infinite sequence $n_0, i_0, n_1, i_1,\ldots$ is in $\mathcal{X}$. Player $II$ tries to prevent that Player $I$ wins the game.
     
     \begin{theorem}[Intuitionistic Determinacy Theorem]\label{T:determinacy} For all $\mathcal{X}\subseteq (\omega\times 2)^\omega$, if $\forall \tau \in \mathcal{C}\exists \alpha \in (\omega\times 2)^\omega[\alpha \in_{II} \tau \;\wedge\; \alpha \in \mathcal{X}]$, then $\exists \sigma\forall \alpha\in (\omega\times 2)^\omega[\alpha \in _I \sigma\rightarrow \alpha \in \mathcal{X}]$.  \end{theorem}
     
     The Theorem is good news for Player $I$. Suppose $\mathcal{X}\subseteq \mathcal (\omega\times 2)^\omega$ is such that Player $I$  is able, once Player $II$ has told her the strategy $\tau$ she intends to follow, to find a play following $\tau$  that is in $\mathcal{X}$. She then may devise a strategy $\sigma$, such that every play following $\sigma$ is in $\mathcal{X}$.   Using $\sigma$, Player $I$ can do without information about the strategy followed by Player $II$. 
     
     We will not give the proof of Theorem \ref{T:determinacy}\footnote{See \cite[Chapter 16]{veldman1981}, \cite{veldman2009} and \cite[Section 9]{veldman2011d}.}. As in Subsection \ref{SS:gamesbyfan},  the Fan Theorem plays a key r\^ole in the argument. From Subsection \ref{SS:nondet}, one may see  it is crucial for Theorem \ref{T:determinacy} that, in $(\omega\times 2)^\omega$, Player $II$ has, at each one of her moves,  only  2 choices. The theorem extends to the case that Player $II$ has, at each one of her moves, finitely many choices.  
     
  \section{Brouwer's Thesis}
  
  \subsection{The Principle of Bar Induction}
   
   The key point in the proof of the next Theorem is a philosophical assumption\footnote{The theorem better might be called an axiom.} one might call \textit{Brouwer's Thesis  (on bars in $\mathcal{N}$)}\footnote{For a discussion, see \cite{veldman2006b}.}. 
   
  \begin{theorem}[Bar Induction]\label{T:bimbid}$\;$
     Let $B\subseteq \mathbb{N}$ be given such that $Bar_\mathcal{N}(B)$. Let $C\subseteq \mathbb{N}$ be given such that $B\subseteq C$ and, for all $s$, $s\in C$ if and only if $\forall n[s\ast\langle n \rangle \in C]$. Then $\langle\;\rangle \in C$.
     
    \end{theorem} 
 
 \begin{proof}  Let $B\subseteq \mathbb{N}$ be given such that $Bar_\mathcal{N}(B)$.

 Define, for every $s$,
 $s$ is \textit{B-safe}, or simply: {\it safe}, if and only if  $Bar_{\mathcal{N}\cap s}(B)$.  
 
 Note: $\langle \;\rangle$ is  safe.
 
 There must exist a \textit{canonical proof} of the fact: 
 
 \begin{center} $\langle \;\rangle$ is  safe. \end{center}
 
 The \textit{conclusion} of the canonical proof is:
 `$\langle \;\rangle$ is  safe.'
 
 The \textit{starting points}  of the canonical proof have the form:
 
     \begin{center}  $s \in B$.
     
     \textit{Therefore:} $s$ is safe.
     
     \end{center}
     
     There are  two kinds of reasoning steps, \textit{forward  steps} and \textit{backward  steps}. 
     
     \smallskip
     A \textit{forward reasoning step} has \textit{infinitely many premises} and is of the following form:
     
     \begin{center} 
     
     $s\ast\langle 0 \rangle$ is safe, 
     $s\ast\langle 1 \rangle$ is safe, 
     $s\ast\langle 2 \rangle$ is safe, 
     $\ldots$
     
     \smallskip
     \textit{Therefore:}  $s$ is safe.
     \end{center}
     
     \smallskip A \textit{backward reasoning step} has the form:
     
       \begin{center}
       
       $s $ is safe.

       \smallskip
       \textit{Therefore:} $s\ast\langle n \rangle$ is safe.
       \end{center}
       \medskip
       
      The canonical proof is not a finite tree but an infinite one. The canonical proof can not be written out explicitly as a finite text.
       
       \medskip
  Let $C\subseteq \mathbb{N}$ be given such that $B\subseteq C$ and, for all $s$, $s\in C$ if and only if $\forall n[s\ast\langle n \rangle \in C]$.

  \medskip
   Now replace in the canonical proof every statement `$s$ is safe' by the statement `$s\in C$'.  The result will be another \textit{valid proof}.
   
   Why? 
   
   As $B\subseteq C$, the new starting points are sound.
   
   \smallskip
   As, for every $s$, if $\forall n[s\ast\langle n \rangle \in C]$, then $s\in C$\footnote{$C\subseteq \mathbb{N}$ with this property is called \textit{inductive}.}, the new forward reasoning steps are sound.

      \smallskip
   As, for every $s$, for every $n$, if $s\in C$, then $s\ast\langle n \rangle \in C$\footnote{$C\subseteq \mathbb{N}$ with this property is called \textit{monotone}.}, the     new backward reasoning steps are sound.

       \smallskip
       We must conclude: the new conclusion is true, that is:
       $\langle \;\rangle\in C$. 
       \end{proof}
       
      \subsection{One needs backward steps}\label{SS:backward}

  \begin{theorem}[Kleene's example\footnote{See \cite[\S 7.14]{kleenevesley65}.}]\label{T:kleenebi}
  
   There exist $B\subseteq C\subseteq \mathbb{N}$ such that $Bar_\mathcal{N}(B)$ and, for all $s$, if $\forall n[s\ast\langle n \rangle \in C]$, then $s\in C$, while the statement `$\langle \;\rangle \in C$' is a reckless one. \end{theorem}
  \begin{proof}  Let $B$  consist of all $s$ such that\\ \textit{either} $\exists n<k_{99}[s = \langle n \rangle]$ \textit{or} $s = \langle \;\rangle$ and  $\exists n[n=k_{99}]\vee \forall n[n<k_{99}]$.

  Let $\alpha$ be given. Either $\alpha(0) < k_{99}$ and $\overline \alpha 1 \in B$, or $\alpha(0)\ge k_{99}$ and $\langle \;\rangle =\overline \alpha 0 \in B$. We thus see: $Bar_\mathcal{N}(B)$.
  
  Let $C$ coincide with $B$.
  
   For all $s$,  if, for all $n$,  $s\ast\langle n \rangle \in C$, then $s=\langle\;\rangle$    and $\forall n[n<k_{99}]$ and $\langle \;\rangle \in 
   C$.
   
   The statement `$\langle\;\rangle \in C$' is equivalent to `$\exists n[n=k_{99}]\vee \forall n[n<k_{99}]$', and is reckless. 
   \end{proof}
   \smallskip
   Consider the set $B$ mentioned in the proof of Theorem \ref{T:kleenebi}. Note that every canonical proof of `$\langle\;\rangle$ \textit{is (B-)safe}'
    must use backward steps. The last step in a canonical  proof of `$\langle\;\rangle$ \textit{is (B-)safe}'  must be a forward step: $\langle \;\rangle$ \textit{is safe} because: {\it $\langle 0\rangle$ is safe,  $\langle 1\rangle$ is safe, $\langle 2\rangle$ is safe, $\ldots$}. 
    For each $n<k_{99}$ the conclusion: {\it $\langle n \rangle$ is safe} follows by a basic step from: $\langle n \rangle \in B$. For each $n \ge k_{99}$ the conclusion: {\it $\langle n \rangle$ is safe} follows by a backward step from: {\it $\langle \; \rangle$ is safe} and that follows in its turn by a basic step from: $\langle \;\rangle \in B$ (as $n\ge k_{99}$).
    
    \smallskip We thus see that, in this particular case, the backward steps can not be missed. In the case of the Fan Theorem, Theorem \ref{T:bft}, as we learned from Theorem \ref{T:fanbi}, one might have claimed there is a canonical proof using only forward reasoning steps.

    \section{Open Induction in $[0,1]$}
 
 Brouwer's main application of the Principle of Bar Induction, Theorem \ref{T:bimbid}, is the Fan Theorem. We now give an example of a  stronger consequence of the Principle.

\begin{definition} We let $(r_0', r_0''), \;(r_1', r_1''), \;(r_2', r_2''), \;\ldots$ be a fixed enumeration of all pairs of rationals. 

 For each $\alpha$, $\mathcal{G}_\alpha :=\{x \in \mathcal{R}|\exists n[r_{\alpha(n)}'<x<r_{\alpha(n)}'']\}.$
 
 For each $a$, $\mathcal{G}_a:=\{x\in\mathcal{R}\mid \exists n<length(a)[r'_{a(n)}<x<r''_{a(n)}]\}$.

 \smallskip $\mathcal{G}\subseteq\mathcal{R}$ is  \emph{open} if and only if, for some $\alpha$, $\mathcal{G} = \mathcal{G}_\alpha$.
 
  \smallskip For all $x,y$ in $\mathcal{R}$ such that $x<y$, we define $[x,y):=\{z\in \mathcal{R}\mid x\le_\mathcal{R} z<_\mathcal{R}y\}$.
 
 \smallskip
  $\mathcal{G}\subseteq\mathcal{R}$ is \emph{progressive in $[0,1]$} if and only if $\forall x \in [0,1][[0, x)\subseteq \mathcal{G} \rightarrow x \in \mathcal{G}]$.
  
   \end{definition}
 
  \begin{theorem}[Principle of Open Induction in ${[}0,1{]}$] $\;$
  
   For every open  $\mathcal{G}\subseteq\mathcal{R}$, if $\mathcal{G}$ is progressive in $[0,1]$, then $[0,1]\subseteq \mathcal{G}$. \end{theorem}
  
Although the theorem may be new to the reader,  \'E. Borel in fact introduced and used it when he proved, in 1895, see \cite{borel}, what is now called the Heine-Borel Theorem: \begin{center} \textit{for each $\alpha$, if $[0,1]\subseteq \mathcal{G}_\alpha$, then $\exists n[[0,1]\subseteq\mathcal{G}_{\overline \alpha n}]$.} \end{center} Borel's argument may be described as follows.   Let $\alpha$ be given such that $[0,1]\subseteq \mathcal{G}_\alpha$. Define $\mathcal{H}:=\{y\in \mathcal{R}\mid \exists n[[0,y]\subseteq \mathcal{G}_{\overline \alpha n}]\}$. Note that $\mathcal{H}$ is open and progressive in $[0,1]$. Applying the principle of Open Induction on $[0,1]$, conclude: $1 \in \mathcal{H}$ and $\exists n[[0,1]\subseteq\mathcal{G}_{\overline \alpha n}]$.

 How does the classical mathematician convince herself of the validity of the Principle of Open Induction in $[0,1]$? 
 
 Given an open  $\mathcal{G}\subseteq\mathcal{R}$ that is progressive in $[0,1]$, she  starts defining an infinite sequence $x_0, x_1, \ldots$ of reals. She defines $x_0:=0$ and notes  $x_0\in \mathcal{G}$. She then finds $x_1>x_0$ such that $[0,x_1)\subseteq \mathcal{G}$. If $x_1\ge 1$,  she is done and  defines, for each $n\ge 1$, $x_n=1$. If not, she  observes: $x_1 \in \mathcal{G}$ and  finds $x_2>x_1$ such that $[0,x_2)\subseteq \mathcal{G}$.  If $x_2 \ge 1$  she is done and  defines, for each $n\ge 2$, $x_n=1$. If not, she observes: $x_2 \in \mathcal{G}$ and finds $x_3>x_2$ such that $[0,x_3)\subseteq \mathcal{G}$. And so on. She thus finds  an infinite  sequence $0\le x_0 \le x_1 \le x_2\le\ldots$ of reals  such that, for each $n$, $[0, x_n)\subseteq \mathcal{G}$ and: $x_{n+1}=x_n$ if and only if $x_n=1$.  The infinite non-decreasing sequence $x_0, x_1, \ldots$  is bounded by $1$ and thus  converges, to, say, $x_{\omega+0}=x_\omega$. Then: $x_\omega \in \mathcal{G}$ as $[x,x_\omega)\subseteq  \mathcal{G}$.   If $x_\omega \ge 1$, she is done and  defines, for each $n\ge 0$, $x_{\omega+n}=1$. If not, she starts again  and  finds  suitable  $x_{\omega +1}, x_{\omega +2}, \ldots$ such that, for each $n$, $[0, x_{\omega +n})\subseteq \mathcal{G}$ and $x_{\omega +n}\le x_{\omega +n+1}$ and: $x_{\omega+n+1}=x_{\omega+n}$ if and only if $x_{\omega+n}\ge 1$.   She then considers $x_{\omega\cdot 2}= x_{\omega+\omega}:=\lim_n x_{\omega+n}$. She has inexhaustible energy and, if needed, finds $x_{\omega\cdot 3}$ $x_{\omega\cdot 4}, \ldots x_{\omega \cdot \omega}, \ldots$. She thus continues her infinite sequence through the countable ordinals.
 
 She believes  this procedure will come to an end, that is: there must exist a countable ordinal $\beta$ such that $x_\beta=1$. She  argues a follows. For each $k>0$ there can be only $k$ countable ordinals $\alpha$ such that $x_\alpha+\frac{1}{k}<x_{\alpha+1}$. Therefore, there are only countably many countable ordinals such that $x_\alpha<x_{\alpha+1}$. Find a countable ordinal $\beta$ that lies beyond every countable ordinal $\alpha$ such that $x_\alpha<x_{\alpha +1}$. Then $x_\beta=x_{\beta+1}$ and: $x_\beta=1$.   
 
 Borel argued in this way, proudly using the countable ordinals introduced by Cantor some 20 years earlier, see \cite{hallett}.

 This argument, from a constructive point of view, is quite fantastic. We   parted company  with the classical mathematician already at the point where she believed to find $x_\omega$. It is not true, constructively, that every non-decreasing bounded infinite sequence of reals has a limit: consider the sequence defined by: $x_n=0$ if $n<k_{99}$ and $x_n=1$ if $ k_{99}\le n$. \footnote{Assume $c:=\lim_n x_n$ exists. If $c>0$, then $\exists n[n=k_{99}]$ and, if $c<1$, then $\forall n[n<k_{99}]$.} 
 
 \smallskip
 The classical mathematician  may propose a second line of thought. She might use {\it a classical mathematician's finest weapon}\footnote{See Section \ref{S:infprime}.}, as follows.  Assume $\mathcal{G}\neq [0,1]$.  Consider $y:=\inf([0,1]\setminus \mathcal{G})$. As $\mathcal{G}$ is open, $y\notin \mathcal{G}$, but also  $[0,y)\subseteq \mathcal{G}$. Contradiction. 
 
 Again, we can not partake in her joy.  It is not clear that we may construct $y$.  And, of course,  it is not clear that, having proved: $\neg\exists x\in[0,1][x\notin \mathcal{G}]$, we might conclude: $\forall x\in[0,1][x\in \mathcal{G}]$.   
 \smallskip
 
 In order to make the problem more pictorial, let us define: \textit{Achilles\footnote{See \cite[p. 338]{veldman2001}.} arrives at $x$} if and only if $[0,x)\subseteq \mathcal{G}$.

  How might  we, constructive mathematicians,  convince ourselves  that Achilles  arrives at $1$?  We should first prove that Achilles will arrive at $\frac{1}{2}$, but it does not seem easier to prove that Achilles will arrive at $\frac{1}{2}$ than to prove that he will arrive at $1$. 
  
  \begin{proof}\footnote{The theorem was found and proven by Thierry Coquand, in 1997. We were discussing then the problem if positive contrapositions of the \textit{minimal bad sequence} arguments used in \cite{Nash-Williams} might be intuitionistically true, see \cite[\S 11]{veldman2004}.}  Let  $\mathcal{G}\subseteq\mathcal{R}$ be open and progressive in $[0,1]$. Find $\alpha$ such that $\mathcal{G} = \mathcal{G}_\alpha$. We  define a mapping $s\mapsto (c_s, d_s)$ that associates to every  $s$ in $\mathbb{N}$ a pair of rationals. The elements $s$ of $\mathbb{N}$ are thought of as coding finite sequences of natural numbers, see Subsection \ref{SS:bcp}, Definition \ref{D:coding}. 
  
  \begin{enumerate}[\upshape (i)] \item $(c_{\langle\;\rangle}, d_{\langle\;\rangle}):=(0,1)$.  \item For each $s$, for each $n$,  find out if   $\forall x \in [0, \frac{c(s)+ d(s)}{2}]\exists i < n [r'_{\alpha(i)}<x < r''_{\alpha(i)}]$.\footnote{Note that this a \textit{decidable} proposition.} \\\textit{If so}, define $(c_{s\ast\langle n\rangle} , d_{s\ast\langle n\rangle})=( \frac{c_s+d_s}{2}, d_s)$, and \\\textit{if not}, define $(c_{s\ast\langle n\rangle} , d_{s\ast\langle n\rangle})=( c_s,\frac{c_s+d_s}{2})$. 
 \end{enumerate} 
 \smallskip Note: for each $s$,  $(c_{s\ast\langle 0\rangle} , d_{s\ast\langle 0\rangle})=(c_s, \frac{c_s+d_s}{2})$.
  \\Also note: for each $s$, $\exists n\forall x \in [0,c_s) \exists i<n[r'_{\alpha(i)}<x<r''_{\alpha(i)}]$.   \\One proves this by induction on the length of $s$.

   \smallskip
   For each $\beta$ in $\mathcal{N}$, we let $x_\beta$ be the real number such that, for each $n$, \\$x_\beta(n)=(c_{\overline\beta n}, d_{\overline \beta n})$. Note: for each $\beta$, $[0, x_\beta)\subseteq \mathcal{G}$ and thus $x_\beta \in \mathcal{G}$.

   \smallskip
   Let  $B$ be the set of all $s$  such that $\exists n\forall x\in [0, d_s]\exists i< n[r'_{\alpha(i)}<x<r''_{\alpha(i)}]$.

   \smallskip
   Let $\beta$ be given. Using the fact: $x_\beta \in \mathcal{G}$, find $q$ such that $r'_{\alpha(q)}<x_\beta<r''_{\alpha(q)}$. \\Find  $m$ such that $r'_{\alpha(q)}<c_{\overline\beta m}<d_{\overline\beta m}<r''_{\alpha(q)}$. \\Find   $p$ such that $\forall x \in [0, c_{\overline \beta m}]\exists i<p[r'_{\alpha(i)}<x<r''_{\alpha(i)}]$.  \\Define $n:=\max(p, q+1)$. \\ Conclude: $\forall x \in [0, d_{\overline \beta m}]\exists i<p][r'_{\alpha(i)} < x<r''_{\alpha(i)}]]$ and: $\overline \beta m \in B$.

   We thus see: $Bar_\mathcal{N}(B)$, i.e.: B is a bar in $\mathcal{N}$.   
   
   \smallskip
   
   Note: $B$ is monotone, i.e.: for all $s$, if $s\in B$, then $\forall n[s\ast\langle n \rangle \in B]$.
   
   \smallskip
   We now prove that $B$ is also inductive. \\Let $s$ be given such that $\forall n[s\ast\langle n \rangle \in B]$. Then,   in particular: $s\ast\langle 0\rangle \in B$. \\ Note $d_{s\ast\langle 0 \rangle}=\frac{c_s+d_s}{2}$ and find $n$ such that $\forall x \in [0, \frac{c_s+d_s}{2}]\exists i<n[r'_{\alpha(i)}<x<r''_{\alpha(i)}]$. \\Conclude: $(c_{s\ast\langle n \rangle}, d_{s\ast\langle n \rangle})= (\frac{c_s+d_s}{2}, d_s)$. Using the fact: $s\ast\langle n \rangle \in B$, find $m$ such that $\forall x \in [0, d_{s\ast\langle n\rangle}]\exists i<m[r'_{\alpha(i)}<x<r''_{\alpha(i)}]$. \\ Note: $d_{s\ast\langle n \rangle }= d_s$ and conclude: $s \in B$.
   
   We thus see: if $\forall n[s\ast\langle n \rangle \in B]$, then $s\in B$. 
   
   \smallskip
  As $Bar_\mathcal{N}(B)$ and $\forall s[s\in B\leftrightarrow \forall n[s\ast\langle n \rangle \in B]]$, we may conclude, using Theorem \ref{T:bimbid}: $\langle \;\rangle \in B$, that is: $\exists n \forall x\in [0,1]\exists i<n[r'_{\alpha(i)}< x < r''_{\alpha(i)}]$ and: $[0,1]\subseteq \mathcal{G}$.   \end{proof}
  
  The Principle of Open Induction in $[0,1]$ plays a large r\^ole in \cite{veldman2011c}.
      \section{Brouwer's Thesis, again}

      \subsection{Using stumps}\label{SS:stumps}
      In intuitionistic analysis, stumps play the r\^ole fulfilled by countable ordinals in classical analysis.\begin{definition} For every $s$ in $\mathbb{N}$, for every  $A\subseteq\mathbb{N}$, 
$s\ast A := \{s\ast t \mid t \in A\}.$

 $\mathbf{Stp}$, a collection  of  subsets of $\mathbb{N}$, called \emph{stumps}, is defined as follows. 

\begin{enumerate}[\upshape (i)]
\item $\emptyset \in \mathbf{Stp}$, and
\item for every infinite sequence $S_0, S_1, \ldots$ of elements of $\mathbf{Stp}$,  the set \\$S:= \{\langle \;\rangle\}\cup\bigcup\limits_{n \in \mathbb{N}}\langle
 n \rangle \ast S_n$ is again an element of $\mathbf{Stp}$, and
\item nothing more: every element of $\mathbf{Stp}$ is obtained by starting from $\emptyset$ and applying the operation mentioned in (ii) repeatedly. \end{enumerate}\end{definition}

Note that, for every stump $S$,  either $S = \emptyset$ or $\langle \;\rangle \in S$.

\begin{definition}\label{D:substump} For every $A\subseteq \mathbb{N}$, for every $n$, $A\upharpoonright\langle n \rangle:= \{s\mid\langle n \rangle \ast s\in A\}$.

 For every non-empty stump $S$, for every $n$, $S\upharpoonright\langle  n\rangle$ is again a stump. 

$S\upharpoonright\langle  n\rangle$ is called \emph{the $n$-th immediate substump of $S$}. \end{definition}

\begin{axiom}[Principle of Induction on $\mathbf{Stp}$]\label{A:stpinduction}

 Let $P\subseteq\mathbf{Stp}$ be given. If $\emptyset \in P$, and
 for every non-empty stump $S$, if $\forall n[S\upharpoonright\langle  n\rangle \in P]$, then $S \in P$, then   $\mathbf{Stp}= P$.  \end{axiom}
 
 \begin{theorem}[Brouwer's Thesis on bars in $\mathcal{N}$\footnote{See \cite[\S 13.0]{veldman1981}, \cite[Theorem 1.1]{veldman2006b} and \cite[Theorem 2]{veldman2008}.
 }]\label{T:brthesis} For every   $B\subseteq \mathbb{N}$   such that  $Bar_\mathcal{N}(B)$,
 there exists a stump $S$ such that $Bar_\mathcal{N}(S \cap B)$. \end{theorem}
 
 \begin{proof} Let $B\subseteq\mathbb{N}$ be given such that $Bar_\mathcal{N}(B)$.

 Let $C$ be the set  of all $s$ such that, for some stump $S$,
 
 $ \forall \alpha \exists n[
   \overline  \alpha n \in S\;\wedge\;\exists t\sqsubseteq s\ast\overline \alpha n[t\in B]].$ 
  
  \smallskip
 Let $s$ be given such that  $s\in B$. Define $S:=\{\langle\;\rangle\}$ and note:\\
    $\forall \alpha [
   \overline  \alpha 0 \in S \;\wedge\; s\ast\overline\alpha 0 \in B].$
   
   We thus see: $B\subseteq C$. 
   
   \smallskip
   We now prove that $C$ is inductive.
   
   Let $s$ be given such that $\forall m[s\ast\langle m \rangle\in C]$.
      
    Find an infinite sequence $S_0, S_1, \ldots$ of stumps such that \\$\forall m \forall \alpha \exists n[
   \overline  \alpha n \in S_m \;\wedge\; \exists t\sqsubseteq s\ast\langle m \rangle \ast \overline\alpha n[t \in B]].$  Define a non-empty stump $S$  such that, for each $m$, $ S\upharpoonright\langle  m\rangle=S_m$ and conclude: 	$\forall \alpha \exists n[
   \overline  \alpha n \in S \;\wedge\; \exists t\sqsubseteq s\ast\overline\alpha n[t \in B]$ and: $s\in C$.
   
   We thus see: for each $s$, if $\forall m[s\ast\langle m \rangle \in C]$, then $s\in C$.

      \smallskip

     Finally, we show that $C$ is monotone.
     
      Let $s$ in $C$ be given.  
      
      Find a stump $S$ such that $\forall \alpha \exists n[
   \overline  \alpha n \in S \;\wedge\; \exists t\sqsubseteq s\ast\overline\alpha n\in B]]$. 
   
   Let $m$ be given and distinguish two cases:
   
   \textit{Case (1)}. $S\upharpoonright\langle  m\rangle \neq \emptyset$. Then $\forall \alpha \exists n[
   \overline  \alpha n \in S\upharpoonright\langle  m\rangle \;\wedge\; \exists t\sqsubseteq s\ast\langle m \rangle  \ast \overline\alpha n \in B]]$, and: $s\ast\langle m\rangle \in C$. 
   
    \textit{Case (2)}. $S\upharpoonright\langle  m\rangle = \emptyset$. Define $T := \{\langle\;\rangle\}$ and note:  \\$\forall \alpha [
   \overline  \alpha 0 \in T \;\wedge\; \exists t\sqsubseteq s\ast\langle   m \rangle\ast \overline\alpha 0[t\in B]]$, and: $s\ast\langle m \rangle \in C$.  
      
      We thus see: for all $s$, for all $m$, if $s\in C$, then $s\ast\langle m \rangle \in C$.
      
      \smallskip Using Theorem \ref{T:bimbid}, we conclude: $\langle \;\rangle \in C$, 
    that is: there exists a stump $S$ such that $\forall \alpha \exists n[
   \overline  \alpha n \in S \;\wedge\;  \overline\alpha n \in B]$, that is:  $Bar_\mathcal{N}(S\cap B)$.  
       \end{proof} 
    
    In the next Section, we show that  Theorem \ref{T:brthesis} is a useful conclusion from Theorem \ref{T:bimbid}.  
       
  \section{Ramseyan Theorems} 
   
   \subsection{Dickson's Lemma}
   
   \begin{definition}  For all $\alpha, \beta$ in $\mathcal{N}$,    $\alpha\circ \beta$ is the element of $\mathcal{N}$ such that, for each $n$, $\alpha\circ\beta(n) = \alpha\bigl(\beta(n)\bigr).$
   
   \smallskip $[\omega]^\omega :=\{\zeta \in \mathcal{N}\mid\forall i[\zeta(i) <\zeta(i+1)]\}.$ 
    \end{definition}

   The following Lemma, a humble member of a family of results called  \textit{Ramseyan Theorems}, plays a r\^ole  in Computer Algebra. The Lemma  lies at the basis of Buchberger's algorithm.\footnote{See \cite[Theorem 5]{cox}.} 
   
   \begin{theorem}[Dickson's Lemma]\label{T:dickson} For all $p>0$, for all $\alpha_0, \alpha_1, \ldots \alpha_{p-1}$ in $\mathcal{N}$, there exist $i,j$ such that $i<j$ and $\forall k<p[\alpha_k(i)\le\alpha_k(j)]$. \end{theorem}
   
 We first sketch  a proof along traditional lines that fails to be constructive.
 
 \smallskip  The proof is by induction to $p$. 
   
   \smallskip First consider the case $p=1$. Let $\alpha_0$ be given and consider $\alpha_0(0)$. Note: $\exists i\le \alpha_0(0)[\alpha_0(i)\le\alpha_0(i+1)]$.
   
   \smallskip Now let $p$ be given such that the case $p$ of the Theorem has been established. We want to prove the case $p+1$.
   
   Let $\alpha_0, \alpha_1, \ldots, \alpha_{p}$ be given.
   
  $(\flat)$ Find $\zeta$ in $[\omega]^\omega$ such that $\forall i[\alpha_p\circ\zeta(i)\le\alpha_p\circ\zeta(i+1)]$.
  
  Applying the induction hypothesis, determine $i,j$ such that, for all $k<p$, \\$\alpha\circ\zeta(i)\le\alpha\circ\zeta(j)$ and conclude: for all $k\le p+1$, $\alpha\circ\zeta(i)\le\alpha\circ\zeta(j)$.
  
  This completes the proof of the induction step.
  
  \smallskip Unfortunately, there are difficulties with $(\flat)$. It is, in general, not possible to find $\zeta$ in $[\omega]^\omega$ such that $\forall i[\alpha_p\circ\zeta(i)\le\alpha_p\circ\zeta(i+1)]$. The following example makes this clear. 
  
  Define $\alpha$ such that, for all $n$, if $n<k_{99}$, then $\alpha(n)=1$, and, if $k_{99}\le n$, then $\alpha(n)=0$. Suppose we find $\zeta$ in $[\omega]^\omega$ such that $\forall i[\alpha\circ\zeta(i)\le\alpha\circ\zeta(i+1)]$. \\If $\alpha\circ\zeta(0) =0$, then $k_{99}\le\zeta(0)$, and, if $\alpha\circ\zeta(0)=1$, then $\forall n[n<k_{99}]$.\\ The statement `$\exists \zeta \in [\omega]^\omega\forall i[\alpha\circ \zeta(i)\le\alpha\circ\zeta(i+1)]$' thus is \textit{reckless} or \textit{hardy}. 
  
  We now give a constructive proof of Dickson's Lemma.
  
  \begin{proof}\footnote{See \cite[Theorem 1.1]{veldman2004}.}
  The proof is by induction on $p$. The case $p=1$ is easy and is done as in the sketch preceding this proof.
  Now let $p$ given such that the case $p$ has been established. We prove the case $p+1$.
  
  \smallskip
  Let $\alpha_0, \alpha_1, \ldots, \alpha_p$ be given. 
  \\Define\footnote{$QED$ (again) stands for: \textit{quod \emph{est} demonstrandum, `what we (still) have to prove'}.} $QED:=\exists i \exists j[i<j\;\wedge\;\forall k\le p[\alpha_k(i)\le\alpha_k(j)]$. 
 \\ We first prove: \begin{quote}$(\ast)$ for each $m$, $\forall n\exists j\ge n[m\le\alpha_p(j)\;\vee\;QED]$.\end{quote}
  This again is done by induction. Note that the case $m=0$ is trivial.
  \\ Now assume $m$ is given and the case $m$ has been established.\\ We prove the case $m+1$.
  \\ Let $n$ be given.\\ Using the induction hypothesis, find $\zeta$ in $[\omega]^\omega$ such that $\forall i[\alpha_p\circ\zeta(i)\ge m \;\vee\;QED]$.\\ Find $i,j$ such that $n<i<j$ and $\forall k<p[\alpha_k\circ\zeta(i)\le\alpha_k\circ\zeta(j)]$. \\Note: either $m+1\le\alpha_p\circ\zeta(i)$ or $m+1\le\alpha_p\circ\zeta(j)$ or $QED$.\\We thus see: $\forall n\exists j\ge n[m+1\le\alpha_p(j)\;\vee\;QED]$.
  This completes the proof of $(\ast)$. 
  
  \smallskip Using $(\ast)$ we find $\zeta$ in $[\omega]^\omega$ such that $\forall n[\alpha_p\circ\zeta(n)\le \alpha_p\circ\zeta(n+1)\;\vee\;QED]$. \\Find $i,j$ such that $i<j$ and $\forall k<p[\alpha_k\circ\zeta(i)\le\alpha_k\circ\zeta(j)]$. \\Note:  $\alpha_p\circ\zeta(i)\le\alpha_p\circ\zeta(j)\;\vee\;QED$ and conclude: $QED$.
  
  This concludes the proof of the case $p+1$.
  
  The theorem thus is proven by induction.\end{proof}
  
   \subsection{How to formulate and extend Ramsey's Theorem?}\label{SS:irt}$\;$
  
 The following Definition extends Definition  \ref{D:coding} in Subsection \ref{SS:bcp}. 
 \begin{definition}\label{D:coding2}
   For all $\alpha$, for all $n$, for all $s$ in $\omega^n$,    $\alpha\circ s$ is the element of $\omega^n$ such that, for each $i<n$, $\alpha\circ s (i) = \alpha\bigl(s(i)\bigr).$
   
   \smallskip
   
    For all $m$, for all $s$ in $\omega^m$, for all $n$, for all $t$ in $\omega^n$  such that $\forall i<n[ t(i)<m]$,  $s\circ t$ is the element  of $\omega^n$ such that  for each  $i<n$, $s\circ t(i) = s\bigl(t(i)\bigr).$
   
   \smallskip
 $[\omega]^n:=\{s\in\omega^n\mid \forall i[i+1<n\rightarrow s(i)<s(i+1)]\}$.

   \smallskip $[\omega]^{<\omega} :=\bigcup_{n}[\omega]^n$.

  \end{definition}

 \begin{definition}\label{D:almost-full}\footnote{This definition improves the definition as used in \cite[Section 4.2]{veldman2004}. The proof of \cite[Theorem 4.4]{veldman2004} is not correct as was pointed out to me by M. Bickford. The argument may be saved if one restricts attention to \textit{strictly increasing} sequences as we do here.}

     $A\subseteq \mathbb{N}$  is  \emph{almost-full} if and only if  $\forall \zeta \in [\omega]^\omega\exists  s\in [\omega]^{<\omega}[\zeta\circ s\in A]$. 
   \end{definition}
 Let $k>0$ be given. The $k$-dimensional case of Ramsey's Theorem is the following statement:
  \begin{quote}$\mathbf{IRT}(k)$: {\it For all $A,B\subseteq [\omega]^k$, \\if $A,B$ are almost-full, then $A\cap B$ is almost-full.}\end{quote}
  
  As, for all $A\subseteq [\omega]^k$,  $A\cap([\omega]^k \setminus A)=\emptyset$ is not almost-full, the classical mathematician draws the following conclusion from $\mathbf{IRT}(k)$: \begin{quote} $\mathbf{IRT}(k)_{class}$: {\it For all $A\subseteq [\omega]^k$,\\ either $A$ is not almost-full or $[\omega]^k\setminus A$ is not almost-full},\end{quote} i.e. 
  \begin{quote} {\it  either $\exists\zeta\in[\omega]^\omega\forall s\in [\omega]^k[\zeta\circ s\in A]$ or $\exists\zeta\in[\omega]^\omega\forall s\in [\omega]^k[\zeta\circ s\notin A]$.}\end{quote}
  
  This is the formulation of the theorem that will be familiar to the classical reader. She also will  be able to conclude (her reading of) $\mathbf{IRT}(k)$ from $\mathbf{IRT}(k)_{class}$.\footnote{Constructively, already $\mathbf{IRT}(1)_{class}$ is a \textit{reckless} or \textit{hardy} statement:\\ consider $A:=\{\langle n \rangle\mid n<k_{99}\}$.}
  
  \smallskip
  F.P. Ramsey proved: $\forall k>0[\mathbf{IRT}(k)_{class}]$, see  \cite{ramsey}.
  
 \smallskip There is an extension of Ramsey's result into the transfinite that is called \textit{the Clopen Ramsey Theorem}.
 
 We need the following definition.  \begin{definition} For every stump $U$, we define $\overline U:=\{s\mid s \notin U \;\wedge\;\forall t\sqsubset s[t\in U]\}$. 
  The set $\overline U$ is called the \emph{border} of the stump $U$. \end{definition}
  
   Let a stump $U$ be given. We introduce the following statement:
  \begin{quote}$\mathbf{IRT}(\overline U)$: {\it For all $A,B\subseteq \overline U $, \\if $A,B$ are almost-full, then $A\cap B$ is almost-full.}\end{quote}
  
  \smallskip
  The \textit{Clopen Ramsey Theorem}\footnote{See   \cite{clote} and \cite{fraisse}.}, $\mathbf{CRT}$, is the statement: \\{\it for every stump $U$, $\mathbf{IRT}(\overline U)$.}
  
  \smallskip
  The \textit{Infinite Ramsey Theorem}\footnote{See \cite[Theorem 7.3]{veldmanbezem}}, $\mathbf{IRT}$, is the statement: {\it for all $k$, $\mathbf{IRT}(k)$.}
  
  \smallskip
  $\mathbf{IRT}_{class}$ is the (constructively wrong) statement: {\it for all $k$, $\mathbf{IRT}(k)_{class}$}. 
  
   \subsection{The proof of the Clopen Ramsey Theorem}\label{SS:proofclopenr}\footnote{The Theorem has been announced in \cite[\S 10.2]{veldmanbezem} and \cite[\S 6]{veldman2008} but, until now, no proof appeared in print.}
 
 The following Definition extends Definition \ref{D:coding2} in Subsection \ref{SS:irt}. 
  \begin{definition}\label{D:secures} For each $n,$ for each $k$, $[n]^k:=\{s\in [\omega]^k\mid\forall i<k[s(i)<n]\}$. 
  
  \smallskip For each $n$, for each $k$, $[n]^{<k}:=\bigcup_{i<k}[n]^i$. 
  
  \smallskip Let a  stump $S$ be given. 
   For every $A\subseteq \mathbb{N}$, \emph{$S$ secures $A$} if and only if \\$\exists m\forall \zeta\in [\omega]^\omega[\zeta(0)>m\rightarrow\exists n\exists s \in [n]^{<n+1}[\overline \zeta n \in S\;\wedge\; \overline\zeta n\circ s \in A]]$.\end{definition} 
  
  \begin{corollary}\label{C:stumpsecurealmostfull} For each $A\subseteq \mathbb{N}$, if $A$ is almost-full, then  there exists a stump $S$ securing $A$. \end{corollary}
  
  \begin{proof} Let $A\subseteq \mathbb{N}$ be almost-full.
  
   Define $B:=\bigcup_n\{u \in \omega^n\mid u \notin [\omega]^n \;\vee\; \exists s\in [n]^{<n+1}[u\circ s \in A]\}$. 
  
 \smallskip We prove that $B$ is a bar in $\mathcal{N}$.
  
  Let $\alpha$ be given. Define $\zeta$ such that $\zeta(0)=\alpha(0)$ and,  for each $n>0$, \\if $\overline \alpha(n+1)\in [\omega]^{n+1}$, then $\zeta(n)=\alpha(n)$, and, if not, then $\zeta(n)=\zeta(n-1)+1$.  
  
  Note: $\zeta\in [\omega]^\omega$. \\Find $n$ such that $\exists s\in [n]^{<n+1}[\overline\zeta n \circ s \in A]$, and, therefore, $\overline \zeta n \in B$. 
  
   Note: if $\overline \alpha n \neq \overline \zeta n$, then $\overline \alpha n \notin [\omega]^n$, and also: $\overline \alpha n \in B$.
  
  In any case, therefore, $\overline \alpha n \in B$.
  
  \smallskip We thus see: $Bar_\mathcal{N}(B)$. 
  
  Applying Theorem \ref{T:brthesis}, find a stump $S$ such that $Bar_\mathcal{N}(S\cap B)$.

  Note: $S$ secures $A$.
  \end{proof}

   \begin{definition}\label{D:Usecures} Let  a stump $U$ be given and assume $A\subseteq \overline U$ and  $u\in U\cup\overline U$. \\  $u\;\;U$-\emph{secures $A$} if and only if either: $u\in A$ or: $u\in U$ and  $\forall t\in\overline U[u\sqsubset t\rightarrow t\in A]$.
   
   \smallskip Let also a stump $S$ be given.  
    \emph{$S\;\;U$-secures $A$} if and only if \\$\exists m\forall \zeta\in [\omega]^\omega[\zeta(0)>m\rightarrow\exists n[\overline \zeta n \in S\;\wedge\;\exists s \in [n]^{<n+1} [\overline\zeta n\circ s\;\; U$-secures $A]]$.\end{definition} 
   
   \begin{corollary}\label{C:stumpsecurealmostfullU} For all stumps $U$, for each $A\subseteq \overline U$, if $A$ is almost-full, then  there exists a stump $S$ $U$-securing $A$. \end{corollary}
   
   \begin{proof} Note: if $S$ secures $A$, then $S\;\;U$-secures $A$ and use  Corollary \ref{C:stumpsecurealmostfull}. \end{proof}
  \begin{theorem}[Intuitionistic Ramsey Theorem]\label{T:irt}\;\begin{enumerate}[\upshape (i)] \item
 $\mathbf{CRT}:$ For all stumps $U$, for all  $A,B\subseteq \overline U$, \\if $A,B$ are both almost-full, then $A\cap B$ is almost-full.

 \item $\mathbf{IRT}:$
  For all $k$, for all  $A,B\subseteq [\omega]^k$, \\if $A,B$ are both almost-full, then $A\cap B$ is almost-full.  \end{enumerate}\end{theorem} 
   \begin{proof} 
   (i) According to Corollaries \ref{C:stumpsecurealmostfull} and \ref{C:stumpsecurealmostfullU}, it suffices to prove: \\for all stumps $U$, $P(U)$, \\where $P(U)$ is the statement:  for all stumps $S$, $Q(U,S)$, \\where $Q(U,S)$ is the statement: for all stumps $T$,  $R(U,S,T)$,\\ where $R( U,S,T)$ is the statement: \\for all $A,B\subseteq \overline U$, if $S$ $U$-secures $A$, and $T$ $U$-secures $B$, then $A\cap B$ is almost-full. 
   
   \smallskip We use Axiom \ref{A:stpinduction} from Subsection \ref{SS:stumps}, the principle of induction on the set $\mathbf{Stp}$ of stumps. Our proof is a proof by \textit{triple induction}. 
   
 \smallskip 1. Note: $\overline \emptyset =\{\langle \; \rangle\}$.
   
   For all $A,B\subseteq\{\langle\;\rangle\}$, if $A,B$ are almost-full, then $A=B=A\cap B=\{\langle\;\rangle\}$. 
   
   We thus see that  $P(\emptyset)$ is true.
   
   \smallskip 2. Let a non-empty stump $U$ be given  such that, for all $n$, $P(U\upharpoonright \langle n\rangle)$.
   
   We intend to prove $P(U)$, that is: for all stumps $S$, $Q(U, S)$.  
   
   We use  again use Axiom \ref{A:stpinduction}. 
   
  2.1.  Note that $Q(U, \emptyset)$ holds a trivial reason:  no subset of $\mathbb{N}$ is secured by $\emptyset$.
   
   \smallskip
  2.2.  Let a non-empty stump $S$ be given such that, for all $n$,  $Q(U,S\upharpoonright \langle n \rangle)$.
   
   We intend to prove: $Q(U, S)$, that is, for all  stumps $T$, $R(U, S,T)$. 
   
   Again, we use Axiom \ref{A:stpinduction}.
   
 \smallskip 2.2.1. Note that $R(U, S, \emptyset)$ holds for a trivial reason: no subset of $\mathbb{N}$ is secured by $\emptyset$.
   
   \smallskip
   2.2.2. Let a non-empty stump $T$ be given such that, for all $n$, $R(U,S, T\upharpoonright\langle n \rangle)$.
   
   We intend to prove: $R(U,S,T)$. 
   
   \smallskip
   Let $A,B\subseteq \overline U$ be given such that $S$ $U$-secures $A$ and $T$ $U$-secures $B$. 
   
   We have to prove that $A\cap B$ is almost-full, i.e. \begin{quote}$(\ast)\;\forall\zeta\in [\omega]^\omega\exists s\in [\omega]^{<\omega}[\zeta\circ s\in A\cap B]$. \end{quote}
   
   It is useful, however, to first prove a slightly weaker statement: \begin{quote} $(\ast\ast)\;\forall\zeta\in [\omega]^\omega\exists s\in [\omega]^{<\omega}[\zeta\circ s\in A\cap B\;\vee\; \\
   (s\neq 0 \;\wedge\;s(0)=0\;\wedge\; \zeta \circ s \in A)]$.
   \end{quote}

    \smallskip  Let $\zeta$ in $[\omega]^\omega$ be given. 
    
   We distinguish two cases.
    
    \smallskip
    \textit{Case (a)}.   $S\upharpoonright \langle \zeta(0)\rangle=\emptyset$. Note:   $\{\langle\;\rangle\}$ $\;U$-secures $A$, and $A=\overline U$ and $A\cap B=B$ is almost-full. One may conclude: $(\ast)$, and, in particular, $\exists s \in [\omega]^{<\omega}[\zeta\circ s \in A\cap B]$. 
  
  \smallskip \textit{Case (b)}.   $S\upharpoonright \langle \zeta(0)\rangle\neq\emptyset$. Then  $\langle \;\rangle \in S\upharpoonright \langle \zeta(0) \rangle$. 
    
    \smallskip
    Define a statement\footnote{$QED$ (again) stands for: \textit{quod \emph{est} demonstrandum, `what we (still) have to prove'}.}\begin{quote} $QED_1:=\exists s\in [\omega]^{<\omega}\setminus\{0\}   [s(0)=0\;\wedge\;\zeta\circ s \in A]$. \end{quote}

   \smallskip 
    Define 
     $A_\zeta:=\{u \in \overline U\mid \neg \exists t[u =\zeta\circ t] \;\vee\; u\in A \;\vee \;QED_1\}$.

    \smallskip Note that, for all $u$ in $U\cup \overline U$, if $\neg\exists t[u=\zeta\circ t]$, then $u$ $\;U$-secures $A$ and also $A_\zeta$.

    \smallskip We now want to prove:   \begin{center} $S\upharpoonright \langle \zeta(0)\rangle$ $U$-secures $A_\zeta$.\end{center} 
     
 \smallskip   First, using the fact that $S$ $U$-secures $A$, find  $m$ such that 
    
    $\forall \eta \in [\omega]^\omega[\eta(0)>m\rightarrow \exists n[\overline \eta n \in S \;\wedge\; \exists s \in [n]^{<n+1}[\overline \eta n \circ s$ $\;U$-secures $ A]]]$.

    Define $q:=\max\bigl(m,\zeta(0)\bigr)$.
    
    Now let $\rho$ in $[\omega]^\omega$ be given such that $\rho(0)>q$.
    
     We want to prove: $\exists n [\overline \rho n \in S\upharpoonright\langle \zeta(0)\rangle\;\wedge\; \exists t \in [\omega]^{<\omega}[\overline\rho n\circ t$ $\;U$-secures  $ A_\zeta]]$.
    
     Consider $\eta:=\langle \zeta(0) \rangle\ast\rho$. 
    
    Find $n,s$ such that     $\overline \eta n \in S$  and $s\in[n]^{<n+1}$ and  $\overline \eta n \circ s$ $\;
    U$-secures  $A$.

    If $s=\langle\;\rangle$,  then $\langle\;\rangle=\rho\circ\langle\;\rangle$ $\;U$-secures $A$ and also $A_\zeta$, 
    
      We thus may assume  $s\neq\langle\;\rangle$. Note:  
 $ \overline \rho(n-1)\in S\upharpoonright\langle \zeta(0)\rangle$. 

There are two cases.

\smallskip \textit{Case (ba)}. $s(0)>0$. Find $u$ such that $length(u)=length(s)$ and \\$\forall i<length(u)[u(i)=s(i)-1]$. \\Note: $\overline \rho(n-1)\circ u=\overline \eta n \circ s$ $\;U$-secures $A$ and also  $A_\zeta$. 

\smallskip \textit{Case (bb)}. $s(0)=0$. 
Find $u$ such that $s =\langle 0\rangle\ast u$. Note: $s(0)=\eta(0)=\zeta(0)$. \\ If $\neg \exists t[\overline \rho n \circ u=\zeta\circ t]$, then also $\neg \exists t[\overline \rho (n-1) \circ u=\zeta\circ t]$ and $\overline \rho(n-1)\circ u$ $\;U$-secures $A$ and also $A_\zeta$.\\
If $\exists t[\overline \rho (n-1)\circ u =\zeta\circ t]$, find $\eta$ in $[\omega]^\omega$ such that $\zeta\circ s\sqsubset \eta$ and $\forall i\exists j[\eta(i)=\zeta(j)]$, so $\eta$ is a subsequence of $\zeta$. Find $n$ such that $\overline \eta n \in \overline U$. As $\zeta\circ s$ $U$-secures $A$, we may conclude: $\overline \eta n \in A$.  Taking $t:=\overline \eta n$, we see: $\exists t[t\neq 0\;\wedge\;t(0)=0 \;\wedge\; \zeta\circ t \in A]$, i.e.  $QED_1$ and we may conclude: $\overline \rho(n-1)\circ \langle\;\rangle=\langle\;\rangle \;U$-secures $A_\zeta$. 

\smallskip

We thus see: for all  $\rho$ in $[\omega]^\omega$, if $\rho(0)>q$, then \\$\exists n [\overline \rho n \in S\upharpoonright\langle \zeta(0)\rangle\;\wedge\; \exists t \in [\omega]^{<\omega}[\overline\rho n\circ t$ $\;U$-secures  $ A_\zeta]]$, that is: \begin{center}$S\upharpoonright\langle \zeta(0)\rangle$ $U$-secures $A_\zeta$. \end{center}

 Also: \begin{center}$T$ $\;U$-secures $B$.\end{center}  Using the assumption $R(U, S\upharpoonright \langle \zeta(0)\rangle, T)$, we conclude:
$A_\zeta \cap B$ is almost-full.

In particular, we may find $s$ in $[\omega]^{<\omega}$ such that $\zeta\circ s \in A_\zeta\cap B$, and therefore, either $\zeta\circ s \in A\cap B$ or $QED_1$.

We thus see: $\forall\zeta\in [\omega]^\omega\exists s\in [\omega]^{<\omega}[\zeta\circ s\in A\cap B\;\vee\; (s\neq 0\;\wedge\;s(0)=0\;\wedge\; \zeta \circ s \in A)]$.

\bigskip Now let $\zeta$ in $[\omega]^\omega$ be given.

Define a statement \begin{center}$QED_2:=\exists s \in [\omega]^{<\omega}[\zeta\circ s \in A\cap B]$. \end{center}
Also define \begin{center}$C:=\{u\in \overline{U\upharpoonright \langle\zeta(0)\rangle}\mid\neg \exists t \in [\omega]^{<\omega}][u=\zeta\circ t]\;\vee\; \langle\zeta(0)\rangle \ast u \in A\;\vee\;QED_2\}$.\end{center} 

We want to prove that $C$ is almost-full.

First, note that, for each $\eta$ in $[\omega]^\omega$, if $\eta(0)=\zeta(0)$ {\it and $\eta$ is a subsequence of $\zeta$}, i.e.: $\forall i\exists j[\eta(i)=\zeta(j)]$,  then either $\exists s\in [\omega]^{<\omega}[\eta\circ s \in A\cap B]$, and, therefore, $QED_2$, or  $\exists s\in [\omega]^{<\omega}\setminus\{0\}[s(0)=0\;\wedge\;\eta\circ s \in A]$, and, therefore, $\exists s \in [\omega]^{<\omega}[\langle\zeta(0)\rangle\ast (\eta\circ s) \in A]$. 

We may even conclude: for \textit{every} $\eta$ in $[\omega]^\omega$ that is a subsequence of $\zeta$, either  $QED_2$, or  $\exists s \in [\omega]^{<\omega}[\langle\zeta(0)\rangle\ast (\eta\circ s) \in A]$, for, given $\eta$ such that $\eta(0)\neq \zeta(0)$ we may apply the previous result to the sequence $\eta^+:=\langle\zeta(0)\rangle\ast\eta$. 

\smallskip

We thus see that, for every $\eta$ in $[\omega]^\omega$ that is a subsequence of $\zeta$, there exists $s$ in $[\omega]^{<\omega}$ such that $\eta \circ s \in C$.

Now let $\eta$ in $[\omega]^\omega$ be given. Define $\eta_\zeta$ in $[\omega]^\omega$ as follows. 
\\If $\exists i[\eta(0)=\zeta(i)]$, define $\eta_\zeta(0)=\eta(0)$, and, if not, define $\eta_\zeta(0) =\zeta(0)$.\\ For each $n$, if $\exists t \in [\omega]^{n+2}[\overline \eta(n+2) = \zeta \circ t]$, then $\eta_\zeta(n+1)=\eta(n+1)$, and, if not, then $\eta_\zeta(n+1)=\zeta(i+1)$, where $i$ satisfies $\eta_\zeta(n)=\zeta(i)$. \\As $\eta_\zeta$ is a subsequence of $\zeta$, find $s$ in $[\omega]^{<\omega}$ such that $\eta_\zeta \circ s \in C$. \\Either $\eta\circ s = \eta_\zeta\circ s \in C$ or $\eta\circ s\neq \eta_\zeta\circ s $. In the latter case, let $n_0$ be the least $n$ such that  $\neg \exists i[\eta(n)=\zeta(i)]$.  Define $\rho$ in $[\omega]^\omega$ such that $\forall n[\rho(n)=n_0+n]$ and find $m$ such that $\overline {\eta\circ\rho}(m) \in \overline{U\upharpoonright \langle\zeta(0)\rangle}$. Define $s:=\overline \rho m$ and   note: $\overline{\eta\circ \rho} (m) = \eta \circ s$ and $\neg \exists u\in [\omega]^{<\omega}[ \eta\circ s = \zeta\circ u]$ and  $\eta\circ s \in C$. 

We thus see that $C$ is almost-full.
 
 We obtained this result using the assumption: 
   $Q(U,S\upharpoonright \langle \zeta(0) \rangle)$.
 
  \smallskip Now define \begin{center}$D:=\{u\in \overline{U\upharpoonright \langle\zeta(0)\rangle}\mid\neg \exists t \in [\omega]^{<\omega}][u=\zeta\circ t]\;\vee\; \langle\zeta(0)\rangle \ast u \in B\;\vee\;QED_2\}$.\end{center}
   
 \smallskip Using the assumption:   $R(U, S, T\upharpoonright\langle \zeta(0)\rangle)$, one may prove: $D$ is almost-full,  by an argument similar to the one that established that  $C$ is almost-full.  
   
   Using the assumption $P(U\upharpoonright\langle\zeta(0)\rangle)$, one then concludes: $C\cap D$ is almost-full.
   
   Find $s$ in $[\omega]^{<\omega}$ such that $\zeta\circ s \in C\cap D$ or $\langle \zeta(0)\rangle \ast\zeta\circ s \in A\cap B$, so, in any case, $QED_2$.

    \smallskip
    Using the assumptions: for all $n$, $P(U\upharpoonright \langle n \rangle)$ (2), and: for all $n$, $Q(U, S\upharpoonright \langle n \rangle)$ (2.2) and: for all $n$, $R(U, S, T\upharpoonright\langle n \rangle)$ (2.2.2), one thus proves: $R(U,S,T)$ , i.e.: for all $A,B\subseteq \overline U$, if $S,T$ $U$-secure $A,B$, respectively, then $A\cap B$ is almost-full, i.e.: $\forall \zeta \in{\omega}^\omega\exists s\in[\omega]^{<\omega}[\zeta \circ s \in A\cap B]$.
    
    Axiom \ref{A:stpinduction} and Corollary \ref{C:stumpsecurealmostfullU} then guarantee the conclusion: $\mathbf{CRT}$.

    \smallskip (ii) This immediately follows from (i), as, for all $k$, $\omega^{<k}$ is a stump and  \\$\overline{\omega^{<k}}=\omega^k$.\footnote{$\mathbf{IRT}$ is proven in \cite[Theorem 7.3]{veldmanbezem} from what is, in this paper, Theorem \ref{T:bimbid}, see also \cite[\S 12]{veldman20b}.  $\mathbf{IRT}(2)$ is obtained in \cite[\S 4]{veldman2004} and \cite[Theorem 9]{veldman2009} from what is, in this paper Theorem \ref{T:brthesis}.}
  \end{proof}
   
   \begin{corollary}\label{C:infrt} For all $k>0$, for all $r>0$,  for all subsets $A_0,A_1, \dots, A_{r-1}$ of $[\omega]^k$, if, for each $j<r$,  $A_j$ is almost-full, then $\bigcap_{j<r} A_j$ is almost-full. \end{corollary}
   
   \begin{proof} Straightforward, by induction. \end{proof}
   \subsection{The (extended) Finite Ramsey Theorem and \\the `compactness argument'}$\;$
   
    The following Definition extends Definition \ref{D:secures} in Subsection \ref{SS:proofclopenr}. Note that, deviating from the usual treatment of Ramsey theorems, we consider finite increasing sequences of nonnegative integers rather than finite sets of nonnegative integers. 
 \begin{definition}  
     For all $r>0$, $r^{<\omega}:=\{s\mid \forall i<length(s)[s(i)<r]\}$.

   \smallskip For all positive integers  $c,m,k,r$, \\$c:[m]^k\rightarrow r$, \emph{($c$ is an $r$-colouring of $[m]^k)$}, if and only if \\$c\in r^{<\omega}$  and $\forall s \in [m]^k[s <\mathit{length}(c)]$. 
   
   If $c:[m]^k\rightarrow r$, one may consider $c$ as an $r$-colouring of the $k$-element subsets of $m=\{0,1,\ldots,m-1\}$.

   \smallskip
   For all positive integers $k,r,n,M$, $M\rightarrow (n)^k_r$ if and only if,\\ for every $c:[M]^k\rightarrow r$, there exists    $t$ in $[M]^n$  such that \\ $ \exists j<r\forall u \in[n]^k[ c(t\circ u)=j]$, i.e. \emph{$t$ is $c$-monochromatic}. 
   
   If $M\rightarrow (n)^k_r$, then, for every $r$-colouring $c$ of the $k$-element subsets of $M=\{0,1, \ldots, M-1\}$ there exists an $n$-element subset $A=\{t(0), t(1), \ldots, t(n-1)\}$ of $M$ such that all $k$-element subsets of $A$ obtain, from $c$, one and the same colour. 
   
   \smallskip
   For all positive integers $k,r,n,M$, $M\rightarrow_{\ast} (n)^k_r$ if and only if,\\ for every $c:[M]^k\rightarrow r$,  there exist   $t,p$ such that $p\ge n$ and \\$t\in [M]^p$ and $p=t(0)$ and    $ \exists j<r\forall u \in[p]^k[ c(t\circ u)=j]$. 
   
   The meaning of this statement is similar to that of the previous one but now the $p$-element subset $A=\{t(0), t(1), \ldots t(p-1)\}$ of $M$  is required to satisfy the conditions $n\le p$ and $p= \min(A)= t(0)$. A finite set $A$ with the property $\#(A)\ge \min(A)$ is sometimes called \emph{relatively large}, see \cite{parisharrington}.

   \smallskip
   $\mathbf{FRT}$,  the \emph{Finite Ramsey Theorem}, is the statement: $$\forall k \forall r \forall n\exists M[M\rightarrow (n)^k_r].$$ 
   
   \smallskip
   $\mathbf{FRT}_{PH}$, the \emph{Extended Finite Ramsey Theorem}, is the statement: $$ \forall k \forall r \forall n\exists M[M\rightarrow_\ast (n)^k_r].$$ 
   
  \end{definition}

 Like $\mathbf{IRT}_{class}$,    $\mathbf{FRT}$ was proven in 1928 by   F.P.~Ramsey, see \cite{ramsey}.  
   
     Ramsey   thought proving  $\mathbf{IRT}_{class}$ was a useful training for someone who wanted to prove  $\mathbf{FRT}$. $\mathbf{FRT}$ was given an independent proof.
   
   Around 1950,  it was seen that a so-called \textit{compactness argument} proves $\mathbf{FRT}$  from  $\mathbf{IRT}_{class}$, see \cite{debruijnerdos} and \cite[\S 1.7 and \S 6.3]{graham}. 
   
   \smallskip In 1976,  the stronger statement $\mathbf{FRT}_{PH}$
  was formulated   by J.~Paris and L.~Harrington. Also  $\mathbf{FRT}_{PH}$ may be proven (classically) from $\mathbf{IRT}_{class}$ by a compactness argument, see \cite{parisharrington} but   $\mathbf{FRT}_{PH}$, although it may be formulated in the language of \textit{Peano Arithmetic},  can not be proven from the axioms of $PA$.\footnote{There is also no such argument in Heyting Arithmetic $HA$, the intuitionistic analog of $PA$. Every theorem of $HA$ is also a theorem of $PA$.} Paris  and Harrington  thus  found a \textit{mathematical} Incompleteness Theorem, thereby creating a stir in the   logic community. 
  
   The proof of the next Theorem is an intuitionistic version of the compactness argument deriving $\mathbf{FRT}_{PH}$ from $\mathbf{IRT}_{class}$.
   
\begin{definition} For each $r>0$, $r^\omega:=\{\chi\mid \forall i[\chi(i)<r]\}$. \end{definition}

\begin{theorem}\label{T:finiteramsey}   $\mathbf{IRT} \rightarrow \mathbf{FRT}_{PH}$.\end{theorem}

\begin{proof}\footnote{See \cite[Corollary 9.5.2]{veldmanbezem}.}
  Let $n,k,r$ be given positive integers. Let $\chi$ be an element of $r^\omega$. We want to prove: $$QED_3:=  \exists p\ge n\exists t\in[\omega]^p[t(0) = p \;\wedge \exists j <r\forall u\in [p]^k[\chi(t\circ u) =j]].$$ For each
   $j<r$, define: $A_j:=\{t \in [\omega]^k\mid\chi(t) \neq j \;\vee\; QED_3\}$.
   
   Let $j<r$ and $\zeta$ in $[\omega]^\omega$ be given. Find $p:=\max\bigl( n, \zeta(0)\bigr)$ and define $t:=\overline \zeta p$. Either: $\forall u\in [p]^k[\chi(t\circ u)=j]$ and $QED_3$, or: $\exists u \in [p]^k[\chi(t\circ u)\neq j]$. 
   
We thus see: for each $j<r$, $A_j$ is almost-full.

Using Corollary \ref{C:infrt}, conclude: $\bigcap_{j<r} A_j$ is $k$-almost-full.  

In particular, $\bigcap_{j<r} A_j$ is inhabited and, therefore, $QED_3$.

\smallskip We thus see:
$\forall \chi \in r^\omega\exists p\ge n \exists t\in[\omega]^p[(t(0) = p\;\wedge\;\exists j<r \forall u\in [p]^k[\chi(t\circ u) =j]].$

\smallskip

Now let $B$ be the set of all $c$ in $r^{<\omega}$ such that  $$\exists p\ge n \exists t\in[\omega]^p[t(0)= p\;\wedge\; \exists j<r\forall u\in [p]^k[t\circ u<\mathit{length}(c)\;\wedge\;c(t\circ u) =j]].$$

Note: $Bar_{r^\omega}(B)$ and:  $r^\omega$ is a fan. 
Using the Fan Theorem, Theorem \ref{T:bft}, find $z,m$ such that $length(z)=m$ and $\forall i<m[z(i)\in B]$ and $\forall \chi \in r^\omega\exists i<m[z(i)\sqsubset\chi]$.  

Define $N:=\max_{i<m}length\bigl(z(i)\bigr)$.

   Find $M$ such that $\forall u<N\forall j<length(u)[u(j)<M]$.  
   
   Then: $M\rightarrow_\ast (n)^k_r$. \end{proof}
   
  \section{Borel sets}
  
 \subsection{Some preparations} We introduce the important notion of a {\it spread}. 
   \begin{definition}\label{D:spread} Let $\beta$ be given. $\beta$ is a \emph{spread-law} if and only if \\$\forall s[\beta(s)=0 \leftrightarrow \exists n[\beta(s\ast\langle n \rangle)=0]]$.
   
   $\mathcal{X}\subseteq\mathcal{N}$ is a \emph{spread} if and only if there exists a spread-law $\beta$ such that \\$\mathcal{X}=\mathcal{F}_\beta:=\{\alpha\mid \forall n[\beta(\overline \alpha n)=0]\}$. 
   \end{definition}
   
   Note that  the spread-law $\beta$ in a sense {\it governs} the set $\mathcal{F}_\beta$. The law makes clear  which steps are allowed during the step-by-step construction of an element $\alpha=\alpha(0), \alpha(1), \ldots$ of $\mathcal{F}_\beta$. In his early publications, Brouwer used the term `{\it set, Menge}' for what he later called spreads. This is because the notion expresses his view in what sense a totality like for instance $\mathcal{R}$ might be called a {\it set}. 
   
   Brouwer's Continuity Principle, Axiom \ref{A:bcp}, extends to spreads:
   \begin{theorem}[The extended Continuity Principle]\label{T:extbcp} $\;$
   
   Let $\beta$ be a spread-law and let $R$ be a subset of $\mathcal{F}_\beta \times \mathbb{N}$.
   
 If $\forall \alpha \in \mathcal{F}_\beta\exists n[\alpha Rn]$, then $\forall \alpha\in\mathcal{F}_\beta\exists m \exists n \forall \beta \in \mathcal{F}_\beta[\overline \alpha m \sqsubset \beta  \rightarrow \beta R n]$.\end{theorem} 
 
 \begin{proof} Let $\beta$ be a spread-law such that $\beta(\langle\;\rangle)=0$. \\Define $\varphi:\mathcal{N}\rightarrow \mathcal{N}$ such that, for each $\alpha$, for each $n$, \\if $\beta(\overline \alpha\bigl(n+1) \bigr)=0$, then $(\varphi|\alpha)(n)=\alpha(n)$, and, \\if not, then $(\varphi|\alpha)(n)= $ the least $p$ such that $\beta\bigl(\overline{(\varphi|\alpha)}n\ast\langle p\rangle\bigr)=0]$. \\Note: $\forall \alpha[\varphi|\alpha \in  \mathcal{F}_\beta]$ and $\forall \alpha \in \mathcal{F}_\beta[\varphi|\alpha=\alpha]$. \\The function $\varphi$ is called a \textit{retraction} of $\mathcal{N}$ onto $\mathcal{F}_\beta$.
 
 \smallskip Now assume:  $\forall \alpha \in \mathcal{F}_\beta\exists n[\alpha Rn]$. Then $\forall \alpha \exists n[(\varphi|\alpha)Rn]$. Using Axiom \ref{A:bcp} conclude: $\forall \alpha \exists m\forall \beta[\overline \alpha m \sqsubset \beta\rightarrow (\varphi|\beta)Rn]$ and:  $\forall \alpha\in\mathcal{F}_\beta\exists m \exists n \forall \beta \in \mathcal{F}_\beta[\overline \alpha m \sqsubset \beta  \rightarrow \beta R n]$.\end{proof}
 
    \begin{definition} For each $\alpha$ in $\mathcal{N}$, for each $n$, we define $\alpha^n$ in $\mathcal{N}$ by: \\$\forall m[\alpha^n(m)=\alpha\bigl(2^m(2n+1)-1\bigr)]$. \\$\alpha^n$ is called \emph{the $n$-th subsequence of $\alpha$}.
    
    For all $\alpha$, for all $n,j$, we define: $\alpha^{n,j}:=(\alpha^n)^j$.  \end{definition}
   
   The next definition explains how continuous functions fom $\mathcal{N}$ to $\mathbb{N}$ and continuous functions from $\mathcal{N}$ to $\mathcal{N}$ are coded by elements of $\mathcal{N}$.
   
   \begin{definition} For each $\varphi$, we define: $\varphi:\mathcal{N}\rightarrow\mathbb{N}$ if and only if $\forall \alpha\exists n[\varphi(\overline \alpha n)\neq 0]$.

For each $\varphi$ such that $\varphi:\mathcal{N}\rightarrow\mathbb{N}$, for each $\alpha$, we let $\varphi(\alpha)$ be the number $ p$ such that $\exists n[\varphi(\overline \alpha n)=p+1]\;\wedge\;\forall i<n[\varphi(\overline \alpha i)=0]]$. 

For each $\varphi$, we define: $\varphi:\mathcal{N}\rightarrow\mathcal{N}$ if and only if $\forall n[\varphi^n:\mathcal{N}\rightarrow \mathbb{N}]$. 

For each $\varphi$ such that $\varphi:\mathcal{N}\rightarrow\mathcal{N}$, for each $\alpha$, we let $\varphi|\alpha$ be the element $ \beta$ of $\mathcal{N}$ such that $\forall n[\beta(n)=\varphi^n(\alpha)]$.\end{definition}

 The following notion plays a key r\^ole in intuitionistic {\it descriptive set theory}, the theory of Borel and projective sets.
 
\begin{definition}[Reducibility]\footnote{In classical descriptive set theory, the reducibility notion introduced here is known as {\it Wadge-reducibility}, see \cite[Section 21.E]{kechris}. One may compare this notion to the notion of \textit{many-one reducibility} from recursion theory.} 

\smallskip For all $\mathcal{X}, \mathcal{Y}\subseteq \mathcal{N}$, for all $\varphi:\mathcal{N}\rightarrow\mathcal{N}$, we define: $\varphi$ \emph{reduces $\mathcal{X}$ to $\mathcal{Y}$}, if and only if $\forall \alpha[\alpha \in \mathcal{X}\leftrightarrow \varphi|\alpha \in \mathcal{Y}]$.   

For all $\mathcal{X}, \mathcal{Y}\subseteq \mathcal{N}$, we define: $\mathcal{X}\preceq \mathcal{Y}$, $\mathcal{X}$ \emph{reduces to} $\mathcal{Y}$, if and only if there exists $\varphi:\mathcal{N}\rightarrow \mathcal{N}$ reducing $\mathcal{X}$ to $\mathcal{Y}$.

\end{definition}
If $\varphi$ reduces $\mathcal{X}$ to $\mathcal{Y}$, then, for each $\alpha$, the question \begin{center} Does $\alpha$ belong to $\mathcal{X}$? \end{center} is reduced to the question: \begin{center} Does $\varphi|\alpha$ belong tot $\mathcal{Y}$? \end{center} 
\subsection{Borel sets of finite rank}We first introduce open subsets and closed subsets of  $\mathcal{N}=\omega^\omega$. 

 \begin{definition} 
 For each $\beta$, we define $\mathcal{G}^1_\beta:=\{\alpha\mid\exists n[\beta(\overline \alpha n)\neq 0]\}$ and \\$\mathcal{F}^1_\beta:=\mathcal{N}\setminus\mathcal{G}^1_\beta=\{\alpha\mid \forall n[\beta(\overline \alpha n)=0]\}$.
 
 We define $\mathcal{E}_1:=\{\alpha\mid\exists n[\alpha(n)\neq 0]\}$ and\\
  $\mathcal{A}_1:=\mathcal{N}\setminus \mathcal{E}_1 =\{\alpha\mid \forall n[\alpha(n)=0]\}=\{\underline 0\}$.

  $\mathcal{X}\subseteq \mathcal{N}$ is \emph{open}, or: $\mathbf{\Sigma}^0_1$, if and only if $\exists \beta[\mathcal{X}=\mathcal{G}^1_\beta]$, \\and \emph{closed}, or: $\mathbf{\Pi}^0_1$, if and only if $\exists \beta[\mathcal{X}=\mathcal{F}^1_\beta]$. \end{definition}

  Note that spreads, as introduced in Definition \ref{D:spread}, are closed subsets of $\mathcal{N}$.  It is not true that every closed subset of $\mathcal{N}$ is a spread. A closed subset $\mathcal{F}$ of $\mathcal{N}$ is a spread if and only if it is {\it located}, i.e. $\exists \beta\forall s[\exists \alpha \in \mathcal{F}[s\sqsubset\alpha]\leftrightarrow \beta(s)=0]$.
  
  Note that, if $\mathcal{G}\subseteq\mathcal{N}$ is open, then its complement $\mathcal{N}\setminus\mathcal{G}$ is a closed subset of $\mathcal{N}$. It is not true that the complement of a closed subset of $\mathcal{N}$ always is an open subset of $\mathcal{N}$.

 \smallskip We now introduce Borel sets of finite rank. Note that we avoid the operation of taking the complement of a given Borel set.
 \begin{definition}
 
 \smallskip For each $n>0$, for each $\beta$,  we define, inductively, \\$\mathcal{G}^{n+1}_\beta=\bigcup_m \mathcal{F}^n_{\beta^m}$, and $\mathcal{F}_\beta^{n+1}=\bigcap_m \mathcal{F}^n_{\beta^m}$. 
  
  For each $n>0$, we define, inductively,\\ $\mathcal{E}_{n+1}:=\{\alpha \mid \exists m[\alpha^m \in \mathcal{A}_n]\}$ and $\mathcal{A}_{n+1}:=\{\alpha \mid \forall m[\alpha^m \in \mathcal{E}_n]\}$.
  
  For each $n>0$, $\mathcal{X}\subseteq \mathcal{N}$ is $\mathbf{\Sigma}^0_n$ if and only if $\exists \beta[\mathcal{X}=\mathcal{G}^n_\beta]$ and \\$\mathcal{X}$ is $\mathbf{\Pi}^0_n$ if and only if $\exists \beta[\mathcal{X}=\mathcal{F}^n_\beta]$.
 \end{definition} 
 
 $\mathcal{X}\subseteq \mathcal{N}$ belongs to $\mathbf{\Sigma}^0_{n+1}$  if and only if there exists an infinite sequence $\mathcal{Y}_0, \mathcal{Y}_1, \ldots$ of elements of $\mathbf{\Pi}^0_n$ such that $\mathcal{X}=\bigcup_m \mathcal{Y}_m$.
 
 $\mathcal{X}\subseteq \mathcal{N}$ belongs to $\mathbf{\Pi}^0_{n+1}$  if and only if there exists an infinite sequence $\mathcal{Y}_0, \mathcal{Y}_1, \ldots$ of elements of $\mathbf{\Sigma}^0_n$ such that $\mathcal{X}=\bigcap_m \mathcal{Y}_m$.

 \subsection{The fine structure of the Borel hierarchy}One may prove that the union of two opens sets is open again but it is not true that the union of two closed sets is a closed sets itself.  It is useful to introduce the following operation.

 \begin{definition} For each $n>0$, for each $\mathcal{X}\subseteq \mathcal{N}$, we define \\$\mathbb{D}^n(\mathcal{X})=\{\alpha\mid\exists i<n[\alpha^i \in \mathcal{X}]\}$. 
  \end{definition}
  
  The next Theorem, Theorem \ref{T:borelsimple}, offers an example of  a union of two closed subsets of $\mathcal{N}$ that is not an intersection of countably many open subsets of $\mathcal{N}$ and, therefore, certainly not a closed subset of $\mathcal{N}$, see item (iii).  Theorem \ref{T:borelsimple}(iv) makes it clear that there are unions of three closed sets not coinciding with any union of two closed sets, and unions of four closed sets not coinciding with any union of three closed sets, and so on.

  \begin{theorem}\label{T:borelsimple} \begin{enumerate}[\upshape (i)]\item For all $\mathcal{X}\subseteq \mathcal{N}$, $\mathcal{X}$ is $\mathbf{\Pi}^0_1$ if and only if $\mathcal{X}\preceq \mathcal{A}_1$.
   \item For all $\mathcal{X}\subseteq \mathcal{N}$, for all $n>0$, $X\preceq \mathbb{D}^n(\mathcal{A}_1)$ if and only if there exists $\mathbf{\Pi}^0_1$ sets $\mathcal{F}_0, \mathcal{F}_1, \ldots, \mathcal{F}_{n-1}$ such that $\mathcal{X}=\bigcup_{i<n}\mathcal{F}_i$.
  \item $\mathbb{D}^2(\mathcal{A}_1)$ is not $\mathbf{\Pi}^0_2$. \item For all $n>0$, $\mathbb{D}^{n+1}(\mathcal{A}_1)\npreceq \mathbb{D}^n(\mathcal{A}_1)$. \end{enumerate}\end{theorem}
  \begin{proof} The proofs of (i) and (ii) are left to the reader.

  \smallskip (iii) Define $\mathcal{T}:=\{\alpha\mid\forall m\forall n[\bigl(\alpha(m)\neq 0\;\wedge\;\alpha(n)\neq 0\bigr)\rightarrow m=n]\}$.\\ $\mathcal{T}$ is the set of all $\alpha$ that assume at most one time a value different from $0$. \\Note that $\mathcal{T}$ is a spread.

  Assume  $\mathbb{D}^2(\mathcal{A}_1)$ is $\mathbf{\Pi}^0_2$. Find $\beta$  such that $\mathbb{D}^2(\mathcal{A}_1)=\{\alpha\mid \forall m \exists n[ \beta^m(\overline \alpha n)\neq 0]\}$.

  Let $\alpha$ in $\mathcal{T}$ be given. Let $m$ be given. Define $\alpha_0, \alpha_1$ such that, for all $n$, for both  $i<2$, $(\alpha_i)^i=\underline 0$ and, for all $s$, if  $\neg\exists p[s=\langle i, p\rangle]$, then $\alpha_i(n)=\alpha(n)$.  \\Note: for both $i<2$, $\alpha^i \in \mathbb{D}^2(\mathcal{A}_1)$. Also note: $\forall n[\overline \alpha n =\overline{\alpha_0}n\;\vee\;\overline\alpha n =\overline{\alpha_1}n]$. \\Find $n_0, n_1$ such that, for both $i<2$,  $\beta^m(\overline \alpha_i n_i)\neq 0$.  Note: either $\overline \alpha_0 n_0\sqsubset \alpha$ or $\overline \alpha_1 n_1\sqsubset \alpha$, and, in both cases, $\exists n [
  \beta^m(\overline{\alpha}n)\neq 0]$. 
  
  We thus see: $\forall \alpha\in\mathcal{T}[\alpha \in \mathbb{D}^2(\mathcal{A}_1)]$. Using Theorem \ref{T:extbcp}, find $m,i$ such that $i<2$ and $\forall \alpha[\overline{\underline 0}m\sqsubset \alpha\rightarrow \alpha^i=\underline 0]$. This gives a contradiction as one may find $\alpha$ in $\mathcal{T}$ such that $\overline{\underline 0}m \sqsubset \alpha$ and $\alpha^i\neq\underline 0$. 
  
  We may conclude: $\mathbb{D}^2(\mathcal{A}_1)$ is not $\mathbf{\Pi}^0_2$. 
  
  \smallskip (iv) Again, consider $\mathcal{T}:=\{\alpha\mid\forall m\forall n[\bigl(\alpha(m)\neq 0\;\wedge\;\alpha(n)\neq 0\bigr)\rightarrow m=n]\}$.
  \\Let $n>0$ be given such that $\mathbb{D}^{n+1}(\mathcal{A}_1)\preceq \mathbb{D}^{n}(\mathcal{A}_1)$. \\ Find $\beta$ such that $\mathbb{D}^{n+1}(\mathcal{A}_1)=\{\alpha\mid\exists i<n\forall j[\beta^i(\overline \alpha j)=0]\}$. \\For each $i$, define $\mathcal{B}_i:=\{\alpha\mid\alpha^i=\underline 0\}$. Note that each $\mathcal{B}_i$ is a spread.\\ Also note: $\mathbb{D}^{n+1}(\mathcal{A}_1)= \bigcup_{i<n+1}\mathcal{B}_i$. \\Using Theorem \ref{T:extbcp}, find, for each $i<n+1$, $m_i$ and $k_i<n$ such that \\$\forall \alpha \in \mathcal{B}_i[\underline{\overline 0}m_i\sqsubset\alpha\rightarrow \forall n[\beta^{k_i}(\overline \alpha n)=0]]$. 
  
  Find $i,j< n+1$ such that $i<j$ and $k_i=k_j$. Note: for all $\alpha$ in $\mathcal{T}$, either $\overline{\underline 0}m_i\sqsubset\alpha$ or $\overline{\underline 0}m_i\sqsubset\alpha$, so, in any case: $\forall n[\beta^{k_i}(\overline\alpha n =0]$ and  $\alpha \in \mathbb{D}^{n+1}(\mathcal{A}_1)$. 
  
   We thus see: $\forall \alpha\in\mathcal{T}[\alpha \in \mathbb{D}^{n+1}(\mathcal{A}_1)]$. Using Theorem \ref{T:extbcp}, find $m,i$ such that $i<n+1$ and $\forall \alpha[\overline{\underline 0}m\sqsubset \alpha\rightarrow \alpha^i=\underline 0]$. This gives a contradiction as one may find $\alpha$ in $\mathcal{T}$ such that $\overline{\underline 0}m \sqsubset \alpha$ and $\alpha^i\neq\underline 0$. 
   \end{proof} 
   
   Theorem \ref{T:borelsimple} gives an inkling of the fine structure of the intuitionistic Borel hierarchy. In fact, even within the class $\mathbf{\Sigma}^0_2$, there are {\it uncountably many } \textit{degrees of reducibility}, see \cite[Theorems 3.5 and 3.9]{veldman2009}.  
   
   \subsection{The Borel hierarchy theorem}The next Theorem, Theorem \ref{T:borelsecondlevel}, makes a start with establishing the Borel hierarchy itself. 
   
  \begin{theorem}\label{T:borelsecondlevel} \begin{enumerate}[\upshape (i)] \item For all $n>0$, for all  $\mathcal{X}\subseteq\mathcal{N}$, $\mathcal{X}$ is $\mathbf{\Sigma}^0_n$ if and only if $\mathcal{X}\preceq \mathcal{E}_n$, and $\mathcal{X}$ is $\mathbf{\Pi}^0_n$ if and only if $\mathcal{X}\preceq \mathcal{A}_n$. 
  
  \item For every $\varphi:\mathcal{N}\rightarrow\mathcal{N}$, there exists $\alpha$ such that, \\for all $m$,  $\alpha^m\in \mathcal{E}_1\leftrightarrow (\varphi|\alpha)^m \in \mathcal{E}_1$, and, therefore, \\$\alpha \in \mathcal{A}_2 \leftrightarrow \varphi|\alpha \in \mathcal{A}_2$ and  $\alpha \in \mathcal{E}_2 \leftrightarrow \varphi|\alpha \in \mathcal{E}_2$. \item For every $\varphi:\mathcal{N}\rightarrow\mathcal{N}$,\\if $\forall \alpha[\alpha \in \mathcal{E}_2\rightarrow \varphi|\alpha \in \mathcal{A}_2]$, then $\exists \alpha[\alpha \in \mathcal{A}_2\;\wedge\;\varphi|\alpha\in \mathcal{A}_2]$. \item For every $\varphi:\mathcal{N}\rightarrow\mathcal{N}$, if $\forall \alpha[\alpha \in \mathcal{A}_2\rightarrow \varphi|\alpha \in \mathcal{E}_2]$, then $\exists \alpha[\alpha \in \mathcal{E}_2\;\wedge\;\varphi|\alpha\in \mathcal{E}_2]$. \end{enumerate} \end{theorem}\begin{proof}  (i) The proof is left to the reader.
  
  \smallskip (ii) Let $\varphi:\mathcal{N}\rightarrow\mathcal{N}$ be given. Define $\alpha$, inductively,  as follows. \\Let $p$ be given such that $\alpha(0), \alpha(1), \ldots, \alpha(p-1)$ have been defined already.\\ Find $m,n$ such that such that $p= 2^n(2m+1)-1$. \\ Define $\alpha(p)=\alpha^m(n):=1$ {\it if} $\exists q<p[\varphi^m(\overline \alpha q) >1\;\wedge\;\forall i<q[\varphi^m(\overline \alpha i)=0]]$,\\ and $\alpha(p):=0$ {\it if not}.
  
    Note that, for all $m$, $\exists n[\alpha^m(n)\neq 0] \leftrightarrow \exists n[(\varphi|\alpha)^m (n)\neq 0]$, and, therefore, \\ $\alpha^m\in\mathcal{E}_1\leftrightarrow (\varphi|\alpha)^m \in \mathcal{E}_1$ and $\alpha^m\in\mathcal{A}_1\leftrightarrow (\varphi|\alpha)^m\in \mathcal{A}_1$. 
  
  \smallskip (iii) Let $\varphi:\mathcal{N}\rightarrow\mathcal{N}$ be given such that $\forall \alpha[\alpha \in \mathcal{E}_2\rightarrow\varphi|\alpha \in \mathcal{A}_2]$. \\Using (i), find $\alpha$ such that $\forall m[\alpha^m\in \mathcal{E}_1\leftrightarrow (\varphi|\alpha)^m \in \mathcal{E}_1]$. \\Let $m$ be given. \\Define $\alpha_0$ such that $(\alpha_0)^m=\underline 0$ and, for all $n\neq m$, $(\alpha_0)^n=\alpha^n$. \\Note: $\alpha_0\in \mathcal{E}_2$, and, therefore,   $\varphi|\alpha_0 \in \mathcal{A}_2$ and $(\varphi|\alpha_0)^m\in \mathcal{E}_1$. \\Find $n$ such that $(\varphi|\alpha_0)^m(n)\neq 0$.  \\{\it Either} $(\varphi|\alpha)^m(n)\neq 0$ and $(\varphi|\alpha)^m \in \mathcal{E}_1$ and also $\alpha^m \in \mathcal{E}_1$, \\\textit{or} $(\varphi|\alpha_0)^m(n)=0 $ and $\varphi|\alpha \perp \varphi|\alpha_0$ and $\alpha\perp\alpha_0$ and $\alpha^m\perp \underline 0$ and $\alpha^m \in \mathcal{E}_1$. 
  \\We thus see: $\forall m[\alpha^m \in \mathcal{E}_1]$. Therefore,   both $\alpha$ and $\varphi|\alpha$ are in $\mathcal{A}_2$. 
  
  \smallskip (iv) The following observation is easy but crucial:
  
  \begin{center} $\forall \alpha[ \exists \sigma \forall m[\alpha^m\bigl(\sigma(m)\bigr)\neq 0]\rightarrow \alpha \in \mathcal{A}_2]$. \end{center}

  For each $\alpha$, for each $\sigma$, we let $\sigma\Join\alpha$ be the element of  $\mathcal{N}$ such that, for each $m$,  $(\sigma\Join\alpha)^m\bigl(\sigma(m)\bigr)=\max\bigl(1, \alpha^m\bigl(\sigma(m))\bigr)$ and, for all $n \neq \sigma(m)$,\\ $(\sigma\Join\alpha)^m(n)= \alpha^m(n)]$. $\sigma\Join\alpha$ may be thought of as `$\alpha$\textit{-corrected-by-$\sigma$}'.\\Note: $\forall \alpha[ \exists \sigma[\alpha=\sigma\Join\alpha]\rightarrow \alpha \in \mathcal{A}_2]$. 
  
  \smallskip Now let $\varphi:\mathcal{N}\rightarrow\mathcal{N}$ be given such that $\forall \alpha[\alpha \in \mathcal{A}_2\rightarrow\varphi|\alpha \in \mathcal{E}_2]$.
  \\ Conclude: $\forall \alpha\exists n[\bigl((\varphi|(\alpha^0\Join\alpha^1)\bigr)^n=\underline 0]$. 
  \\Using Axiom \ref{A:bcp}, find $m,n$ such that $\forall \alpha[\underline{\overline0}m\sqsubset\alpha \rightarrow \bigl((\varphi|(\alpha^0\Join\alpha^1)\bigr)^n=\underline 0]$.
  \\Define $\beta$ in $\mathcal{C}$ such that $\forall i\forall j [\beta^i(j)=1\leftrightarrow (i<m \;\wedge\;j=0)]$. 
    \\Let $j$ be given.  Find $p$ such that $\varphi^{n,j}(\overline \beta p) \neq 0$ and  $\forall i<p[\varphi^{n,j}(\overline \beta i)=0]$.
 \\Find $\alpha$ such that $\overline{\underline 0}m\sqsubset \alpha$ and $\overline \beta p\sqsubset \alpha^0 \Join \alpha^1$. 
  \\Conclude: $(\varphi|\beta)^n(j)= \bigl(\varphi|(\alpha^0\Join\alpha^1)\bigr)^n(j) =0$ and $\varphi^{n,j}(\overline\beta p)=1$. 
  \\We thus see:    $\forall j[(\varphi|\beta)^n(j)=0]$ and $(\varphi|\beta)^n=\underline 0$ and $\varphi|\beta \in \mathcal{E}_2$. 
  \\We thus found $\beta$ such that both $\beta$ and $\varphi|\beta$ are in $\mathcal{E}_2$.
  \end{proof} Note that the proofs of Theorem \ref{T:borelsecondlevel}(ii) and (iii) are elementary in the sense that they do not use intuitionistic axioms like the Continuity Principle or the Fan Theorem. 
   The constructive argument for Theorem \ref{T:borelsecondlevel}(ii) extends to a constructive argument establishing \begin{quote} {\it For each $n>0$, for each $\varphi:\mathcal{N}\rightarrow \mathcal{N}$, there exists $\alpha$ such that \\$\alpha\in \mathcal{E}_n\leftrightarrow \varphi|\alpha \in \mathcal{E}_n$ and $\alpha\in \mathcal{A}_n\leftrightarrow \varphi|\alpha \in \mathcal{A}_n$.}\end{quote}
  Unfortunately, one can't conclude from this, constructively, that $\mathcal{A}_n$ does not reduce to $\mathcal{E}_n$. Assuming that $\varphi:\mathcal{N}\rightarrow\mathcal{N}$ reduces $\mathcal{A}_n$ to $\mathcal{E}_n$, one finds, using Theorem \ref{T:borelsecondlevel}(ii), $\alpha$ such that $\alpha \in \mathcal{A}_n\leftrightarrow \alpha \in \mathcal{E}_n$. If $n>1$, one can't derive a contradiction from this statement.  
  
  The intervention of the Continuity Principle in the proof of Theorem \ref{T:borelsecondlevel}(iv) is crucial. One may extend this argument to a proof of: \begin{quote} {\it For each $n>0$, for each $\varphi:\mathcal{N}\rightarrow \mathcal{N}$, \\if $\forall \alpha[\alpha \in \mathcal{E}_n\rightarrow \varphi|\alpha \in \mathcal{A}_n]$, then $\exists \alpha[\alpha \in \mathcal{A}_n\;\wedge\;\varphi|\alpha\in \mathcal{A}_n]$,  and, \\  if $\forall \alpha[\alpha \in \mathcal{A}_n\rightarrow \varphi|\alpha \in \mathcal{E}_n]$, then $\exists \alpha[\alpha \in \mathcal{E}_n\;\wedge\;\varphi|\alpha\in \mathcal{E}_n]$. }\end{quote}
  
  This shows that $\mathcal{E}_n$ \textit{positively refuses to reduce to} $\mathcal{A}_n$, and $\mathcal{A}_n$ positively refuses to reduce to $\mathcal{E}_n$.
  
  The theorem extends into the transfinite and then may be called the \textit{Intuitionistic Borel Hierarchy Theorem}, see \cite[Theorems 7.9 and 7.10]{veldman2008}.

  \section{Notions of finiteness}\label{S:finiteness} \subsection{Some examples} We  study decidable subsets of the set $\mathbb{N}$.
  
  Such sets may be called `finite' in various constructively different ways.
  
  \begin{definition}For every $\alpha$, we define $D_\alpha:=\{n\mid\alpha(n)\neq 0\}$.
  
  $D_\alpha$ is called \textit{the subset of $\mathbb{N}$ \emph{decided by $\alpha$}}.

  \smallskip $\mathbf{Fin}:=\{\alpha\mid\exists n\forall m>n[\alpha(m)=0]\}$.
  
  $D_\alpha$ is \emph{finite} if and only if $\alpha\in \mathbf{Fin}$.\end{definition}
  
  Note: $\alpha\in\mathbf{Fin}$  if and only if we can calculate the number of elements of $D_\alpha$. 
  
   \begin{definition} For every $\mathcal{X}\subseteq \mathcal{N}$, we define:\\ $\mathcal{X}^+:=\{\alpha\mid\exists n\forall m>n \forall n[\alpha(m)\neq 0\rightarrow \alpha\in \mathcal{X}]\}$.
   
   $D_\alpha$ is \emph{perhaps-finite}\footnote{Cf. Subsection \ref{SS:mmp}.} if and only if $\alpha \in \mathbf{Fin}^+$,

   $D_\alpha$ is \emph{perhaps-perhaps-finite} if and only if $\alpha \in \mathbf{Fin}^{++}=(\mathbf{Fin}^+)^+$, \end{definition}
 
  Some examples are useful. 
  
  \smallskip 1. Define $\alpha$ such that $\forall n[\alpha(n)\neq 0\leftrightarrow n=k_{99}]$. Then $\{k_{99}\}=\{n\mid n=k_{99}\}$ is a decidable subset of $\mathbb{N}$. The statement `$\alpha \in \mathbf{Fin}$' is equivalent to `$\exists n[n=k_{99}]$ or $\forall n[n<k_{99}]$' and thus is \textit{reckless} or \textit{hardy}. We do know $D_\alpha$ has at most one element, but we do not know if the number of elements of $D_\alpha$ is $0$ or $1$.
  
  On the other hand, $\alpha\in \mathbf{Fin}^+$. For assume we find $m$ such that $\alpha(m)\neq 0$, Then $m=k_{99}$ and $D_\alpha=\{m\}$ is finite. Therefore: $\forall m[\alpha(m)\neq 0
  \rightarrow \alpha\in \mathbf{Fin}]$ and: $\alpha \in \mathbf{Fin}^+$.
  
  \smallskip 2. Define $\alpha$ such that $\forall n[\alpha(n)\neq 0\leftrightarrow k_{99}\le n <2\cdot k_{99}]$. Again, the statement `$\alpha \in \mathbf{Fin}$' is reckless, and again, $D_\alpha$ is perhaps-finite, although, unlike in the case of example 1, we are unable to give an upper bound for the number of elements of $D_\alpha$.
  
  \smallskip 3. Let $\gamma,\delta$ be given and define $\alpha$ such that $\forall n[\alpha(n)\neq 0\leftrightarrow (n=k_\gamma\;\vee\;n=k_\delta)]$, so $D_\alpha=\{k_\gamma, k_\delta\}$. The reader may find out herself that $D_\alpha$ is perhaps-perhaps-finite and that the statement: `$D_\alpha$ is perhaps-finite' may be \textit{reckless}.
  
 \subsection{Extension into the transfinite}The process of taking so-called `\textit{perhapsive extensions}' of the set $\mathbf{Fin}$ may be continued into the transfinite.

   \begin{definition} We define a collection $\mathcal{E}$ of subsets of $\mathcal{N}$ by means of the following inductive definition. \begin{enumerate}[\upshape (i)] \item $\mathbf{Fin} \in \mathcal{E}$. \item For every $\mathcal{X}$ in  $\mathcal{E}$, also $\mathcal{X}^+\in \mathcal{E}$. \item For every infinite sequence $\mathcal{X}_0, \mathcal{X}_1, \ldots$ of elements of $\mathcal{E}$, such that, for each $n$, $(\mathcal{X}_n)^+\subseteq \mathcal{X}_{n+1}$, also $\bigcup_n\mathcal{X}_n \in \mathcal{E}$. \item Clauses (i), (ii), (iii) produce all elements of $\mathcal{E}$. 
   \end{enumerate}
     \end{definition}
    
    \begin{definition}\smallskip  
  
   For every $\mathcal{X}\subseteq\mathcal{N}$, $\mathcal{X}^\neg:=\mathcal{N}\setminus\mathcal{X}:=\{\alpha\mid \neg(\alpha \in \mathcal{X})\}$ and \\$\mathcal{X}^{\neg\neg}=(\mathcal{X}^\neg)^\neg$. \end{definition}
    
     \begin{theorem}\label{T:finiteness}\begin{enumerate}[\upshape (i)] \item For all $\mathcal{X}$ in $\mathcal{E}$, $\mathbf{Fin}\subseteq \mathcal{X}\subseteq \mathbf{Fin}^{\neg\neg}$. \item For all $\mathcal{X}$ in $\mathcal{E}$, for all $\alpha$, $\forall s[\alpha \in \mathcal{X}\leftrightarrow s\ast \alpha \in \mathcal{X}]$. \item For all $\mathcal{X}$ in $\mathcal{E}$, $\mathcal{X}\subsetneq\mathcal{X}^+$. \end{enumerate} \end{theorem} \begin{proof}\footnote{See \cite{veldman1995}, \cite{veldman1999} and \cite{veldman2005}.} (i) The proof is by induction, following the definition of the class $\mathcal{E}$. 
     
   \smallskip  1. Obviously, $\mathbf{Fin}\subseteq \mathbf{Fin}^{\neg\neg}$, see Subsection \ref{SS:useful}. 
     
  \smallskip   2. Assume $\mathbf{Fin}\subseteq \mathcal{X}\subseteq\mathbf{Fin}^{\neg\neg}$. Assume $\alpha \in \mathcal{X}^+$. \\
     Find $m$ such that $\forall n>m[\alpha(n)\neq 0 \rightarrow \alpha \in \mathcal{X}$ and distinguish two cases: 
     
     \textit{Case (a)}. $\exists n>m[\alpha(n)\neq 0]$. Then $\alpha \in \mathcal{X}$ and, therefore: $\alpha \in \mathbf{Fin}^{\neg\neg}$.
     
     \textit{Case (b)}. $\neg \exists n>m[\alpha(n) \neq 0]$. Then $\forall n>m[\alpha(n)=0]$, so $\alpha \in \mathbf{Fin}\subseteq\mathbf{Fin}^{\neg\neg}$. 
     
     We thus see: if $\exists n>m[\alpha(n)\neq 0]\;\vee\;\neg\exists n>m[\alpha(n)\neq 0]$, then $\alpha \in \mathbf{Fin}^{\neg\neg}$. 
     
      We may conclude: $\alpha \in \mathbf{Fin}^{\neg\neg}$, see Subsection \ref{SS:useful}. 
      
      We thus see: $\forall \alpha \in \mathcal{X}^+[\alpha \in \mathbf{Fin}^{\neg\neg}]$, and conclude: $\mathbf{Fin}\subseteq\mathcal{X}\subseteq\mathcal{X}^+\subseteq \mathbf{Fin}^{\neg\neg}$. 
      
    \smallskip  3. Assume: for all $n$, $\mathbf{Fin}\subseteq \mathcal{X}_n\subseteq\mathbf{Fin}^{\neg\neg}$. Clearly, then $\mathbf{Fin}\subseteq \bigcup_n \mathcal{X}_n\subseteq\mathbf{Fin}^{\neg\neg}$.

   \medskip (ii) We again use induction on the definition of the class $\mathcal{E}$. 
   
   \smallskip 1.  Note: $\forall \alpha\forall s[\alpha \in \mathbf{Fin}\leftrightarrow s\ast\alpha\in \mathbf{Fin}]$. 
   
   \smallskip 2. Let $\mathcal{X}$ in $\mathcal{E}$ be given such that $\forall \alpha\forall s[\alpha \in \mathcal{X}\leftrightarrow s\ast\alpha \in \mathcal{X}]$. Let $s,\alpha$ be given. If $s\ast\alpha \in \mathcal{X}^+$, find $m$ such that $\forall n>m[s\ast \alpha(n)\neq 0\rightarrow s\ast\alpha \in \mathcal{X}]$ and conclude: for all $n>m$, if $\alpha(n)\neq 0$ then $length(s)+n >m$ and $s\ast \alpha \in \mathcal{X}$ and $\alpha \in \mathcal{X}$, i.e. $\alpha \in \mathcal{X}^+$. Conversely, if $\alpha \in \mathcal{X}^+$, find $m$ such that $\forall n>m[\alpha(n)\neq 0\rightarrow \alpha \in \mathcal{X}]$ and conclude: for all $n>m+length(s)$, if $s\ast\alpha(n)\neq 0$, then $\alpha\bigl(n-length(s)\bigr)\neq 0$ and $n-length(s)>m$ and $\alpha\in \mathcal{X}$ and $s\ast \alpha \in \mathcal{X}$, i.e. $s\ast\alpha \in \mathcal{X}^+$.
   
   \smallskip 3. Assume: for all $n$, $\forall \alpha\forall s[\alpha\in \mathcal{X}_n \leftrightarrow s\ast\alpha\in \mathcal{X}_n]$. \\Clearly, then $\forall \alpha\forall s[\alpha\in \bigcup_n\mathcal{X}_n \leftrightarrow s\ast\alpha\in \bigcup_n\mathcal{X}_n$.
   
  \medskip  (iii)  We shall  prove: for all $\mathcal{X}$ in $\mathcal{E}$ there exists a spread $\mathcal{F}$ such that $\mathcal{F}\subseteq \mathcal{X}^+$ and not $\mathcal{F}\subseteq \mathcal{X}$. The proof is by induction, following the definition of the class $\mathcal{E}$.
    
    \smallskip 1. Consider $\mathcal{T}=\{\alpha\mid\forall m\forall n[\bigl(\alpha(m)\neq 0\;\wedge\;\alpha(n)\neq 0\bigr)\rightarrow m=n]\}$.
 \\Note that $\mathcal{T}$ is a spread. Note that, for every $\alpha$ in $\mathcal{T}$, for every $n$, if $\alpha(n)\neq 0$, then $\forall m>n[\alpha(m)=0]$ and $\alpha \in \mathbf{Fin}$. We thus see: $\mathcal{T}\subseteq \mathbf{Fin}^+$
 
  Now assume $\mathcal{T}\subseteq \mathbf{Fin}$. Then $\forall \alpha \in \mathcal{T}\exists n\forall m>n[\alpha(m)=0]]$. \\Using Theorem \ref{T:extbcp}, find $p,n$ such that $\forall \alpha \in \mathcal{T}[\overline{\underline 0}p\sqsubset\alpha\rightarrow \forall m>n[\alpha(m)=0]]$. \\ Conclude: $n+1\ge p$ and consider $\alpha:=\overline{\underline 0} (n+1)\ast\langle 1 \rangle\ast \underline 0$.
  \\Note: $\alpha \in \mathcal{T}$,  $\overline{\underline 0}p\sqsubset \alpha$ and $\alpha(n+1)\neq 0$.\\ Clearly, we reached a contradiction. 
  
  \smallskip 2. Let $\mathcal{X}$ in $\mathcal{E}$ be given and let $\mathcal{F}$ be a spread such that $\mathcal{F}\subseteq \mathcal{X}^+$ and not $\mathcal{F}\subseteq \mathcal{X}$. Define $\mathcal{F}^\ast:=\{\alpha\mid \forall n \forall p\forall \beta[\alpha=\overline{\underline 0}n\ast\langle p+1\rangle\ast \beta \rightarrow \beta \in \mathcal{F}]\}$. \\Note: $\mathcal{F}^\ast$ is a spread and $\underline 0 \in \mathcal{F}^\ast$.  \\Note that, for every $\alpha$ in $\mathcal{F}^\ast$, for every $n$, if $\underline{\overline 0}n\sqsubset\alpha$ and $\alpha(n)\neq 0$, then there exists $\beta$ in $\mathcal{F}$ such that $\alpha=\overline\alpha(n+1)\ast\beta$, and by (ii), also $\alpha \in \mathcal{F}$ and $\alpha\in \mathcal{X}^+$. \\We thus see: $\mathcal{F}^\ast\subseteq \mathcal{X}^{++}$.
  
  Now assume $\mathcal{F}^\ast \subseteq \mathcal{X}^+$. \\Then $\forall \alpha \in \mathcal{F}^\ast\exists n\forall m>n[ \alpha(m)\neq 0\rightarrow \alpha \in \mathcal{X}]$. Using Theorem \ref{T:extbcp}, \\find $p,n$ such that $\forall \alpha \in \mathcal{F}^\ast[\overline{\underline 0}p\sqsubset\alpha\rightarrow \forall m>n[\alpha(m)\neq 0\rightarrow\alpha \in \mathcal{X}]]$.\\We may assume $n+1>p$. Note that, for each $\beta$ in $\mathcal{F}$, $\overline{\underline 0}n\ast\langle 1\rangle \ast\beta \in \mathcal{F}^\ast$ and: $\overline{\underline 0}n\ast\langle 1\rangle\ast \beta \in \mathcal{X}$, and by (ii), $\beta \in \mathcal{X}$. \\We thus see: $\mathcal{F}\subseteq \mathcal{X}$ and we know that this leads to a contradiction.\\Conclude: not: $\mathcal{F}\subseteq \mathcal{X}^+$.
  
  \smallskip 3. Let $\mathcal{X}_0, \mathcal{X}_1,\ldots$ be an infinite sequence of elements of $\mathcal{E}$ such that, for each $n$, $(\mathcal{X}_n)^+\subseteq \mathcal{X}_{n+1}$, and let $\mathcal{F}_0, \mathcal{F}_1, \ldots$ be an infinite sequence of spreads such that, for each $n$, $\mathcal{F}_n\subseteq (\mathcal{X}_n)^+\subseteq \mathcal{X}_{n+1}$  and not $\mathcal{F}_n\subseteq \mathcal{X}_n$. \\Define $\mathcal{F}^\ast:=\{\alpha\mid \forall n\forall p\forall \beta[\alpha=\underline{\overline 0}n\ast\langle p+1\rangle\ast \beta\rightarrow \beta \in \mathcal{F}_n\}$. \\ Note: $\mathcal{F}^\ast$ is a spread and $\underline 0\in \mathcal{F}^\ast$. \\Note that, for every $\alpha $ in $\mathcal{F}^\ast$, for every $n$, if $\overline{\underline 0}n\sqsubset\alpha$ en $\alpha(n)\neq 0$, then there exists $\beta$ in $\mathcal{F}_n\subseteq \mathcal{X}_{n+1}$ such that $\alpha=\overline \alpha(n+1)\ast\beta$, and by (ii), also $\alpha \in \mathcal{X}_{n+1}\subseteq \bigcup_i\mathcal{X}_i$.\\We thus see: $\mathcal{F}^\ast\subseteq (\bigcup_i\mathcal{X}_i)^+$. 
  
  Now assume: $\mathcal{F}^\ast\subseteq \bigcup_i\mathcal{X}_i$. Using Theorem \ref{T:extbcp}, find $p,i$ such that\\ $\forall \alpha \in \mathcal{F}^\ast[\overline{\underline 0}p\sqsubset \alpha\rightarrow \alpha \in \mathcal{X}_i]$. Define $n:=\max(p,i)$ and note: for all $\beta$ in $\mathcal{F}_n$, $\overline{\underline 0}n\ast\langle 1 \rangle\ast \beta \in \mathcal{F}^\ast$, so  $\overline{\underline 0}n\ast\langle 1 \rangle\ast \beta \in \mathcal{X}_i$, and, by (ii), also $\beta \in \mathcal{X}_i$.\\ We thus see: $\mathcal{F}_n \subseteq \mathcal{X}_i\subseteq \mathcal{X}_n$. Contradiction.
    \end{proof}
    
    Theorem \ref{T:finiteness} shows the expressive power of the language of intuitionistic mathematics. 
    In  Theorem \ref{T:finiteness}(i) the set $\mathbf{Fin}^{\neg\neg}$ might be replaced by the set \begin{center} $\mathbf{AlmostFin}:=\{\alpha\mid\forall\zeta\in [\omega]^\omega\exists n[\alpha\circ\zeta(n)=0]\}$, \end{center}
 as one may prove: $\bigcup \mathcal{E}\subseteq \mathbf{AlmostFin}$. \\Using Theorem \ref{T:brthesis}, one may prove: $\mathbf{AlmostFin}\subseteq \bigcup\mathcal{E}\subseteq \mathbf{Fin}^{\neg\neg}$.
 
 The set $\mathbf{AlmostFin}$ is an example of a `simple' 
 $\mathbf{\Pi}^1_1$  or \textit{co-analytic} set that fails to be positively Borel, see \cite[\S 3]{veldman2005}.  
 
 \section{Avoiding Brouwer's `axioms'.}  
    \subsection{Bishop-style constructivism} E. Bishop (1928-1983) started his own school of constructive mathematics in the 1960s, see \cite{bishop} and \cite{bishopbridges}. Although Bishop admired Brouwer for his criticism of the non-constructive nature of much of mathematics,  and for his heroic attempt to do something about it, he  also had strong hesitations about Brouwer's work. He did not accept Brouwer's adoption of the Continuity Principle and the Fan Theorem. He thought the arguments in favour of these principles bizarre, metaphysical and  mystical and judged that Brouwer, by bringing them in, had spoiled his own good cause.  He declared Brouwer's intuitionism to be dead and wanted a new beginning.
    
    Nevertheless, developing his constructive mathematics, he felt the need for some of the consequences of Brouwer's principles. 
    
    Bishop's course was the following.  He declared the notion of `{\it pointwise continuity}' of real functions to be `{\it irrelevant}'.  There is no need then for arguments as given in Section \ref{S:bcp}. He \textit{defined} a function from $\mathcal{R}$ to $\mathcal{R}$ to be continuous if and only if it is uniformly continuous on every closed and bounded subinterval of $\mathcal{R}$.   He thus buys for nothing what was for Brouwer the main consequence of the Fan Theorem, Theorem \ref{T:uc}. There is no need then for arguments `proving' the Fan Theorem as given in Section \ref{S:fantheorem}.\footnote{Bishop's definition is questionable. F. Waaldijk discovered that the statement that the composition $f\circ g$ of Bishop-continuous functions $f,g$ always is Bishop-continuous itself is equivalent to  the Fan Theorem, see \cite[Corollary 9.10]{veldman2011b}.} 
    
    Bishop's attitude might perhaps be called `{\it pragmatic}' in view of Brouwer's more deeply going analysis of mathematical thinking.  The principles proposed by Brouwer, even if one does not want to subscribe to the way he defends them, deserve to be discussed as possible starting points for our common mathematical discourse. Not doing so, Bishop denied himself the possibility of a further going perspective. 
    
    Bishop calls his own attitude {\it realistic}, contrasting it with the  attitude of the `classical' mathematician that uses non-constructive arguments. Unlike Brouwer,  he does not fiercely attack the classical mathematician but calls him `{\it idealistic}'. In view of Subsection \ref{SS:notoplato}, where the platonists were dubbed `realists', this terminology is funny. 
    
    \subsection{Martin-L\"of's `constructive mathematics'} P. Martin-L\"of (1942-$\;$) made a serious study of Brouwer's work and came to the following view in  \cite{martinloef}. Every object in constructive mathematics should be given by a `\textit{a finite configuration of signs}'.  Therefore, an infinite sequence $\alpha$ of natural numbers always should be given by a (finitely given) algorithm.  Creating $\alpha$ by choosing its values $\alpha(0), \alpha(1), \ldots$ one-by-one in an infinite process of free choices, as Brouwer wanted to do, is out of the question.   As we saw in Subsection \ref{SS:failurefan}, the Fan Theorem fails in such a context. Nevertheless, Martin-L\"of formulates and upholds a `Fan Theorem'. His strategy is similar to Bishop's as sketched in the previous Subsection. He \textit{redefines} the meaning of the statement `$\forall \alpha \in \mathcal{C}\exists n[\overline \alpha n \in B]$'.  Presenting his definition somewhat freely, one might say: he proposes to {\it define} `$\forall \alpha \in \mathcal{C}\exists s \in B[s\sqsubset \alpha]$' as: `there exists a canonical proof of `$\langle\;\rangle$ is $B$-secure' as we explained this in Section \ref{S:fantheorem}. Martin-L\"of extends this strategy to the Bar Theorem, see the proof of Theorem \ref{T:bimbid}.
    
    \subsection{An application} The Fan Theorem and Brouwer's Continuity Principle are used for a reconstruction of Cantor's Uniqueness Theorem in \cite{veldman18}.


\begin{thebibliography}{50} \bibitem{brouwer08b} M. van Atten and G. Sundholm, L.E.J. Brouwer's ‘Unreliability of the Logical Principles’: A New Translation, with an Introduction,
\textit{History and Philosophy of Logic} 38(2017)24-47.
\bibitem{benacerrafputnam} P. Benacerraf and H. Putnam, {\it Philosophy of Mathematics, selected readings}, Prentice-Hall, Inc., London etc., 1964. 

\bibitem{bishop} E. Bishop, {\it Foundations of Constructive Analysis}, McGraw-Hill, New York, 1967.

\bibitem{bishopbridges} E. Bishop and D. Bridges, \textit{Constructive Analysis}, Grundlehren der mathematischen Wissenschaften vol. 279, Springer-Verlag, Berlin, 1985.
\bibitem{bolzano} B. Bolzano, {\it  Rein analytischer Beweis des Lehrsatzes, dass zwischen je zwey Werthen, die ein entgegengesetztes Resultat gewähren, wenigstens eine reelle Wurzel der Gleichung liege}, Prague, 1817.

\bibitem{borel}\'E. Borel, Sur quelques points de la th\'eorie des fonctions, {\it Annales Scientifiques de l'\'Ecole Normale Sup\'erieure} (3)12 (1895),  pp. 9-55, also in \cite[pp. 239-287]{borelcw}.
\bibitem{borelcw} \textit{Oeuvres de \'Emile Borel}, Tome 1, \'Editions du Centre National de Recherche Scientifique, Paris, 1972. 
\bibitem{bridgesrichman} D. Bridges and F. Richman, \textit{Varieties of Constructive Mathematics}, Cambridge Universiry press, Cambridge 1987.

\bibitem{brouwer07} L.E.J.~Brouwer, {\it Over de grondslagen der wiskunde (On the foundations of mathematics)}, Thesis, Amsterdam, 1907, 183 p., also in \cite{brouwer75}, pp. 13-101. \bibitem{brouwer08} L.E.J. Brouwer, De onbetrouwbaarheid der logische principes, (The unreliability of logical principles), \textit{Tijdschrift voor Wijsbegeerte}, 2(1908)152-158, also in \cite{brouwer75}, pp.107-111, and \cite{brouwer08b}.
\bibitem{brouwer12} L.E.J.~Brouwer, Intuitionism and formalism, {\it Bull. Am. Math. Soc.} 20(1913)81-96, also in \cite{brouwer75}, pp. 123-138.
\bibitem{brouwer27} L.E.J.~Brouwer, \"Uber Definitionsbereiche von Funktionen, \textit{Math. Annalen} 97(1927)60-75, also in \cite{brouwer75}, pp. 390-405.
\bibitem{brouwer48a} L.E.J. Brouwer, Essentieel negatieve eigenschappen, {\it Proc. Akad. Amsterdam} 51(1948)963-965 = {\it Indagationes Mathematicae} 10(1948)322-324, transl. Essentially negative properties, in \cite{brouwer75}, pp. 478-479.
\bibitem{brouwer48} L.E.J. Brouwer, Consciousness, Philosophy and Mathematics, \textit{Proc. 10th International Congress of Philosophy}, Amsterdam 1948, pp. 1235-1249, also in \cite{brouwer75}, pp. 480-494.

\bibitem{brouwer51} L.E.J. Brouwer, An intuitionist correction of the fixed-point theorem on the sphere, \textit{Proc. Roy. Soc. London Ser. A }213(1951)1-2, also in \cite{brouwer75}, pp. 506-507.
\bibitem{brouwer54} L.E.J. Brouwer, Points and spaces, \textit{Can. J. Math.} 6(1954)1-17, also in \cite{brouwer75}, pp. 522-538.
\bibitem{brouwer75} L.E.J.~Brouwer, \textit{Collected Works, Vol. I: Philosophy and
 Foundations of Mathematics}, ed.~A.~Heyting, North Holland Publ. Co., Amsterdam,
 1975. \bibitem{debruijnerdos} N.G.~de Bruijn and P.~Erd\"os, A Color Problem for Infinite Graphs and a Problem in the Theory of Relations,  \textit{Nederl. Akad. Wetensch. Proc. Ser. A} 54(1951)371-373. 
 \bibitem{cantor}G. Cantor, \"Uber eine Eigenschaft des Inbegriffes aller reellen algebraischen Zahlen, \textit{Crelles Journal f. Mathematik}, 77(1874) pp. 258-262, also in: \cite{cantorcw}, pp. 115-118. \bibitem{cantorcw}G. Cantor, \textit{Gesammelte Abhandlungen mathematischen und philosophischen Inhalts, herausgegeben von E. Zermelo}, Springer, Berlin 1932, Reprint 1980.     

\bibitem{clote} P.~Clote, A Recursion Theoretic Analysis of the Clopen Ramsey Theorem, \textit{J. Symbolic Logic} 49(1984)376-400.
\bibitem{cox} D. Cox, J. Little, D. O'Shea, \textit{Ideals, Varieties and Algorithms, an Introduction to Computational Algebraic Geometry and Commutative Algebra}, Springer Verlag, New York, 1992. 
\bibitem{dedekind} R. Dedekind, \textit{Was sind und was sollen die Zahlen?} Vieweg, Braunschweig, 1888.
\bibitem{dickson} L.E. Dickson, Finiteness of the odd perfect and primitive abundant numbers with $n$ distinct factors, \textit{Am. J. Math.} 35(1913)413-126.
 \bibitem{euclid} Euclid, {\it The thirteen books of the Elements, translated with introduction and commentary by Sir Thomas L. Heath}, Vol. II, Books III-IX, Dover Publications, New York, 1956. 
 
 \bibitem{fraisse}R.A. Fra\"iss\'e, \textit{Theory of relations}, North Holland Publ. Co. , Amsterdam, 1986. 
 \bibitem{gentzen} G. Gentzen, Untersuchungen \"uber das logische Schlieszen, {\it Mathematische Zeitschrift}, 39(1935)176-210, 405-431, English translation in \cite{gentzencw},  pp. 68-131.
 \bibitem{gentzencw} {\it The Collected Papers of Gerhard Gentzen}, ed. A. Szabo, North-Holland Publishing Company, 1969.
 \bibitem{gielends}
W.~Gielen, H.~de Swart and W.~Veldman, The Continuum
Hypothesis in Intuitionism, {\em The Journal of Symbolic Logic} 46(1981)121-136.
\bibitem{goedel47} K. G\"odel, What is Cantor's continuum problem?, {\it Am. Math. Monthly} 54(1947)515-525, revised version  in \cite{benacerrafputnam}, pp. 258-273, also in \cite{goedelcw}, pp.154-187. 
 \bibitem{goedelcw} K. G\"odel, {\it Collected Works, Volume II, Publications 1938-1974}, ed. S. Feferman, J.W. Dawson Jr., S.C. Kleene, G.H. Moore, R.M. Solovay, J.
  van Heijenoort, Oxford University Press, 1990. 
 \bibitem{graham} R.L. Graham, B.L. Rothschild and  J.H. Spencer, \textit{Ramsey Theory}, John Wiley and Sons, New York, { \it Second  Edition}, 1990.
 \bibitem{griss} G,F.C. Griss, Negationless intuitionistic mathematics I, {\it Proc. Akad. Amsterdam} 49(1946)1127-1133 = {\it Indagationes Mathematicae} 8(1946)675-681.
 \bibitem{hallett} M. Hallett, Towards a Theory of Mathematical Research Programs, {\it The British Journal for the Philosophy of Science}, 30(1979)1-25 and 135-159.
 \bibitem{hardy} G. H. Hardy, {\it A mathematician's apology, with a foreword by C.P.~Snow},  Cambridge University Press, 1967. \bibitem{heyting56} A. Heyting, Intuitionism, an Introduction, \textit{North Holland Publ. Co.} 1956. 
 \bibitem{kant} I. Kant, {\it Prolegomena zu einer jeden k\"unftigen Metaphysik, die als Wissenschaft wird auftreten k\"onnen}, 1783.
 \bibitem{kechris} A.S. Kechris, \textit{Classical Descriptive Set Theory}, Graduate Texts in Mathematics, Vol. 156, Springer-Verlag, New York etc., 1995.
 \bibitem{kleene52} S.C. Kleene, Recursive functions and intuitionistic mathematics, \textit{Proceedings of the International Congress of mathematicians,(Cambridge, Mass., U.S.A., Aug. 30 - Sept. 6, 1950)}, 1952, vol. I, pp. 679-685.
\bibitem{kleene52a} S.C. Kleene, {\it Introduction to Metamathematics},  North-Holland Publ. Co., Amsterdam, P. Noordhoff N.V., Groningen, D. Van Nostrand Comp. Inc., New York and Toronto, 1952. 
\bibitem{kleenevesley65} S.C. Kleene, R.E. Vesley, {\it The Foundations of Intuitionistic mathematics, especiaaly in relation to Recursive Functions}, North-Holland Publ. Co., 1965.
\bibitem{martinloef} P. Martin-L\"of, \textit{Notes on Constructive Mathematics}, Almqvist and Wiksell, Stockholm, 1970. 
\bibitem{Nash-Williams} Crispin St. J. A. Nash-Williams, On well-quasi-ordering finite trees, {\it Proceedings of the Cambridge Philosophical Society} 59(1963)833–835.
 \bibitem{parisharrington} J. Paris, L. Harrington, A Mathematical Incompleteness in Peano Arithmetic, in: \textit{Handbook of Mathematical Logic, ed. J. Barwise, Studies in Logic and the Foundations of Mathematics, vol. 90}, Amsterdam(North Holland Publ. Co.) 1977, pp. 1133-1142.
 \bibitem{proclus} Proclus, {\it A commentary on the first book of Euclid's Elements, translated, with introduction and notes by G.R.~Morrow}, Princeton University Press, 1970.
 \bibitem{ramsey} F.P. Ramsey, On a problem in formal logic, \textit{Proc. London Math. Soc.} 30(1928)264-286. 
 
 \bibitem{veldman1981}W.~Veldman,  {\it Investigations in Intuitionistic Hierarchy
Theory}, Ph.D.~Thesis, Katholieke Universiteit Nijmegen, 1981. 
\bibitem{veldman1982} W. Veldman, \textit{On the continuity of functions in intuitionistic real analysis, some remarks on Brouwer's paper: `\"Uber Definitionsbereiche von Funktionen'}, Report 8210, Mathematisch Instituut, Katholieke Universiteit Nijmegen, 1982.  

\bibitem{veldman82}W.~Veldman, On the constructive contrapositions of two
 axioms of countable choice, in: A.S.~Troelstra, D.~van Dalen(ed.), \textit{The
 L.E.J. Brouwer Centenary Symposium}, North-Holland Publ. Co., Amsterdam, 1982, pp.
 513-523.

\bibitem{veldman1995}
W.~Veldman,  Some intuitionistic variations on the notion of a
finite set of natural numbers, in: H.C.M.~de Swart, L.J.M.~Bergmans
(ed.), {\it Perspectives on Negation,essays in honour of Johan J.~de
Iongh on the occasion of his 80th birthday}, Tilburg University
Press, Tilburg, 1995, pp.~177-202.

\bibitem{veldman1999}
W.~Veldman,  On sets enclosed between a set and its double
complement, in: A. Cantini e.a.(ed.), {\em Logic and Foundations of 
Mathematics}, Proceedings Xth International Congress on Logic,
Methodology and Philosophy of Science, Florence 1995, Volume III, Kluwer
Academic Publishers, Dordrecht, 1999, pp. 143-154.

 \bibitem{veldman2001} W.~Veldman, Understanding and using Brouwer's Continuity
 Principle, in: U.~Berger, H.~Osswald, P.~Schuster (ed.), \textit{Reuniting the
 Antipodes, constructive and nonstandard views of the continuum, Proceedings of a
 Symposium held in San Servolo/Venice, 1999}, Kluwer, Dordrecht, 2001,
 pp. 285-302.
 \bibitem{veldman2001a} W.~Veldman, Bijna de waaierstelling {\it Almost the Fan Theorem}, \textit{Nieuw Archief voor Wiskunde}, vijfde serie, deel 2(2001), pp. 330-339.
 
\bibitem{veldman2004} W.~Veldman, An intuitionistic proof of Kruskal's Theorem, \textit{Arch. Math. Logic} 43(2004)215-264.

\bibitem{veldman2005} W.~Veldman, Two simple sets that are not positively Borel, {\em Annals of Pure and Applied Logic} 135(2005)151-2009.

\bibitem{veldman05} W.~Veldman, Perhaps the Intermediate Value Theorem, \textit{Journal of Universal Computer Science} 11(2005)2142-2158.
\bibitem{veldman2006b} W. ~Veldman, Brouwer's Real Thesis on Bars, in: G.~Heinzmann, G.~Ronzitti, eds., {\em Constructivism: Mathematics, Logic, Philosophy and Linguistics,  Philosophia Scientiae, Cahier Sp\'ecial 6}, 2006, pp. 21-39.



\bibitem{veldman2008}W.~Veldman, Some Applications of Brouwer's Thesis on Bars, in: M.~van Atten, P.~Boldini, M.~Bourdeau, G.~Heinzmann, eds., \textit{One Hundred Years of Intuitionism (1907-2007), The Cerisy Conference}, Birkh\"auser, Basel etc., 2008, pp. 326-340.
\bibitem{veldman2009} W.~Veldman,  The Problem of  Determinacy of Infinite Games from an Intuitionistic Point of View, 
in: O.~Majer, P.-V.~Pietarinen, T.~Tulenheimo, eds.,  \textit{Logic, Games and Philosophy: Foundational Perspectives}, Springer, 2009, pp. 351-370.


\bibitem{veldman2011b}W.~Veldman,  Brouwer's Fan Theorem as an axiom and as a contrast to Kleene's Alternative,   \textit{Archive for Mathematical Logic} 53(2014)621-693.

\bibitem{veldman2011d} W. ~Veldman, The Fan Theorem, its strong negation and the determinacy of games, arXiv: 1311.6988.                                                                                                         
\bibitem{veldman2011c}
W. ~Veldman, The Principle of Open Induction on Cantor space  and the Approximate-Fan Theorem,  August 2014, arXiv 1408.2493.

\bibitem{veldman18} W. Veldman,  Retracing Cantor’s first steps in Brouwer’s company, \textit{Indagationes \\Mathematicae},  29(2018)161-201. 
\bibitem{veldman20} W.~Veldman, Intuitionism is all bosh, entirely, unless it is an inspiration, in: G. Alberts, L. Bergmans, and F. A. Muller (eds.),
	{\it  Significs  and the Vienna Circle: Intersections}, 
	Springer Verlag, Dordrecht, 2020, {\it to appear}.
	
	\bibitem{veldman20b} W. Veldman, Treading in Brouwer's footsteps, in: A. Rezu\c{s} (ed.), {\it Contemporary Logic and Computing}, [Series: {\it Landscapes in Logic}, Volume 1], College Publications, London, 2020, pp. 355-396.
	
	\bibitem{veldmanbezem}W. Veldman and M. Bezem, 
	 Ramsey's Theorem and the Pigeonhole Principle in intuitionistic mathematics, 
 	{\it Journal of the London Mathematical Society}  47(1993)193--211.
\end{thebibliography}
\end{document}